\documentclass[11pt]{article}
\usepackage[utf8]{inputenc}
\usepackage{amssymb,amsbsy,amsmath,amscd,amsthm,amsfonts,MnSymbol}
\usepackage{cite}
\usepackage{natbib}
\usepackage[colorlinks]{hyperref}
\usepackage{graphicx}
\usepackage{verbatim}
\graphicspath{{Images/}}
\usepackage{float}
\usepackage{mathtools}

\usepackage{geometry}
\usepackage{ulem} 
\usepackage{bbm}
\usepackage[shortlabels]{enumitem}

\newtheorem{definition}{Definition}
\newtheorem{theorem}{Theorem}
\newtheorem{example}{Example}
\newtheorem{lemma}{Lemma}
\newtheorem{proposition}{Proposition}

\newtheorem{corollary}{Corollary}

\newtheorem{condition}{Condition}
\newtheorem{remark}{Remark}

\newcommand{\R}{\mathbb R}
\newcommand{\E}{\mathbb E}

\newcommand{\Loss}{\mathcal{L}}
\newcommand{\real}{\mathbb{R}}
\newcommand{\mprob}{\mathbb{P}}

\newcommand{\dist}{\mathrm{d_H}}

\newcommand{\GS}{\mathrm{GS}}

\newcommand{\DS}{\displaystyle}

\newcommand{\thetahat}{\hat{\theta}}
\newcommand{\thetatil}{\widetilde{\theta}}

\title{Differentially private inference via noisy optimization}

\author{Marco Avella Medina\thanks{Department of Statistics, Columbia University} 
\and 
Casey Bradshaw\footnotemark[1]
\and Po-Ling Loh\thanks{Statistical Laboratory, University of Cambridge}}

\begin{document}

\maketitle

\begin{abstract}
We propose a general optimization-based framework for computing differentially private M-estimators and  a new method for constructing differentially private confidence regions. Firstly, we show that robust statistics can be used in conjunction with noisy gradient descent or noisy Newton methods in order to obtain optimal private estimators with global linear or quadratic convergence, respectively. We establish local and global convergence guarantees, under both local strong convexity and self-concordance, showing that our private estimators converge with high probability to a small neighborhood of the non-private M-estimators. Secondly, we tackle the problem of parametric inference by constructing differentially private estimators of the asymptotic variance of our private M-estimators. This naturally leads to approximate pivotal statistics for constructing confidence regions and conducting hypothesis testing. We demonstrate the effectiveness of a bias correction that leads to enhanced small-sample empirical performance in simulations. We illustrate the benefits of our methods in several numerical examples.
\end{abstract}

\section{Introduction}

Over the last decade, differential privacy has evolved from a rigorous paradigm derived by theoretical computer scientists for releasing sensitive data to a technology deployed at scale in numerous applications \citep{erlingssonetal2014, dingetal2017, tangetal2017, garfinkeletal2019}. The setting assumes the existence of a trusted curator who holds the data of individuals in a database, and the goal of privacy is to simultaneously protect individual data while allowing statistical analysis of the aggregate database. Such protection is guaranteed by differential privacy in the context of a remote access query system, where a statistician can only indirectly access the data, e.g., by obtaining noisy summary statistics or outputs of a model. Injecting noise before releasing information to the statistician is essential for preserving privacy, and the noise should be as small as possible in order to optimize statistical performance of the released statistics.
 
In this paper, we consider estimation and inference for M-estimators. Inspired by the work of  \cite{songetal2013, bassilyetal2014, leeandkifer2018, feldmanetal2020}, we propose noisy optimization procedures  that output differentially private counterparts of standard M-estimators. The central idea of these methods is to add noise to every iterate of a gradient-based optimization routine in a way that causes each iterate to satisfy a targeted differential privacy guarantee.  Even though this idea is now fairly common in the literature, our proposed methodology is novel in the following respects:
 \begin{enumerate}[(a)]
     \item Noisy gradient descent: While various versions of such algorithms have appeared in the literature, our first contribution is to provide a \emph{complete, global, finite-sample convergence analysis} under local strong convexity. We demonstrate that the resulting algorithm converges linearly to a near-optimal neighborhood of the target population parameter, which in turn shows that the resulting estimators are nearly minimax optimal and asymptotically normal, as their non-private counterparts. Notably, in the noisy optimization framework we consider, the rate of convergence is extremely important in characterizing the statistical efficiency loss incurred by the differentially private optimizers. We also point out some flaws in common implementations relying on clipped gradients and and show that these problems disappear when one considers appropriate M-estimators from robust statistics.
     \item Noisy Newton: A second main contribution of our work is to introduce a differentially private counterpart of Newton's method. Each step of this algorithm adds noise to both the gradient and Hessian of the loss function at the current iterate, as both quantities may cause privacy leakage. Our theory shows that under either local strong convexity or self-concordance, the noisy Newton algorithm converges quadratically to an optimal neighborhood of the the target population parameter. Similar to the classical convergence analysis of Newton algorithms, the convergence analysis of our algorithm has two phases---a phase consisting of ``damped Newton" steps guaranteeing improved objective values with step size $\eta<1$ when the iterates are far from the solution, and a ``pure Newton" phase with step size $\eta=1$ when the iterates are close enough to the solution. In contrast to standard non-private algorithms, we do not rely on backtracking, as this step is not easily implemented under privacy constraints. Instead, we take damped noisy Newton steps with a fixed step size until a verifiable condition is met indicating that the algorithm can safely enter the pure Newton stage. As evidenced by our simulations, our noisy Newton algorithm can lead to significant improvements over noisy gradient descent when the algorithms are initialized sufficiently far from the global optimum. A recent improvement of noisy Newton has been  proposed by \cite{ganesh2023faster} after an early version of this paper was posted to arXiv.
     
     \item Noisy confidence regions: Our final contribution is to introduce a new approach for constructing confidence regions based on a noisy sandwich formula for M-estimators. The idea is fairly simple---since our noisy optimizers are asymptotically normal with a known formula for the asymptotic variance, we employ differentially private estimators of the asymptotic variance. We rely on matrix-valued privacy-inducing noise in order to make the two main components of the sandwich formula differentially private. We further propose a bias correction method that significantly improves the coverage probability of our confidence regions for small sample sizes. In short, we provide a general technique for constructing asymptotically valid confidence regions, which provides substantial numerical improvements over the few existing alternatives in regression settings \citep{sheffet2017, barrientosetal2019, wangetal2019, avella2021}. 
     
 \end{enumerate}

 \subsection{Related literature}

A sizable body of work is devoted to developing differentially private approaches for convex optimization in the context of empirical risk minimization. A first general optimization construction explored in the literature consists  of perturbing the objective function so as to ensure that the resulting minimizer is differentially private. Representative work exploring this idea includes \cite{chaudhuriandmonteleoni2008}, \cite{chaudhurietal2011}, \cite{kiferetal2012},
\cite{jainandthakurta2014} and more recently \cite{slavkovicandmolinari2021}.  

A related idea that is closer to the methods studied in our work considers iterative gradient-based algorithms, where the algorithm employs noisy gradients designed to ensure differential privacy at each iteration. Many such empirical risk minimization algorithms have been proposed \citep{songetal2013, bassilyetal2014, wangetal2017, leeandkifer2018, iyengaretal2019, bassilyetal2019, balleetal2020, wang2018regression}, motivated by online problems \citep{jainetal2012, duchietal2018}, multiparty classification \citep{rajkumarandargawal2012},
Bayesian learning
\citep{wangetal2015}, high-dimensional regression \citep{talwaretal2015,caietal2019,caietal2020}, and  deep learning \citep{abadietal2016, buetal2020}.  Typical theoretical guarantees for such noisy gradient-based methods are stated in terms of the expected sub-optimality gap, i.e.,  the distance between the objective function at the private randomized solution and the non-private optimal solution. As argued in Remark \ref{rem:opt_gap}, such statements imply weaker estimation error guarantees than the main results of our work; we need the stronger guarantees to construct valid confidence intervals.

While most of the literature has focused on noisy stochastic gradient descent, we utilize full gradient evaluations, similar to standard implementations in non-private statistical software. Many existing results cannot be directly applied to our setting \citep{bassilyetal2014,feldmanetal2020}; moreover, even existing methods that use full gradients rely on different analyses \citep{caietal2019, caietal2020}, as we avoid truncation, allow unbounded input variables and parameter spaces, and explicitly track the impact of potentially bad starting values.  We can achieve all of the above by considering locally strongly convex or self-concordant objective functions defining Fisher-consistent bounded-influence M-estimators. We believe that all of these points are important, as they make our methods work well under standard assumptions for non-private settings.

On the optimization side, our work is related to inexact oracle methods~\citep{sun2020composite,devolder2014first} and stochastic optimization~\citep{d2008smooth,ghadimi2012optimal,wang2017stochastic}. Indeed, our proofs for the convergence of noisy gradient descent and noisy Newton's method rely on showing that with high probability, the noise introduced to the gradients and Hessians has a negligible effect on the convergence of the iterates (up to the order of the statistical error of the non-noisy versions of the algorithms). The theory of inexact oracle methods similarly derives results for the output of iterative optimization algorithms when gradients and/or Hessians are computed up to a certain level of accuracy. However, whereas inexact oracles rely on approximate gradients, approximate gradients alone do \emph{not} constitute inexact oracles unless the domain is bounded~\citep{devolder2014first}, so our results are not directly implied by existing literature in this area. The work on inexact second-order methods is sparser: \cite{sun2020composite} studied global and local convergence of an inexact oracle version of Newton's method, but only covers standard self-concordant functions (denoted $(\gamma, 3)$-self-concordant in our paper), whereas our main focus is on pseudo-self-concordant (i.e., $(\gamma, 2)$-self-concordant) functions, which are appropriate for our $M$-estimation framework. The stochastic optimization literature is more directly applicable to the noisy gradient setting in our paper, and our privatized gradients can be viewed as a special instantiation of stochastic gradients, where noise is introduced not due to sampling error, but in order to preserve privacy. However, our results are also not direct consequences of existing literature on first-order~\citep{ghadimi2012optimal} or second-order~\citep{wang2017stochastic} methods, since the objective functions we consider are at most locally strongly convex, and our noisy Newton algorithms employ a particular version of a noisy Hessian that is motivated by differential privacy. We further note that while our global convergence analysis of the noisy Newton method for self-concordant functions builds upon the work of \cite{karimireddyetal2018}, we show that the Hessian stability condition imposed in that paper, along with an additional uniform bound on the Hessian, are enough to ensure a much stronger guarantee of local quadratic convergence after an initial epoch of linear convergence.

The literature on private inference is relatively limited, and has only recently received some attention from computer science \citep{uhleretal2013, gaboardietal2016, rogersandkifer2017,karwaandvadhan2017,awanandslavkovic2018, acharyaetal2018,wangetal2019,covingtonetal2021,chadhaetal2021}. This problem has also been studied in the statistics literature in a regression setting, but in many cases, suffers from the same drawbacks encountered in estimation, e.g., the need to assume bounded data \citep{yuetal2014,  wangetal2019}, resorting to truncation  \citep{  barrientosetal2019}, and requiring very large sample sizes for the methods to work well \citep{sheffet2017}, especially in light of the expected $\sqrt{n}$-consistency of M-estimators \citep{barberandduchi2014, caietal2019, avella2021}. Recent Bayesian inference methods include \citep{savitskyetal2019,kulkarnietal2021, penaandbarrientos2021}. We note that while \cite{avella2021} provides a simple method for differentially private inference, the method requires a sufficiently large sample size that depends on problem parameters that are difficult to quantify in practice. Our method guarantees differential privacy for all sample sizes and gives theoretically near-optimal convergence results as well as good small-sample performance in simulated examples. From a methodological perspective, we avoid imposing very strong conditions such as bounded data and/or bounded parameter spaces; practically, our contributions provide methods which are easier to implement and lead to better inference than existing alternatives.

\subsection{Outline}
 
Section~\ref{SecBackground} introduces basic concepts of differential privacy and M-estimators. Section~\ref{SecNGD} presents our noisy gradient descent algorithm, with a complete global convergence analysis under local strong convexity. Section~\ref{SecNewton} introduces our noisy Newton algorithm and provides two parallel global convergence theories, one relying on local strong convexity and the other on self-concordance. We also provide several examples and numerical illustrations. In Section~\ref{SecInference}, we present our new approach for constructing confidence regions, including a small-sample bias correction.
Section~\ref{SecDiscussion} concludes the paper with a discussion of our results and  future research directions. All the proofs of our results are relegated to the Appendix, together with real-data applications and empirical comparisons to existing methods.

\section{Background}
\label{SecBackground}

We first present notation and fundamentals of privacy and $M$-estimation.

\subsection{Notation}

For a vector $v\in\mathbb{R}^m$, we write $\|v\|_2^2=v^\top v$. For a positive definite matrix $A$, we denote $\|v\|_A^2 = v^\top A v$ and $\|A\|_2=\max_{v:\|v\|_2=1}\|Av\|_2$. We denote the smallest and largest eigenvalues of a symmetric matrix $A$ by $\lambda_{\min}(A)$ and $\lambda_{\max}(A)$, respectively.
We write $f=O(g)$ when $f(\cdot)\leq C g(\cdot)$ for any admissible arguments of $f(\cdot)$ and $g(\cdot)$ and some positive constant $C$. Similarly, we write $f=\Omega(g)$ when $f(\cdot)\geq c g(\cdot)$ for any admissible arguments of $f(\cdot)$ and $g(\cdot)$ and some positive constant $c$. We write $f\asymp g$ when both $f=O(g)$ and $f=\Omega(g)$. For sequences of random variables $\{X_n\}$ and $\{Y_n\}$, we write $X_n = O_p(Y_n)$ to denote boundedness in probability, i.e., for every $\epsilon > 0$, there exist $M$ and $N$ such that $\mprob\left(\left|\frac{X_n}{Y_n}\right| < M\right) > 1-\epsilon$ for all $n \ge N$.
We write $x_{1:n} \in \real^{n \times m}$ to denote the data matrix with $i^{\text{th}}$ row equal to $x_i \in \real^m$. For any two matrices $x_{1:n},x_{1:n}'\in\R^{n\times m}$, we define their Hamming distance $\DS \dist(x_{1:n},x_{1:n}'):=\left|\{i=1,\ldots,n:x_i\neq x_i'\}\right|$ to be the number of coordinates which differ between $x_{1:n}$ and $x_{1:n}'$.
For $\theta \in \real^p$ and $r > 0$, we write $\mathcal{B}_r(\theta)$ to denote the Euclidean ball of radius $r$ around $\theta$. For a set (event) $E$, let $\mathbbm{1}_E$ or $\mathbbm{1} \{E\}$ denote the indicator function.


\subsection{Gaussian differential privacy}

A \textit{random function} $ h:\R^{n\times m}\to\R^p$ maps each $x_{1:n}\in\R^{n\times m}$ to a Borel-measurable random variable $h(x_{1:n})$.  A statistic is a deterministic function of an observed data set $x_{1:n}$, e.g., the sample mean $\bar{x}$, whereas $h(x)$ is a randomized estimate, i.e., the output of a randomized algorithm obtained by perturbing the deterministic output $\bar{x}$.  

%
\begin{definition}
[$(\epsilon, \delta)$-differential privacy]
\label{DefDP}
Let $\varepsilon,\delta>0$. A random function $ h$ is $(\varepsilon,\delta)$-differentially private (DP) if and only if for every pair of data sets $x_{1:n},x_{1:n}'\in\mathbb{R}^{n\times m}$ such that $\dist(x_{1:n},x_{1:n}') = 1$ and all Borel sets $B\subseteq \mathbb{R}^p$, we have
$\mathbb{P}[ h(x_{1:n})\in B]\leq e^\varepsilon \mathbb{P}[ h(x_{1:n}')\in B]+\delta$.
\end{definition}

The probabilities in Definition~\ref{DefDP} are computed over the randomness of the function $h$, for any fixed neighboring data sets $x_{1:n}$ and $x_{1:n}'$.  The most basic approach for constructing such a randomized function consists of adding random noise to a (deterministic) statistic, as we will discuss below. 
The definition of $(\varepsilon,\delta)$-DP limits the ability of an adversary to identify the presence of an individual in a data set based on released outputs, as it is difficult to distinguish between the  distributions of $h(x_{1:n})$ and $h(x_{1:n}')$. This leads to a hypothesis testing interpretation of differential privacy pointed out by \cite{wassermanandzhou2010}. Indeed, one can interpret differential privacy as a protection guarantee against an adversary that tests two simple hypotheses of the form $H_0: x_i=s$ vs.\ $H_1:x_i=t$. Privacy is ensured when this testing problem is difficult, and $(\varepsilon,\delta)$-DP assesses the hardness of the problem via an approximate worst-case likelihood ratio of the distributions of $h(x_{1:n})$ and $h(x_{1:n}')$ over all neighboring data sets. 

Building on the hypothesis testing interpretation, Dong et al.~\cite{dongetal2021} advocated for a new definition of Gaussian differential privacy that we will use in our paper. The definition involves a transparent interpretation of the privacy requirement: determining whether any individual is in a data set is at least as hard as distinguishing between two normal distributions $N(0,1)$ and $N(\mu,1)$ based on one random draw.  The formal definition is a little more complex:

\begin{definition}
[Gaussian differential privacy]
\label{def:GDP}
Let $h:\R^{n\times m}\to\R^p$ be a random function.
\begin{enumerate}
    \item We say that $h$ is $f$-DP if any $\alpha$-level test between simple hypotheses of the form  $H_0: x_i=t\,\mbox{ vs. }\, H_1:x_i=s$ has power function 
 $\beta(\alpha)\leq 1-f(\alpha)$, where $f$ is a convex, continuous, non-increasing function satisfying $f(\alpha)\leq 1-\alpha$ for all $\alpha\in[0,1]$.
    \item We say that $h$ is $\mu$-Gaussian differentially private (GDP) if $h$ is $f$-DP and
$f(\alpha)\geq \Phi(\Phi^{-1}(1-\alpha)-\mu)$ for all $\alpha\in[0,1]$, where $\Phi(\cdot)$ is the standard normal cdf.
\end{enumerate}
\end{definition}

The following notion of sensitivity will be central in our construction of differentially private procedures. In particular, it is used in the most basic algorithms that make some output $h(x)$ private by simply releasing $h(x)+u$, where $u$ is an independent noise term whose variance is scaled according to the sensitivity of $h$.

\begin{definition}
[Global sensitivity]
Let $g:\R^{n\times m}\to\R^p$ be deterministic. 
The global sensitivity of $g$ is defined by $\GS_g=\sup_{x_{1:n},x_{1:n}'\in\R^{n\times m}}\{\|g(x'_{1:n})-g(x_{1:n})\|_2:  \dist(x_{1:n},x_{1:n}') = 1\}$.
 \end{definition}

The following theorem concerns a procedure that will be a primary building block for our private algorithm. This method is called the \emph{Gaussian mechanism}, and can be easily tuned to achieve a desired $(\varepsilon,\delta)$-DP \cite[Theorem A.1]{dworkandroth2014} or $\mu$-GDP guarantee, as stated below:

\begin{theorem}
[Theorem 1 in \cite{dongetal2021}]
\label{thm:DP-GS}
	Let $g:\R^{n\times m}\to\R^p$ be a function with finite global sensitivity $\GS_{g}$.
	Let $Z$ be a standard normal $p$-dimensional random vector. For all $\mu>0$ and $x\in\R^{n\times m}$, the random function $ h(x)=g(x)+\frac{\GS_{g}}{\mu}Z $ is $\mu$-GDP.
\end{theorem}

Our proposed optimization methods will rely on compositions of a sequence of differentially private outputs computed using the same data set, where each step uses information from prior private computations. A  question of paramount importance is to characterize the overall privacy guarantee of such an analysis. Intuitively, this guarantee should degrade as one composes more and more private outputs. \cite{dongetal2021} very elegantly argued that in such scenarios, Gaussian differential privacy is very special: Its privacy guarantee under composition can be characterized as the $K$-fold composition of $\mu_k$-GDP mechanisms, and is exactly $\mu$-GDP, where $\mu=\sqrt{\mu_1^2+\dots+\mu_K^2}$. More fundamentally,  GDP is  a canonical privacy guarantee in an asymptotic sense, as a central limit theorem phenomenon  shows that the composition of a large number of $f$-DP algorithms is approximately $\mu$-GDP for some parameter $\mu$ \cite{dongetal2021}.

\subsection{M-estimators for parametric models} 

M-estimators are a simple class of estimators, stemming from robust statistics and constituting a general approach to parametric inference \citep{huber1964,huberandronchetti2009}. They will be the main tool used in our private approach to estimation and inference. 
Suppose we observe an i.i.d.\ sample $x_1,\dots,x_n$ with cdf $F$ and wish to estimate a population parameter $\theta_0=T(F)$ lying in a parameter space $\Theta$. Our construction of differentially private estimators relies on noisy optimization  techniques that will lead to private counterparts of  M-estimators $\hat\theta=T(F_n)$ of $\theta_0$, defined as minimizers of the form
\begin{equation}
    \label{eq:M-estimator-rho}
    \hat{\theta}=\underset{\theta\in\Theta}{\mbox{argmin}}\,\mathcal{L}_n(\theta)=\underset{\theta\in\Theta}{\mbox{argmin}}\,\frac{1}{n}\sum_{i=1}^n\rho(x_i,\theta)=\underset{\theta\in\Theta}{\mbox{argmin}}\,\mathbb{E}_{F_n}[\rho(X,\theta)],
\end{equation}
where $F_n$ denotes the empirical distribution of $x_1,\dots,x_n\in\mathbb{R}^m$. This class of estimators is a generalization of the class of maximum likelihood estimators.

When $\rho$ is differentiable and convex, the estimator $\hat\theta$
can be viewed as a solution to $\frac{1}{n}\sum_{i=1}^n\Psi(x_i,\hat\theta)=0$,
where $\Psi(x,\theta)=\frac{\partial}{\partial\theta}\rho(x,\theta)$.
M-estimators defined by $\Psi(x,\theta)$ which is bounded in $x\in\mathcal X\subseteq\mathbb{R}^m$ are particularly appealing in robust statistics, since the influence function of $T(F)$ is then bounded, ensuring infinitesimal robustness to outliers \citep{hampeletal1986}.

Under mild general conditions, e.g., $\theta_0 = \arg\min_\theta \mathbb{E}_{F}[\rho(X,\theta)]$ and $\mathbb{E}_{F}[\Psi(X,\theta_{0})]=0$ \citep{huber1967},
M-estimators are asymptotically normal: $\sqrt{n} (\hat\theta - \theta_0) \to_d
{\cal N} (0 , V (\Psi , F))$,
where
\begin{equation}
\label{EqnSandwich}
\begin{split}
& V (\Psi , F) :=  M (\Psi , F)^{-1} Q (\Psi, F) M (\Psi,F)^{-1}, \\
M (\Psi , F) & := -\E_F[ \dot\Psi (Z , \theta_0) ], \qquad Q (\Psi , F) := \E_F[ \Psi (Z , \theta_0)  \Psi (Z , \theta_0)^\top],
\end{split}
\end{equation}
and $\dot\Psi(x,\theta) := \frac{\partial}{\partial\theta^\top}\Psi(x,\theta)$. We will establish analogous results for our noisy estimators.
 
 
\section{Randomized M-estimators via noisy gradient descent}
\label{SecNGD}

The theory of empirical risk minimization suggests various algorithms to compute the global optimum~\eqref{eq:M-estimator-rho} over a convex set $\Theta$ \citep{boydandvanderberghe2004}. One of the most elementary methods is gradient descent, which uses the iterates 
\begin{equation}
    \label{eq:GD}
    \theta^{(k+1)}=\theta^{(k)}-\frac{\eta}{n}\sum_{i=1}^n\Psi(x_i,\theta^{(k)}).
\end{equation}
In a classical statistical setting, one would consider the iterates~\eqref{eq:GD} until a numerical condition is met, e.g.,  $\|\frac{1}{n}\sum_{i=1}^n\Psi(x_i,\theta^{(k)})\|_2<\epsilon$, for a prespecified tolerance level $\epsilon$. Since the optimization error can be made negligible by setting $\epsilon$ arbitrarily small, classical statistical theory usually ignores the effect of carrying out inference on the final iterate $\theta^{(K)}$ as opposed to the global optimum $\hat\theta$, which is typically analyzed theoretically. Indeed, one can in principle take $K$ as large as needed in order to ensure that $\theta^{(K)}$ is numerically identical to $\hat\theta$. This is in stark contrast to our proposed differentially private version of gradient descent, since $K$ needs to be set \emph{before} running the algorithm in order to ensure that the final estimator respects a desired level of differential privacy. Intuitively, the larger the number of data (gradient) queries of the algorithm, the more prone it will be to privacy leakage. Our algorithm restricts the class of M-estimators to satisfy the following uniform boundedness condition.
\begin{condition}
\label{ass:bounded_gradient}
The gradient of the loss satisfies $\sup_{x\in\mathcal X,\theta\in\Theta}\| \Psi(x,\theta)\|_2 \leq B < \infty$.
\end{condition}
Note that the choice of the loss function $\rho$ implies a known bound on the gradient $B$, which guarantees  that $2B$ is an upper bound for the global sensitivity of $ \Psi(x,\theta)$. Our private version of gradient descent considers the following noisy version of the iterates~\eqref{eq:GD}:
\begin{equation}
    \label{eq:NGD}
    \theta^{(k+1)}=\theta^{(k)}-\frac{\eta}{n}\sum_{i=1}^n\Psi(x_i,\theta^{(k)})+\frac{2\eta B\sqrt{K}}{\mu n}Z_{k},
\end{equation}
where the final estimate is again denoted by $\theta^{(K)}$. Here, $\eta$ is the step size, and $\{Z_k\}_{k=0}^{K-1}$ is a sequence of i.i.d.\ standard $p$-dimensional Gaussian random vectors.  The number of iterations $K$ needs to be set beforehand and critically impacts the statistical performance of this estimator, as discussed below. We choose the initial iterate $\theta^{(0)}$ in a non-data-dependent manner (e.g., $\theta^{(0)} = 0$), with specific choices of $\theta^{(0)}$ to be highlighted in the simulations to follow. We note that many related variants of this noisy gradient descent procedure exist in the literature  \citep{songetal2013, bassilyetal2014, leeandkifer2018, feldmanetal2020}; however, our analysis provides novel insights into the properties of this algorithm, as our convergence analysis relies on local strong convexity and provides a general consistency and asymptotic normality theory for differentially private M-estimators.


\subsection{Convergence analysis}

Taking $B=0$ in the iterates~\eqref{eq:NGD} recovers standard gradient descent. Classical optimization theory tells us that for the choice $K = O(\log(1/\Delta))$, gradient descent incurs an optimization error of $\|\theta^{(K)}-\hat\theta\|_2=O(\Delta)$. Therefore, $\log(n)$ steps suffice to ensure that the optimization error matches the parametric convergence rate $O\left(\sqrt{\frac{p}{n}}\right)$.  We will require local strong convexity (LSC) of the loss $\Loss_n$ in a ball of radius $r$ around the true parameter $\theta_0$, as well as global smoothness. Strong convexity and smoothness are both standard conditions for the convergence analysis of gradient-based optimization methods~\citep{boydandvanderberghe2004}. Useful properties of strongly convex and smooth functions are reviewed in Appendix~\ref{AppConvex}.
\begin{condition}
\label{ass:RSC/RSS}
 The loss function $\mathcal{L}_n$ is locally $\tau_1$-strongly convex and $\tau_2$-smooth, i.e.,
\begin{align*}
\mathcal{L}_n(\theta_1)-\mathcal{L}_n(\theta_2) & \geq \langle \nabla\mathcal{L}_n(\theta_2) ,\theta_1-\theta_2 \rangle +\tau_1\|\theta_1-\theta_2\|_2^2, \quad \forall \theta_1, \theta_2\in \mathcal{B}_r(\theta_0), \\
\mathcal{L}_n(\theta_1)-\mathcal{L}_n(\theta_2) & \leq \langle \nabla \mathcal{L}_n(\theta_2) ,\theta_1-\theta_2 \rangle +\tau_2\|\theta_1-\theta_2\|_2^2, \quad \forall \theta_1, \theta_2\in \Theta\subseteq \mathbb R^p.
\end{align*}
\end{condition}
The following theorem shows that the iterates~\eqref{eq:NGD} are $\mu$-GDP and converge to the non-private M-estimator as $n$ increases. Specifically, it shows that the private iterates lie in a neighborhood of the target non-private M-estimator whose radius is comparable to the privacy-inducing noise added in each noisy gradient descent step. The proof is in Appendix~\ref{AppThmNGD}, and hinges on the fact that the assumptions on the loss function and $\theta^{(0)}$ ensure that all iterates lie in the LSC ball $\mathcal{B}_r(\theta_0)$, with high probability (cf.\ Lemmas~\ref{lem:LSCball} and~\ref{lemma1_alternative}.) 
\begin{theorem}
\label{thm:NGD}
Assume $\Loss_n$ satisfies Conditions \ref{ass:bounded_gradient} and \ref{ass:RSC/RSS}, and suppose $\mathcal{L}_n(\theta^{(0)})-\mathcal{L}_n(\thetahat)\leq \frac{r^2}{4} \tau_1$ and $\hat\theta\in\mathcal{B}_{r/2}(\theta_0)$. Suppose $\Loss_n$ is twice-differentiable almost everywhere in $\mathcal{B}_r(\theta_0)$. Further let $ \eta\leq \frac{1}{2}\min\left\{ \frac{1}{\tau_2}, 1\right\}$,  $n = \Omega\left(\frac{\sqrt{K \log(K/\xi)}}{\mu}\right)$, and $K =\Omega (\log n)$. Then (i) $\theta^{(K)}$ is $\mu$-GDP, and (ii) $\|\theta^{(K)}-\hat\theta\|_2\leq   C\frac{\sqrt{K}(\sqrt{p}+\sqrt{2\log(K/\xi)})}{\mu n} $, with probability at least $1-\xi$, where $C$ is a constant depending on $ B$, $\tau_1$, $\tau_2$, and $\eta$. 
\end{theorem}

\begin{remark}
Here and in the sequel, we treat $\tau_1$, $\tau_2$, and $r$ as constants, and only track the dependence of the error bounds and sample size requirements on $n$ and $K$. Note that Condition~\ref{ass:RSC/RSS} is a statement about the empirical loss function $\Loss_n$, so the parameters $(\tau_1, \tau_2, r)$ could in principle also be functions of $n$. However, in the $M$-estimation settings we consider, it can be shown that Condition~\ref{ass:RSC/RSS} holds with high probability for fixed values of $(\tau_1, \tau_2, r)$ when $n$ is sufficiently large, i.e., under the prescribed minimum sample size conditions appearing in the statements of the theorems. (For additional details, see the discussion in Section~\ref{SecExamples} below.)
\end{remark}

\begin{remark}\label{rem:opt_gap}
It is standard in differential privacy to present convergence results in terms of the expected sub-optimality gap \citep{bassilyetal2014, bassilyetal2019,feldmanetal2020,songetal2021}, building on stochastic gradient descent analysis. Assuming strong convexity, the best results among them give, in our notation, 
\begin{equation}
    \label{eq:exp_gap}
    \mathbb{E}[\mathcal{L}_n(\theta^{(K)})]-\mathcal{L}_n(\hat\theta)\leq C\frac{p}{\mu n}.
\end{equation} 
Instead, Theorem \ref{thm:NGD} proves a sub-Gaussian deviation error of the last iterate by adapting \emph{gradient descent} analysis. This distinction is important for the following reasons:
\begin{enumerate}
\item Even assuming strong convexity, stochastic gradient descent cannot guarantee geometrically decaying sub-optimality gaps \cite[Theorem 2]{agarwaletal2012}. One can expect the slower convergence rate to carry over to noisy stochastic gradient descent. A central idea in our proof of Theorem \ref{thm:NGD} is to show that the  geometric convergence of gradient descent is preserved for noisy iterates. Our analysis also bypasses the need to state convergence results in expectation. 

\item There is no obvious way to convert sub-optimality gaps stated in expectation to parameter error bounds on $\|\theta^{(K)}-\hat\theta\|$. Assuming the bound~\eqref{eq:exp_gap} and $\alpha$-strong convexity, Jensen's inequality implies that $\|\mathbb{E}[\theta^{(K)}]-\hat\theta\|_2^2\leq C\frac{p}{\alpha\mu n}$,
which is unsatisfactory because it bounds the error of the \emph{expectation} of the last iterate and gives the wrong rate of convergence. Alternatively, one could obtain a finite-sample deviation bound from inequality~\eqref{eq:exp_gap} via Markov's inequality, but it would result in a much weaker version of Theorem \ref{thm:NGD}. A tight characterization of the convergence rate of $\theta^{(K)}$ is necessary for accurate inference. 
\item Results regarding the convergence of the expected sub-optimality gap can be misleading in the sense that they do not point out the correct statistical cost of privacy induced by larger noise due to a large number of iterations. As a result, it is common to see theorems where the number of iterations is taken to be $K=C n$ or $K=Cn^2$ \citep{bassilyetal2014, bassilyetal2019,feldmanetal2020,songetal2021}. It is easy to see from our proof of Lemma \ref{lemma2} that taking expectations in the bound~\eqref{lem2.2} leads to 
    $$\mathbb{E}[\mathcal{L}_n(\theta^{(K)})]-\mathcal{L}_n(\hat\theta)\leq \kappa^K[\mathcal{L}_n(\theta^{(0)})-\mathcal{L}_n(\hat\theta)]+\frac{4\eta^2B^2 K p}{\mu^2n^2},$$
    where $\kappa\in(0,1)$. This bound is misleading because it suggests one could take $K=n$ and have a negligible error compared to the statistical error $O\left(\sqrt{\frac{p}{n}}\right)$. This is because by taking the expectation, one ignores the linear term in $N_k$, which does not affect the bias but is the dominating variance term. Our high-probability analysis, on the other hand, keeps track of the increasing effect of the total number of iterations on the variance. This new analysis points out explicitly the importance of choosing a small number of iterations. 
   \end{enumerate}

\end{remark}

Theorem \ref{thm:NGD} immediately allows us to derive statistical properties of the noisy gradient descent estimator. The following result is a consequence of the triangle inequality and standard M-estimation theory, e.g., \citep[Corollary 6.7]{huberandronchetti2009}, which guarantees that $\|\thetahat - \theta_0\|_2 = O_p\left(\sqrt{\frac{p}{n}}\right)$.

\begin{corollary}
\label{CorNGD}
Assume the conditions of Theorem \ref{thm:NGD}, and suppose $\theta_0$ is such that $\mathbb{E}[\Psi(Z,\theta_0)]=0$ and $\mathbb{E}[\dot\Psi(x,\theta)]$ is continuous and invertible in a neighborhood of $\theta_0$. Then
\begin{equation*}
\textrm{(i) } \theta^{(K)}-\theta_0=\hat\theta-\theta_0+O_p\left(\frac{\sqrt{K \log K}}{\mu n}\right), \quad \text{and} \quad \textrm{(ii) } \sqrt{n}(\theta^{(K)}-\theta_0)\to_d N(0,V (\Psi , F_{\theta_0})).
\end{equation*}
\end{corollary}

\begin{remark}
Corollary~\ref{CorNGD}(i) implies that $\theta^{(K)} =  \theta_0 + O_p\left(\sqrt{\frac{1}{n}} + \frac{\sqrt{K \log K}}{\mu n}\right)$. Thus, we see that the extra $O_p\left(\frac{\sqrt{K \log K}}{\mu n}\right)$ term is the ``cost of privacy" of our noisy gradient descent algorithm. Note that lower bounds of $\Omega_p\left(\frac{1}{\epsilon n}\right)$ on the cost of privacy were derived for the regression setting in \cite{caietal2019} and \cite{caietal2020} under the framework of $(\epsilon, \delta)$-DP. If we were to adopt the $(\epsilon, \delta)$-DP framework in our analysis, the noise introduced at each iterate would be of the form $\Theta\left(\frac{K\sqrt{\log(c/\delta)}}{\epsilon n}\right) \eta Z_k$ rather than $\Theta\left(\frac{\sqrt{K}}{\mu n}\right) \eta Z_k$, so the same optimization-theoretic analysis in Theorem~\ref{thm:NGD} would lead to a cost of privacy of $O_p\left(\frac{\sqrt{\log(c/\delta)}}{\epsilon n}\right)$, matching the known lower bounds up to $\log n$ factors. Indeed, the same lower bounds for $(\epsilon, \delta)$-DP, combined with the fact that an algorithm is $\mu$-GDP if and only if it is $(\mu, \delta(\mu))$-DP, for $\delta(\mu) = \Phi(-1 + \mu/2) - e^\mu \Phi(-1 - \mu/2)$ \citep[Corollary 2.13]{dongetal2021}, can be used to derive a lower bound of $\Omega_p\left(\frac{1}{\mu n}\right)$ on the cost of privacy in the $\mu$-GDP setting for small values of $\mu$, showing that our estimation error in Corollary~\ref{CorNGD} is minimax optimal (up to logarithmic factors). We note also that we are not writing out explicitly the dependence on $p$ (only on $K$ and $n$) in Corollary~\ref{CorNGD}. In fact, the bound $B$ would scale as $\sqrt{p}$ if, e.g., $\Psi$ were coordinatewise bounded, and the constant $C$ appearing in Theorem~\ref{thm:NGD}(ii) depends on the local strong convexity/smoothness constants, which may also have a dependence on $p$ in high dimensions.
\end{remark}

\begin{remark}
A more careful analysis would allow us to remove the twice-differentiability assumption on $\Loss_n$ in Theorem~\ref{thm:NGD} and Corollary~\ref{CorNGD}, since the gradient descent algorithm only requires one derivative, and asymptotic normality of M-estimators does not require the loss function to be twice-differentiable \citep{huber1967}.
\end{remark}


Theorem~\ref{thm:NGD} requires thesub-optimality gap of the objective function between the initial value $\theta^{(0)}$ and the global optimum $\thetahat$ to be bounded by $\frac{r^2}{4} \tau_1$. The following proposition, proved in Appendix~\ref{App:prop:NGD_bad_starting_value}, shows that an initial fixed number of noisy gradient descent iterations can be used to ensure that the starting value condition is met.
  
\begin{proposition}
\label{prop:NGD_bad_starting_value}
Assume Conditions \ref{ass:bounded_gradient} and \ref{ass:RSC/RSS} hold and $\eta\leq \frac{1}{2\tau_2}$. Let $R = \|\theta^{(0)} - \thetahat\|_2$.
Then there exists a constant $c_0 > 0$ such that with probability at least $1-\xi_0$, after $K_0 = \frac{R^2}{\eta \Delta}$ noisy gradient descent iterations~\eqref{eq:NGD} and for  $n\geq C_0\frac{(R + K_0 + 1)\{4\sqrt{p}+2\sqrt{2\log(K_0/\xi_0)}\}\sqrt{K_0}}{\Delta \mu}$, we have
$\mathcal{L}_n(\theta^{(K_0)})-\mathcal{L}_n(\thetahat)\leq \Delta$, where $C_0$ is a constant which depends on $\eta$ and $B$.
\end{proposition}

Taking $\Delta=\frac{r^2}{4}\tau_1$ in Proposition~\ref{prop:NGD_bad_starting_value}, we see that $K_0 = \frac{4}{\tau_1 r^2 \eta}\|\theta^{(0)} - \thetahat\|_2^2$ iterations of noisy gradient descent are sufficient to ensure that $\theta^{(K_0)}$ meets the initialization condition of Theorem \ref{thm:NGD}. We remark that both Theorem \ref{thm:NGD} and Proposition \ref{prop:NGD_bad_starting_value} require the minimum sample size scaling $n = \Omega\left(\frac{\sqrt{K \log K}}{\mu}\right)$. Therefore, the additional $K_0 \asymp \|\theta^{(0)} - \theta_0\|_2^2$ iterations needed to ensure the starting value assumption of Theorem \ref{thm:NGD} does not substantively affect the conclusion of the theorem, since that result already assumes that $K = \Omega(\log n)$.

\begin{remark}
The expression for the minimum sample size in Proposition~\ref{prop:NGD_bad_starting_value} involves $R$, which is defined in terms of $\thetahat$ and is therefore also a function of $n$. However, recall that we are working in a scenario where $\|\thetahat - \theta_0\|_2 = O_p\left(\sqrt{\frac{p}{n}}\right)$, so $R = \|\theta^{(0)} - \theta_0\|_2 + o_p(1)$.
\end{remark}

\subsection{Examples}

We now present two numerical simulations to illustrate the theory thus far. In particular, we demonstrate the behavior of noisy gradient trajectories and the benefits of our proposed approach based on M-estimation, in contrast to gradient clipping. For more comparisons to clipping and other benchmark methods, see Appendix~\ref{AppBenchmark}.


\subsubsection{Linear regression}
\label{SecLinReg}

We now explore the behavior of the noisy gradient descent algorithm on simulated data from a linear regression model. The data $\left \{ \left (x_{i},y_{i} \right ) \right \} _{i=1}^{n}$ are generated according to the model $y_{i}=x_{i}^{\top} \beta +\epsilon_{i}$, where $\epsilon_i \stackrel{i.i.d.}{\sim} N(0,\sigma^2)$ and the covariate vectors are given by $x_{i}=\left (1, z_{i} \right )^{\top}$, where $z_{i} \overset{i.i.d.}{\sim} N(0, \sigma_{z}^2 \mathbb{I}_{3})$.
We take our loss function to be
\begin{equation}
\label{EqnLinearLoss}
    \mathcal{L}_n(\beta,\sigma)=\frac{1}{n}\sum_{i=1}^{n} \left( \sigma \rho_{c} \Big( \frac{y_{i}-x_{i}^{\top} \beta}{\sigma} \Big)+\frac{1}{2} \kappa_{c}  \sigma \right) w(x_{i}),
\end{equation}
where $\rho_{c}$ is the Huber loss with tuning parameter $c$, the constant $\kappa_{c}$ is chosen to ensure consistency of $\hat{\sigma}$, and $w(x_{i})=\textrm{min} \left(1,\frac{2}{\left \|x_{i}  \right \|_{2}^{2}} \right)$ downweights outlying covariates. By construction, the gradient of this loss function with respect to $\theta=(\beta,\sigma)$ has finite global sensitivity:
The global sensitivity of $ \nabla_{\beta}\mathcal{L}_n(\beta,\sigma)$ is $2\sqrt{2} c$ and the global sensitivity of $\nabla_{\sigma}\mathcal{L}_n(\beta,\sigma)$ is  $\frac{1}{2}c^{2}$, resulting in a global sensitivity of $ \sqrt{8c^{2}+\frac{1}{4}c^{4}}$ for $ \nabla_{\theta}\mathcal{L}_n(\theta)$ (this quantity takes the place of $2B$ in equation~\eqref{eq:NGD}). Setting $\beta= \left (1,1,1,1 \right )^{\top}$, $\sigma=2=\sigma_{z}$, $c=1.345$, and $n=1000$, Figure~\ref{fig:NGD Behavior} plots sample trajectories of the noisy gradient descent iterates for the coordinate of the parameter vector corresponding to $\beta_{2}$. The optimization was initialized at $\beta^{(0)} = 0$ and $\sigma^{(0)} = 1$.

\begin{figure}[h]
    \centering
    \includegraphics[width=14cm]{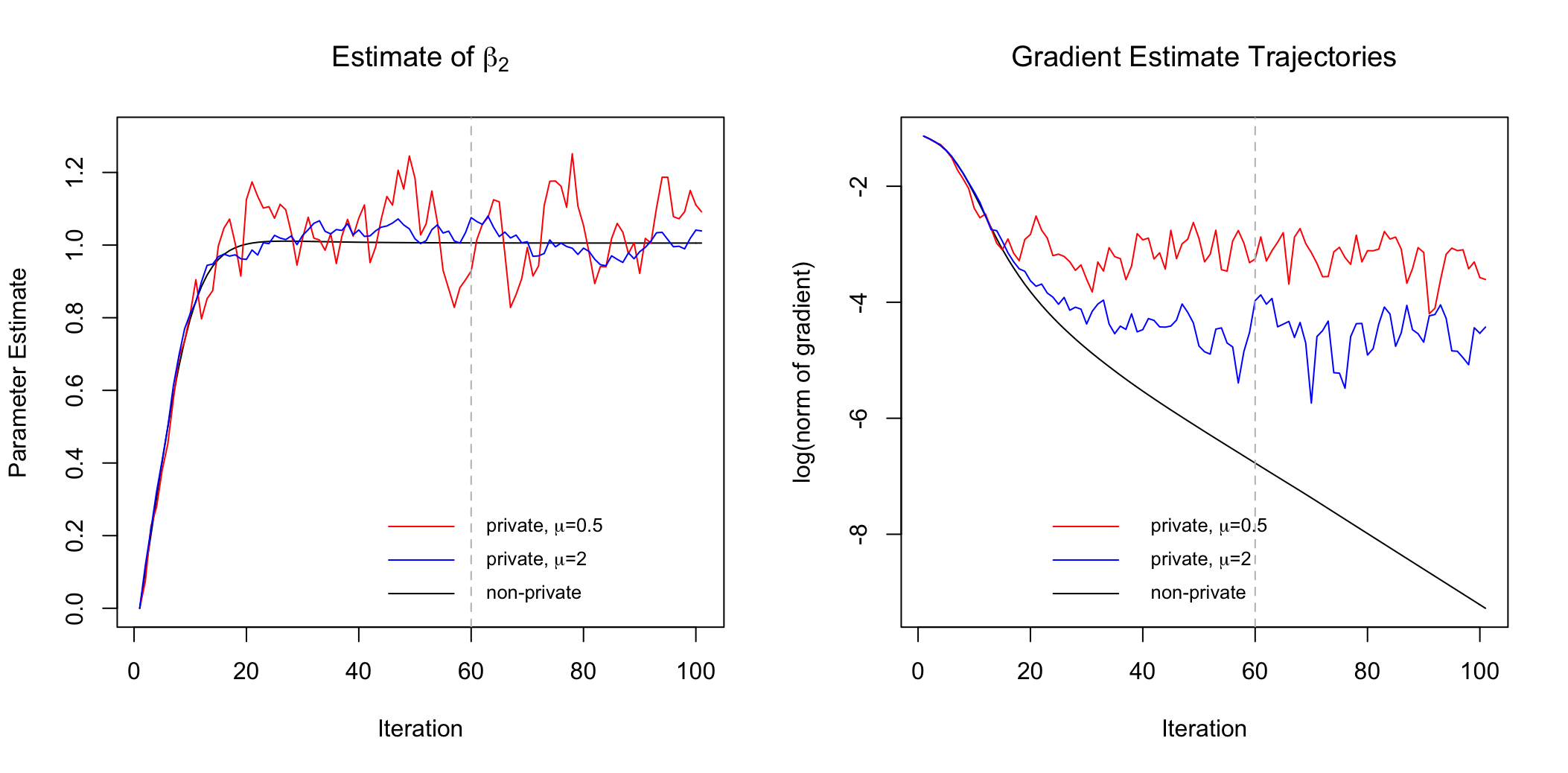} \\
    \begin{tabular}{cc}
    (a) \hspace{1in} & \hspace{1in} (b)
    \end{tabular}
    \caption{Noisy gradient descent trajectories for linear regression. (a) Estimates of a single coordinate of the regression vector. (b) Gradient of the loss function evaluated at the current iterate, plotted on a log scale.}
    \label{fig:NGD Behavior}
\end{figure}

Figure~\ref{fig:NGD Behavior}(a) shows that in the early iterations, both private and non-private estimates move away from the initial value toward the true parameter value. While the non-private version converges to the true parameter value in later iterations, the $\mu$-GDP version varies in a window around the true value as the iterates progress. 
Since the random noise added to the gradient at each iteration~\eqref{eq:NGD} has the same fixed variance for a given sample size, the gradient of our loss function does not become arbitrarily small as the number of iterations increases, nor do the values at successive iterations become arbitrarily close to each other (cf.\ Figure~\ref{fig:NGD Behavior}(b)). From a practical standpoint, we can assess convergence of our algorithm by considering whether the gradient of the loss function is still large relative to the standard deviation of the random noise term. As noted above, the maximum number of iterations $K$ must be set beforehand. However, if the loss function gradient is already small relative to the standard deviation of the noise term at some iteration $k < K$, empirical evidence suggests no practical advantage to continuing through all $K$ budgeted iterations to obtain $\left (\beta^{(K)},\sigma^{(K)} \right )$.


\subsubsection{Clipping and logistic regression}
\label{SecLogisticReg}

Applying noisy gradient descent to optimize an objective function with bounded gradients is an alternative to explicitly clipping gradients to achieve finite global sensitivity. One motivation for avoiding clipping is that the resulting estimators may fail to be consistent: Suppose $x_1,\dots,x_n\in\mathcal{X}\subseteq\mathbb{R}^m$ are i.i.d.\ according to $F_{\theta_0}$. If one seeks a differentially private maximum likelihood estimator via clipped gradients, one will in fact be computing a differentially private counterpart of the clipped maximum likelihood estimator $ \tilde\theta=T(F_n)$ defined as the solution to the equation $\frac{1}{n}\sum_{i=1}^nh_c\left(\nabla\log f(x_i;\tilde\theta)\right)=0$,
where $f(\cdot;\theta_0)$ is the density function and $h_c(z)=z\min\left\{1,\frac{c}{\|z\|_2}\right\}$ is the multivariate Huber function \cite[p.239]{hampeletal1986}. As \cite{songetal2021} demonstrate in the case of generalized linear models, this clipped estimator can be characterized as the minimizer of a Huberized version of the original loss function. However, while clipping guarantees a bounded sensitivity of the estimating equations, the clipped maximum likelihood estimator is in general not consistent, since the estimating equations are in general not unbiased, i.e., $\mathbb{E}_{F_{\theta_0}}[h_c\left(\nabla\log f(x_i;T(F_{\theta_0}))\right)]\neq 0$. Hence, even though gradient clipping is a common suggestion in the differential privacy literature, it is not the most appealing from a statistical viewpoint.
 
 We note that in a classical linear regression setting with squared loss and symmetric errors  about zero, clipping does not lead to inconsistent estimators. However, we do encounter this issue when estimating the parameters of a logistic regression model (with the cross-entropy loss). For example, we compare the performance of our proposed noisy gradient descent algorithm \eqref{eq:NGD} on simulated data from a linear regression model and a logistic regression model, using either Mallows weights (as in equation~\eqref{EqnLinearLoss}) or gradient clipping.


Figure~\ref{fig:clipping_linear} in Appendix~\ref{AppClipping} compares the results of (a) clipping and (b) Mallows weights on simulated linear regression data, showing that both methods lead to consistent estimators. Figure~\ref{fig:clipping_logistic} in Appendix~\ref{AppClipping} compares these methods on simulated data from a logistic regression model, showing that the parameter estimates from gradient clipping exhibit bias which does not shrink toward zero with the sample size.
Data for the linear regression simulation were generated according to the model $y_{i}=x_{i}^{\top} \beta +\epsilon_{i}$, where $\beta= \left (1.5,1,-1,0.5 \right )^{\top}$, $\epsilon_i \stackrel{i.i.d.}{\sim} N(0,2^2)$, and the covariate vectors are $x_{i}=\left (1, z_{i} \right )^{\top}$, where $z_{i} \overset{i.i.d.}{\sim} N(0,  \mathbb{I}_{3})$. At each sample size, 400 repetitions were performed. The gradient of the loss function was clipped such that its $\ell_2$-norm was no larger than 1. The optimization was initialized at $\beta^{(0)} = 0$, with $\sigma^{2}$ assumed to be known. For the logistic regression simulation, data were generated from the model $y_{i} \sim  \textrm{Bernoulli} \left( \left \{ 1+\textrm{exp}(-x_{i}^{\top} \beta ) \right \} ^{-1}\right)$, with the same value of $\beta$ and the same scheme for generating the covariates $x_{i}$ as in the linear regression simulation. The instantiation with Mallows weights used a weighted version of the usual cross-entropy loss, i.e.,

\begin{equation}
\label{EqnLogisticLoss}
    \mathcal{L}_n(\beta)=\frac{1}{n}\sum_{i=1}^{n} \left( -y_{i} \log \left( \frac{1}{1+\exp(x_{i}^\top \beta)} \right) + (1-y_{i}) \log \left(\frac{\exp(x_{i}^\top \beta)}{1+\exp(x_{i}^\top \beta)} \right) \right)w(x_{i}),
\end{equation}
where $w(x_{i})=\textrm{min} \left(1,\frac{2}{\left \|x_{i}  \right \|_{2}^{2}} \right)$. Since $\nabla \Loss_n (\beta) = \frac{1}{n} \sum_{i=1}^{n} ((1+e^{-x_{i}^\top \beta})^{-1} - y_{i}) x_{i}w(x_{i})$, we implement the update~\eqref{eq:NGD} with $B= \sup_{x, \beta \in \mathbb{R}^{4}, y \in \{ 0,1 \} } \| (1+e^{-x^\top \beta})^{-1} - y)xw(x) \|_{2} = \sqrt{2}$.

The gradient clipping instantiation used the usual cross-entropy loss, with loss function gradients clipped at a threshold of 1. Again, 400 repetitions were performed at each sample size, and the optimization was initialized at $\beta^{(0)} = 0$.





\section{Randomized M-estimators via noisy Newton's method}
\label{SecNewton}

We now present a noisy version of Newton's method, as a second-order alternative to the noisy gradient descent algorithm described in Section~\ref{SecNGD}. Recall that the classical Newton-Raphson algorithm finds the global optimum $\hat\theta$ of the objective~\eqref{eq:M-estimator-rho} via the iterations
\begin{equation}
    \label{eq:Newton}
    \theta^{(k+1)}=\theta^{(k)}-\left(\sum_{i=1}^n\dot\Psi(x_i,\theta^{(k)})\right)^{-1}\sum_{i=1}^n\Psi(x_i,\theta^{(k)}).
\end{equation}
The key difference between the Newton-Raphson optimization procedure and gradient descent is the use of the Hessian term $\nabla^2 \mathcal{L}_n$. Our differentially private version of this algorithm will therefore add noise to both the gradient and the Hessian of the empirical loss function. Note that we will assume throughout this section that $\nabla^2 \Loss_n$ exists everywhere.


\subsection{Matrix-valued noise for private Hessians}
\label{SecMatrixNoise}

It is important for  the convergence of the iterations \eqref{eq:Newton} that the matrix $M_n(\theta)=\frac{1}{n}\sum_{i=1}^n\dot\Psi(x_i,\theta)$ be positive definite. Since we will use noisy versions of this matrix, we also want to guarantee that the randomized quantities remain positive definite. 
%
%
%
%
Our approach exploits the fact that in many M-estimation problems, the matrix $M_n(\theta)$ can be viewed as an empirical covariance matrix of the form $\frac{1}{n}\sum_{i=1}^n a_ia_i^\top$. 
An intuitive idea for outputting a differentially private matrix is to add i.i.d.\ noise to each individual component. The following result shows that we can indeed add a symmetric matrix with appropriately scaled element-wise i.i.d.\ Gaussian noise \cite[Algorithm 1]{dworketal2014}:

\begin{lemma}
[Matrix Gaussian mechanism]
\label{lem:GaussianMatrix}
Consider a data matrix $A\in\mathbb{R}^{n\times m}$ such that each row vector $a_i$ satisfies $\|a_i\|_2\leq 1$. Further define the function $h(A)=\frac{1}{n}A^\top A$, and let $W$ be a symmetric random matrix whose upper-triangular elements, including the diagonal, are i.i.d.\ $\frac{1}{\mu n}N(0,1)$. Then the random function $\tilde{h}(A)=h(A)+W$ is $\mu$-GDP.
\end{lemma}

\begin{remark}
\label{RemProjection}
One obvious drawback of the matrix Gaussian mechanism described in Lemma~\ref{lem:GaussianMatrix} is that the noise is not positive definite. This could be problematic for the computation of differentially private positive definite Hessians, especially for small sample sizes. Note, however, that $\tilde{h}(A)$ can be projected onto a cone of positive definite matrices $\{H: H\succeq \varepsilon I\}$ defined by $\varepsilon>0$ via the convex optimization problem
$\tilde{h}(A)_+ = \arg\min_{H \succeq \varepsilon I} \|H-\tilde{h}(A)\|_2$. By definition,
$\|\tilde{h}(A)_+-\tilde{h}(A)\|_2 \leq
\|h(A)-\tilde{h}(A)\|_2$, so the triangle inequality yields
\begin{equation*} \label{eqProjectL2}
  \| \tilde{h}(A)_+-{h}(A)\|_2 \leq  \|\tilde{h}(A)_+-\tilde{h}(A)\|_2
  + \| \tilde{h}(A)-{h}(A)\|_2 \leq 2  \|\tilde{h}(A)-{h}(A)\|_2.
\end{equation*}
Hence, the price to pay is no more than a factor of two, which does not affect the order of the convergence rate. In practice, the projection amounts to truncating the eigenvalues as $\max\{\lambda_j,\varepsilon\}$ \citep[p.399]{boydandvanderberghe2004}. The projected matrix is clearly also differentially private, as it results from a deterministic post-processing step applied to a differentially private output.
\end{remark}

\subsection{Differentially private Newton's method and convergence analysis}
\label{SecNewtonAlgo}

We will require some regularity conditions on the Hessian matrix $\nabla^2\mathcal{L}_n(\theta)$. In particular, we need the Hessian to be factorizable in a way that allow us to leverage Lemma \ref{lem:GaussianMatrix}, which is typical for problems with linear predictors such as linear regression, robust regression, and generalized linear models. We will also assume that the spectral norm of the Hessian is uniformly bounded.


 
\begin{condition}
\label{ass:Hessian}
The Hessian is positive definite for $\theta\in\mathcal{B}_r(\theta^*)$ and of the form $\nabla^2\mathcal{L}_n(\theta)=\frac{1}{n}A^\top A=\frac{1}{n}\sum_{i=1}^na(x_i,\theta)a(x_i,\theta)^\top$, where
$\sup_{x,\theta}\|a(x,\theta)\|_2^2 \leq \bar{B}<\infty$.
\end{condition}


We note that the constant $\bar{B}$ introduced in Condition \ref{ass:Hessian} implies $\tau_2$-smoothness with $\tau_2 = \frac{\bar{B}}{2}$.
We are now ready to present our differentially private counterpart of the Newton iterates~\eqref{eq:Newton}. We propose a noisy damped quasi-Newton method that follows the updates 
\begin{equation}
    \label{eq:NNewton}
    \theta^{(k+1)}=\theta^{(k)}-\eta H_k^{-1}\left(\frac{1}{n}\sum_{i=1}^n\Psi(x_i,\theta^{(k)})+\frac{2B\sqrt{2K}}{\mu n}Z_{k}\right),
\end{equation}
 where $\eta>0$ is the step size, $B \ge \sup_{x\in\mathcal X,\theta\in\Theta}\|\Psi(x,\theta)\|_2$ upper-bounds the gradients as before, $\bar{B}$ is as in Condition \ref{ass:Hessian},  $\{Z_k\}_{k=1}^K$ is a sequence of i.i.d.\ standard $p$-dimensional Gaussian random vectors, and $H_k=\frac{1}{n}\sum_{i=1}^n\dot\Psi(x_i,\theta^{(k)}) +\frac{2\bar{B}\sqrt{2K}}{\mu n}W_k$ is a noisy Hessian, where $\{W_k\}_{k=1}^K$ is a sequence of i.i.d.\  symmetric random matrices whose upper-triangular elements, including the diagonals, are i.i.d.\ standard normal. In practice, one can set the smallest eigenvalues of $H_k$  to some small positive value $\varepsilon$, as discussed in Remark~\ref{RemProjection}. If $\varepsilon\to 0$, the effect will be asymptotically negligible and will not affect the theoretical conclusions stated in our paper. As in the case of noisy gradient descent, we choose the initial iterate $\theta^{(0)}$ in a non-data-dependent manner (e.g., $\theta^{(0)} = 0$), with specific choices of $\theta^{(0)}$ to be highlighted in the simulations to follow. We note that this noisy Newton’s method algorithm can be interpreted as iterative applications of the SSP procedure (introduced by \cite{vu2009differential} and applied to linear regression by \cite{wang2018regression}), as each step minimizes a certain least squares objective, although this interpretation does not directly imply the convergence results of Theorems \ref{thm:NNewton} and \ref{thm:NNewtonSC} below.

Finally, we note that the level of noise introduced to both the gradient and Hessian terms to ensure privacy is $O\left(\frac{1}{n}\right)$, which is appreciably smaller than the $\Theta\left(\frac{1}{\sqrt{n}}\right)$ fluctuations involved if we were to treat the gradient and Hessian as empirical versions of $\nabla \Loss_n(\theta^{(k)})$ and $\nabla^2 \Loss_n(\theta^{(k)})$ (or the level of noise introduced by subsampling-type stochastic optimization methods)~\citep{RooMah19}. To our knowledge, the bounds derived for the iterative algorithms studied in this paper do not immediately follow from any known results in stochastic optimization.


\subsubsection{Local strong convexity theory}
\label{SecLSC}

Assuming local strong convexity, we can show that the noisy Newton algorithm leads to the same statistical error bounds as noisy gradient descent, but requires fewer iterations to achieve convergence. In order to establish this result, we will also require the Hessian to be Lipschitz continuous, which is commonly assumed to establish the quadratic convergence of Newton's method under strong convexity \cite[Ch 9.5]{boydandvanderberghe2004}:
\begin{condition}
[Lipschitz continuity of Hessian]
\label{ass:Lipschitz_Hessian}
The Hessian $\nabla^2\mathcal{L}_n(\theta)$ is  $L$-Lipschitz continuous,  i.e., $\|\nabla^2\mathcal{L}_n(\theta_1)-\nabla^2\mathcal{L}_n(\theta_2)\|_2\leq L\|\theta_1-\theta_2\|_2$ for all $\theta_1,\theta_2\in \real^p$.
\end{condition}

The following theorem, proved in Appendix~\ref{AppThmNewton}, shows that under standard regularity conditions for robust M-estimators, $\Omega(\log\log n)$ iterations of the noisy Newton step~\eqref{eq:NNewton} suffice to obtain a $\mu$-GDP estimator lying in a neighborhood of $\hat\theta$ whose radius is proportional to the privacy-inducing noise of the algorithm.

\begin{theorem}
\label{thm:NNewton}
Assume Conditions \ref{ass:bounded_gradient}, \ref{ass:RSC/RSS}, \ref{ass:Hessian}, and~\ref{ass:Lipschitz_Hessian}, and
suppose $\hat\theta \in \mathcal{B}_{r/2}(\theta_0)$
and $\|\nabla\mathcal{L}_n(\theta^{(0)})\|_2 \leq \min\left\{\tau_1 r, \frac{\tau_1^2}{L}\right\}$.
Let $n=\Omega\left(\frac{\sqrt{Kp \log(Kp/\xi)}}{\mu}\right)$ and $K =\Omega ( \log\log n )$. The $K^{\text{th}}$ noisy Newton iterate~\eqref{eq:NNewton} with $\eta = 1$ 
satisfies (i) $\theta^{(K)}$ is $\mu$-GDP, and (ii) $\|\theta^{(K)}-\hat\theta\|_2\leq C\frac{\sqrt{Kp\log(K/\xi)}}{\mu n}$, with probability at least $1-\xi$, where $C$ depends on $(L, \tau_1, B, \bar{B})$. 
\end{theorem}

\begin{remark}
\label{cor:NNewton_bad_starting_value}
Theorem \ref{thm:NNewton} imposes a slightly different initialization condition than Theorem \ref{thm:NGD}, namely that the gradient of the loss at the initial point must be small. However, this condition can similarly be guaranteed after a few initial iterates of noisy gradient descent: taking $\Delta = \min\left\{\frac{r^2}{4}\tau_1, \frac{\tau_1^3 r^2}{4\tau_2^2}, \frac{\tau_1^5}{4\tau_2^2L^2}\right\}$ implies that $\theta^{(K_0)}\in \mathcal{B}_{r}(\theta_0)$ by Proposition \ref{prop:NGD_bad_starting_value} and Lemma \ref{lem:LSCball}, where $K_0 = c_0 \|\theta^{(0)} - \thetahat\|_2^2$.
Hence, local strong convexity further implies that $\Delta \ge \mathcal{L}_n(\theta^{(K)})-\mathcal{L}_n(\hat\theta)\geq \tau_1\|\theta^{(K)}-\hat\theta\|_2^2$, which combined with smoothness gives
$\|\nabla\mathcal{L}_n(\theta^{(K_0)})\|_2 \leq 2\tau_2\|\theta^{(K)}-\hat\theta\|_2 \leq 2\tau_2\sqrt{\frac{\Delta}{\tau_1}} \leq \min\left\{\tau_1 r, \frac{\tau_1^2}{L}\right\}$.
\end{remark}
 

We note that in practice, in order to benefit from the improved iteration complexity of the noisy Newton algorithm, one should be able to assess whether the initial condition $\|\nabla\mathcal{L}_n(\theta^{(k)})\|_2\leq \frac{\tau_1^2}{L}$ is met. If this condition fails to hold, the algorithm can diverge, as illustrated in Section~\ref{SecPractical}. We note that this is analogous to the well-known behavior of Newton's method. This drawback has been addressed in the literature by backtracking, but even that analysis relies on strong convexity \cite[Ch. 9.5]{boydandvanderberghe2004}. We present a private implementation of backtracking line search in Appendix \ref{AppBacktrack}. In the context of differential privacy, this backtracking implementation consumes additional privacy budget. In Section~\ref{SecPractical}, we present an alternative approach which comes at no additional privacy cost.


Finally, as in the case of Corollary~\ref{CorNGD}, it follows directly from Theorem \ref{thm:NNewton} and standard M-estimation theory that the noisy Newton algorithm leads to $\mu$-GDP estimators that are $\sqrt{n}$-consistent and asymptotically normally distributed.

 \begin{corollary}
 \label{corollary2}
 Assume the conditions of Theorem \ref{thm:NNewton} and let  $\theta_0$ be such that $\mathbb{E}[\Psi(Z,\theta_0)]=0$ and $\mathbb{E}_F[\dot\Psi(x,\theta)]$ is continuous and invertible in a neighborhood of $\theta=\theta_0$. Then
\begin{equation*}
\text{(i) } \theta^{(K)}-\theta_0=\hat\theta-\theta_0+O_p\left(\frac{\sqrt{K\log K}}{\mu n}\right), \quad \text{and} \quad \text{(ii) } \sqrt{n}(\theta^{(K)}-\theta_0)\to_d N(0,V (\Psi , F_{\theta_0})).
\end{equation*}
 \end{corollary}


\subsubsection{Self-concordance theory}

We now derive results for \emph{global} convergence of our noisy Newton algorithm, under an alternative assumption of self-concordance. As discussed in Section~\ref{SecLSC} above, fast convergence of the noisy Newton algorithm (Theorem~\ref{thm:NNewton}) is only guaranteed for a suitably close initialization, which we propose to obtain via an initial batch of iterates from noisy gradient descent (cf.\ Remark~\ref{cor:NNewton_bad_starting_value}). However, one practical drawback of the local strong convexity analysis is that we need a method for detecting when the gradient condition $\|\nabla \Loss_n(\theta^{(0)})\|_2 \le \frac{\tau_1^2}{L}$ is obtained by the noisy gradient descent iterates, which requires (approximately) computing the LSC parameter $\tau_1$. As the results of this section show, imposing a self-concordance rather than LSC assumption on $\Loss_n$ allows us to (a) obtain a more easily checked initial value condition on $\theta^{(0)}$, and (b) use noisy Newton iterates for the entire duation of the algorithm, rather than switching from noisy gradient descent to noisy Newton. As the examples in Section~\ref{SecExamples} illustrate, self-concordance is satisfied by several robust M-estimators that are useful both from a privacy and statistical error perspective.

We will use the following notion of generalized self-concordance~\citep{sun2019generalized}:
\begin{definition}
[Generalized self-concordance]
A function $f:\mathbb{R} \rightarrow \mathbb{R}$ is \emph{$(\gamma, \nu)$-self-concordant} if $|f'''(x)| \le \gamma \left(f''(x)\right)^{\nu/2}$, for all $x$. A multivariate function $f:\mathbb{R}^p \rightarrow \mathbb{R}$ is $(\gamma, \nu)$-self-concordant if
$\left|\langle \nabla^3 f(x) [v] u, u \rangle \right| \le \gamma \|u\|_{\nabla^2 f(x)}^2 \|v\|_{\nabla^2 f(x)}^{\nu-2} \|v\|_2^{3-\nu}$, for all $x, u, v \in \real^p$.
\end{definition}
We will primarily be interested in $\nu \in \{2, 3\}$. Appendix~\ref{AppSC} contains several useful results about generalized self-concordant functions.
The following theorem concerns convergence of the noisy Newton iterates assuming a starting value condition, stated in terms of the \emph{Newton decrement} function $\lambda(\theta) := \left((\nabla \Loss_n(\theta))^T \nabla^2 \Loss_n(\theta)^{-1} (\nabla \Loss_n(\theta))\right)^{1/2}$, which appears in the usual convergence analysis of Newton's method under self-concordance.
The proof of the theorem leverages stability properties of the Hessian of generalized self-concordant functions, as outlined in Appendix~\ref{AppStable}, allowing us to derive uniform lower bounds on the minimum eigenvalue of the Hessian at successive iterates without requiring local strong convexity. In particular, Lemma~\ref{LemSCMult} introduces a quantity $\tau_{1,r}$ that bounds the minimum eigenvalues of Hessians at successive iterates, which we use with $r = \frac{2}{\gamma}$ in place of the LSC parameter.

\begin{theorem}
\label{thm:NNewtonSC}
Suppose $\Loss_n$ satisfies Condition~\ref{ass:Hessian} and is $(\gamma, 2)$-self-concordant, and
suppose $\hat\theta \in \mathcal{B}_{1/\gamma}(\theta_0)$ and $\lambda_{\min}^{-1/2} (\nabla^2 \Loss_n(\theta^{(0)})) \lambda(\theta^{(0)}) \le \frac{1}{16 \gamma}$.
Let $n = \Omega\left(\frac{\sqrt{Kp \log(Kp/\xi)}}{\mu}\right)$ and $K = \Omega(\log\log n)$. Then the $K^{\text{th}}$ noisy Newton iterate~\eqref{eq:NNewton} with $\eta = 1$ satisfies (i) $\theta^{(K)}$ is $\mu$-GDP, and (ii) $\|\theta^{(K)}-\hat\theta\|_2\leq C\frac{\sqrt{Kp\log(Kp/\xi)}}{\mu n}
    $, with probability at least $1-\xi$, where $C > 0$ is a constant depending on $\gamma$, $B$, and $\bar{B}$.  
\end{theorem}

The proof of Theorem~\ref{thm:NNewtonSC} is provided in Appendix~\ref{AppThmNewtonSC}. Whereas the traditional analysis of Newton's method via (canonical) self-concordance shows that $\lambda(\theta^{(k)})$ decays quadratically with $k$, our proof for $(\gamma, 2)$-self-concordant functions, based on adapting the arguments of \cite{sun2019generalized}, first establishes quadratic convergence of the quantity $\lambda_{\min}^{-1/2} (\nabla^2 \Loss_n(\theta^{(k)})) \lambda(\theta^{(k)})$.

Comparing the initial value assumptions of Theorem~\ref{thm:NNewtonSC} with those of Theorem~\ref{thm:NNewton}, the only condition we need to check involves comparing the rescaled Newton decrement at $\theta^{(0)}$ with $\frac{1}{16\gamma}$. In practice, $\gamma$ is directly computable from the M-estimator used in the regression problem (cf.\ Section~\ref{SecExamples} below). Hence, the initial condition for self-concordant functions is much more practically checkable than the one stated in Theorem~\ref{thm:NNewton} for functions which satisfy LSC.

\begin{remark}
\label{RemNewtonGenSC}
In light of results on second-order inexact oracle algorithms for $(\gamma, 3)$-self-concordant functions in optimization, it is natural to wonder if a version of Theorem~\ref{thm:NNewtonSC} could also be derived for $(\gamma, 3)$-self-concordant functions. Indeed, as the examples in Section~\ref{SecExamples} will show, loss functions of interest have been proposed which were designed to be $(\gamma, 3)$-self-concordant. In fact, the convergence results stated in Theorem~\ref{thm:NNewtonSC} actually hold for $(\gamma, \nu)$-self-concordant losses, for any $\nu \ge 2$: By Lemma~\ref{LemSC2}, the assumption that $\Loss_n$ is $(\gamma, \nu)$-self-concordant, together with Condition~\ref{ass:Hessian}, implies that $\Loss_n$ is also $(\bar{B}^{\nu/2 - 1} \gamma, 2)$-self-concordant.
\end{remark}
The following proposition can be used to establish \emph{global convergence} under the same conditions as Theorem~\ref{thm:NNewtonSC}. It shows that the sub-optimality gap of successive noisy Newton iterates  decreases geometrically, without assuming a sufficiently close initialization. The proof is contained in Appendix~\ref{AppThmNewtonGlobal}; the main argument adapts the global convergence analysis of \cite{karimireddyetal2018} for (non-noisy) Newton's method under self-concordance. This analysis builds on the fact that $(\gamma,2)$-self concordance implies $(\exp(\gamma R_0))$-stable Hessians, where $R_0$ is as stated in Proposition \ref{ThmNewtonGlobal}. In the statement of the proposition, we abuse notation slightly and use $\tau_{1, R_0}$ to denote the quantity defined in Lemma~\ref{LemSCMult} with $\theta_0$ replaced by the global optimum $\thetahat$.


\begin{proposition}
\label{ThmNewtonGlobal}
Suppose $\Loss_n$ is a $(\gamma, 2)$-self-concordant function which satisfies Condition~\ref{ass:Hessian}. Let $\Delta_0 = \Loss_n(\theta^{(0)}) - \Loss_n(\thetahat)$, and define $R_0 = g^{-1}\left(\frac{\gamma^2\Delta_0}{\lambda_{\min}(\nabla^2 \Loss_n(\thetahat))}\right)$, where $g(t) := e^{-t} + t - 1$ for $t > 0$.
Let $K \ge 1$, and suppose the step size satisfies $\eta \le \min\left\{\frac{1}{2\exp(2\gamma R_0)}, \frac{4\tau_{1,R_0}^2}{\bar{B}^2}\right\}$ and the sample size satisfies $n = \Omega\left(\frac{\sqrt{Kp \log(Kp/\xi)}}{\mu}\right)$. Then with probability at least $1-\xi$, the Newton iterates satisfy $\mathcal{L}_n(\theta^{(K)}) - \mathcal{L}_n(\hat{\theta}) \le \left(1 - \frac{\eta}{\exp(2\gamma R_0)}\right)^K \left(\mathcal{L}_n(\theta^{(0)}) - \mathcal{L}_n(\hat{\theta})\right) + r_{priv}$, where $r_{priv} = C' 
\frac{\sqrt{Kp\log(Kp/\xi)}}{\mu n}
$, and $C' > 0$ is a constant depending on $R_0$, $B$, and $\bar{B}$.
\end{proposition}


\begin{remark}
\label{rem:init_SC}
As argued in Remark~\ref{cor:NNewton_bad_starting_value}, the suboptimality bound
$\Loss_n(\theta) \le \Loss_n(\thetahat) + \Delta$ implies that $\|\nabla \Loss_n(\theta)\|_2 \le \bar{B} \sqrt{\frac{\Delta}{\tau_{1,R_0}}}$, for $\Delta \le \Delta_0$, since $\theta \in \mathcal{B}_{R_0}(\thetahat)$ by Lemma~\ref{LemB4} and $\tau_{1, R_0}$ takes the place of the LSC parameter. Hence, if we take $\Delta \le \min\left\{\Delta_0, \frac{\tau_{1,R_0}}{(32\gamma \bar{B})^2}\right\}$, we can guarantee that $\lambda(\theta) = \|\nabla \Loss_n(\theta)\|_{\nabla^2 \Loss_n(\theta)^{-1}} \le 2\tau_{1,R_0} \cdot \bar{B} \sqrt{\frac{\Delta}{\tau_{1,R_0}}} \le \frac{1}{16\gamma}$. Hence, Proposition~\ref{ThmNewtonGlobal} can be used to ensure that the starting condition of Theorem~\ref{thm:NNewtonSC} holds after $O\left(\log\left(\min\left\{\Delta_0, \frac{\tau_{1,R_0}}{(32\gamma \bar{B})^2}\right\}\right)\right)$ iterates, which is negligible in light of the $K = \Omega(\log \log n)$ iterates used in Theorem~\ref{thm:NNewtonSC}.
\end{remark}
\begin{remark}
\label{rem:Prop2_rates}
One can further refine the suboptimality bound of Proposition \ref{ThmNewtonGlobal} in order to obtain convergence rates for $\theta^{(K)}$ when $K$ is sufficiently large. In particular, one can adapt the argument in the proof of Theorem \ref{thm:NGD} in order to guarantee that with high probability, $\|\theta^{(K)}-\hat\theta\|_2=O(r_{priv})$; however, the suboptimality guarantee as stated is sufficient, since it already implies a negligible initial number of iterations prior to applying Theorem~\ref{thm:NNewtonSC}.
\end{remark}

As in the case of Corollaries~\ref{CorNGD} and~\ref{corollary2}, Theorem~\ref{thm:NNewtonSC} immediately implies the following:

\begin{corollary}
 \label{corollary3}Assume the conditions of Theorem \ref{thm:NNewtonSC} and let  $\theta_0$ be such that $\mathbb{E}[\Psi(Z,\theta_0)]=0$ and $\mathbb{E}[\dot\Psi(x,\theta)]$ is continuous and invertible in a neighborhood of $\theta_0$. Then
 \begin{equation*}
 \text{(i) } \theta^{(K)}-\theta_0=\hat\theta-\theta_0+O_p\left(\frac{\sqrt{K \log(K)}}{\mu n}\right), \quad \text{and} \quad \text{(ii) } \sqrt{n}(\theta^{(K)}-\theta_0)\to_d N(0,V (\Psi , F_{\theta_0})).
 \end{equation*}
 \end{corollary}

\subsection{From univariate to multivariate self-concordant functions}
\label{SecExamples}

Although we have proved our optimization results for general functions, our primary focus in this paper is regression $M$-estimators. Accordingly, we now discuss how univariate self-concordant functions can be composed to obtain self-concordant $M$-estimators. We will focus on $(\gamma, \nu)$-self-concordance with $\nu = 2$ or 3. We will discuss specific examples in Appendix~\ref{AppSCExamples}.

\noindent{\textbf{Regression with Mallows weights:}}
We first consider Mallows-weighted estimators
\begin{equation}
\label{EqnMallows}
\Loss_n(\theta) = \frac{1}{n} \sum_{i=1}^n \rho(y_i, x_i^T \theta) w(x_i).
\end{equation}
One can leverage recent results in \cite{sun2019generalized} to show that if $\rho$ is univariate $(\gamma,2)$-self-concordant, that the Mallows estimator is also (multivariate) self-concordant:

\begin{lemma}
\label{LemMallows}
Suppose $\rho$ is $(\gamma, 2)$-self-concordant. Then the Mallows loss~\eqref{EqnMallows} is $(\gamma \max_i \|x_i\|_2, 2)$-self-concordant.
\end{lemma}




The proof of Lemma~\ref{LemMallows} is provided in Appendix~\ref{AppSCLemma}.
As Lemma~\ref{LemMallows} suggests, Mallows weighting is only desirable when the Euclidean norm of the covariates is bounded.
Note that if we instead assume that $\rho$ is $(\gamma, 3)$-self-concordant, Lemma~\ref{LemSCAffine} in Appendix~\ref{AppSCLemma} implies that $\rho(y_i, x_i^T \theta)$ is $(\gamma, 3)$-self-concordant. Then Lemma~\ref{LemSCSum} implies that $\Loss_n$ is $\left(\gamma, \sqrt{\frac{n}{\max_i w(x_i)}}, 3\right)$-self-concordant. However, both the growth of the self-concordance parameter with $\sqrt{n}$ and the inverse dependence on the Mallows weights is problematic.

\noindent{\textbf{Linear regression with Schweppe weights:}}
We now consider linear regression losses
\begin{equation}
\label{EqnSchweppeLoss}
\Loss_n(\theta) = \frac{1}{n} \sum_{i=1}^n \rho((y_i-x_i^T \theta)v(x_i)),
\end{equation}
where $\rho: \real \rightarrow \real$. In particular, we use weight functions $v:\mathbb{R}^p\mapsto\mathbb{R}_{\geq 0}$ such that $\|v(x)x\|_2\leq C <\infty$. This objective function corresponds to a robust regression loss with Schweppe weights \citep[Ch.\ 6.3]{hampeletal1986}. It is an alternative to the Mallows loss function, and bypasses the problem of requiring covariates to be bounded in order to obtain a multivariate self-concordance parameter which agrees with the univariate self-concordance parameter up to a scale factor.

We can derive the following result, also proved in Appendix~\ref{AppSCLemma}:
\begin{lemma}
\label{LemSchweppe}
Suppose the Schweppe weights satisfy $\|v(x) x\|_2 \le C < \infty$.
\begin{itemize}
\item[(i)] Suppose $\rho$ is $(\gamma, 2)$-self-concordant. Then the Schweppe loss~\eqref{EqnSchweppeLoss} is $(\gamma C, 2)$-self-concordant.
\item[(ii)] Suppose $\rho$ is $(\gamma, 3)$-self-concordant. Also suppose $\|\rho''\|_\infty < \infty$. Then the Schweppe loss~\eqref{EqnSchweppeLoss} is $(\gamma C \|\rho''\|_\infty^{1/2}, 2)$-self-concordant.
\end{itemize}
\end{lemma}


\subsection{Practical considerations}
\label{SecPractical}

Theorems~\ref{thm:NNewton} and~\ref{thm:NNewtonSC} prove that the noisy Newton algorithm with $\eta=1$ converges quadratically to a nearly-optimal neighborhood of the target parameter $\hat\theta$ when the starting value lies in a suitable neighborhood of the solution. Figure~\ref{fig:log_gradient_comparison}(a) illustrates this improved performance of (noisy) Newton's method relative to (noisy) gradient descent, applying both methods to data simulated from a linear regression model.
%
\begin{figure}[h]
\begin{tabular}{cc}
    \includegraphics[width=7.5cm]{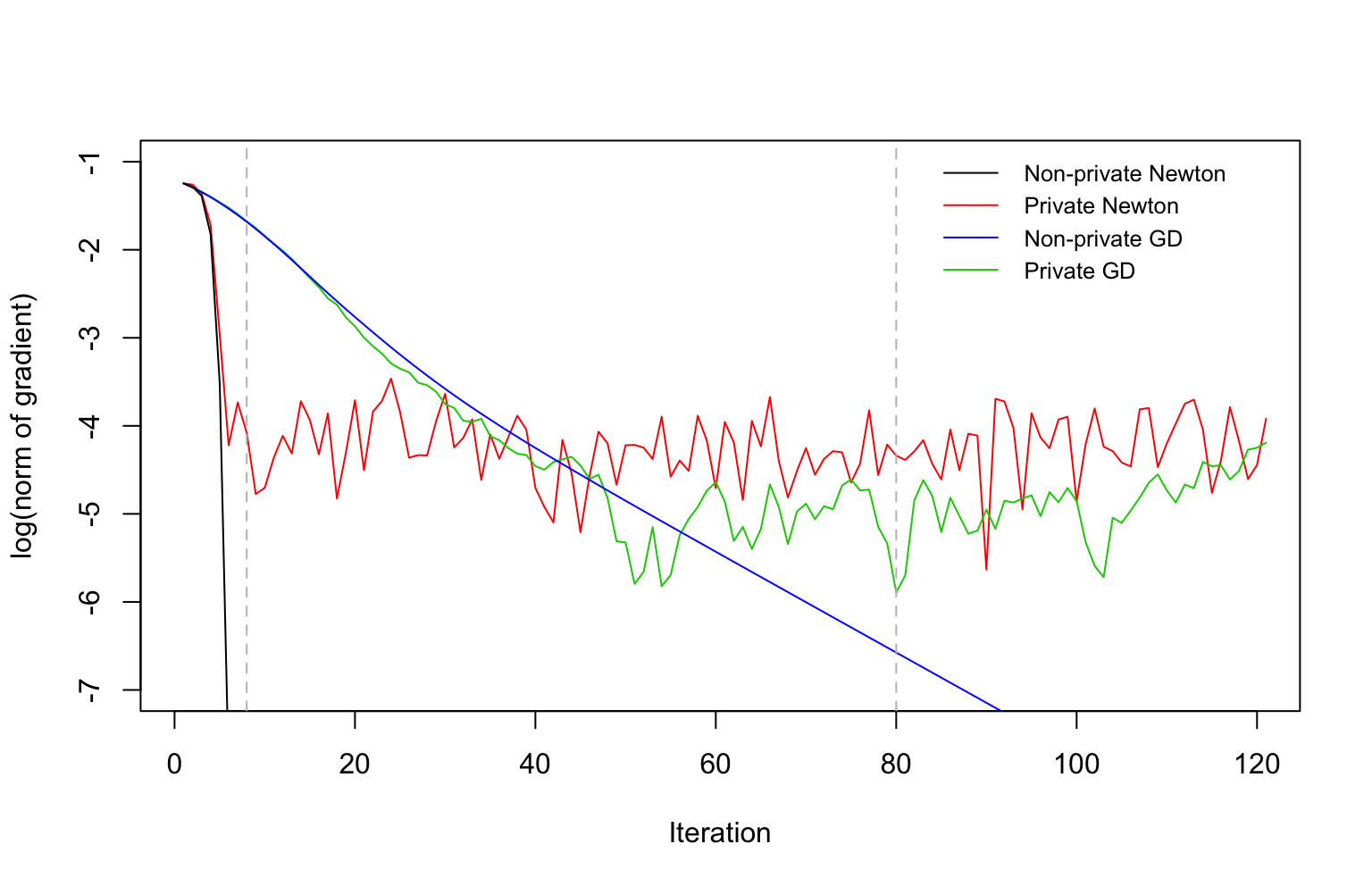} &
    \includegraphics[width=7cm]{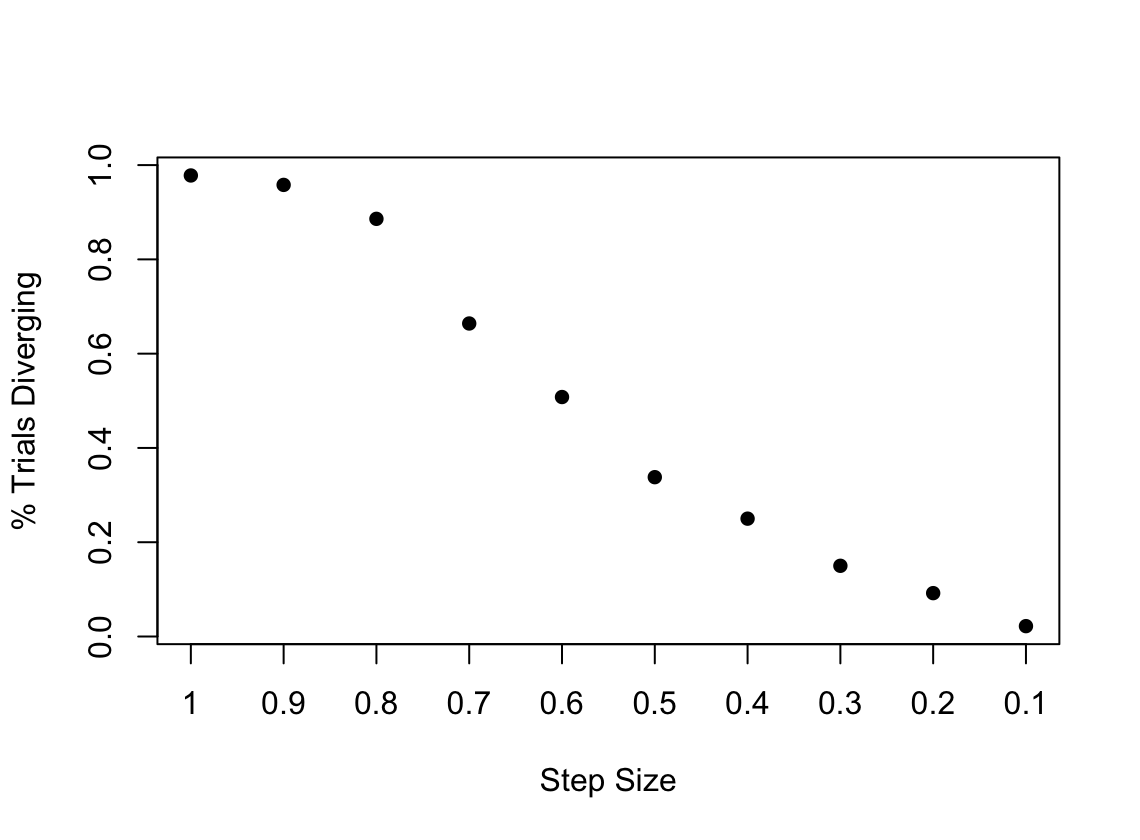} \\
    (a) & (b)
    \end{tabular}
    \caption{(a) Noisy Newton's method vs.\ gradient descent. The vertical axis records the norm of the gradient trajectory on a log scale. (b) Divergence of noisy damped Newton. When $\eta = 1$ (pure Newton), the algorithm may not converge if the initial point is far from the global optimum $\thetahat$. Using smaller values of $\eta$ remedies this problem.}
    \label{fig:log_gradient_comparison}
\end{figure}
In this example, the noisy Newton algorithm is calibrated to achieve 2-GDP in 8 iterations, while noisy gradient descent is calibrated to achieve 2-GDP in 80 iterations. This is meant to reflect the fact that Newton's method tends to converge faster, so in practice, we would schedule fewer iterations of Newton's method compared to gradient descent. To retain the 2-GDP privacy guarantees, we would terminate the algorithms at the gray lines on the plot (iteration 8 for Newton and iteration 80 for gradient descent), but for purpose of illustration, we have forced both algorithms to continue, with the same amount of noise added in the extra steps. All optimizations were initialized at $\beta^{(0)} = 0$, with $\sigma^{2}$ assumed to be known. Taking the loss function \eqref{EqnLinearLoss}, the Hessian is $\nabla^{2} \Loss_{n}(\beta) = \frac{1}{n} \sum_{i=1}^{n} \mathbbm{1} \Bigl\{ \Big| \frac{y_{i}-x_{i}^{\top} \beta}{\sigma} \Big| <c \Bigr\} x_{i} x_{i}^{\top} w(x_{i})$. Thus, to implement the update~\eqref{eq:NNewton}, we choose $\bar{B} = \sup_{x, \beta \in \mathbb{R}^{4}, y \in \mathbb{R} } \Big\| \mathbbm{1} \Bigl\{ \Big| \frac{y-x^{\top} \beta}{\sigma} \Big| <c \Bigr\} x \sqrt{w(x)} \Big\|_{2}^{2} = 2$.

\subsubsection{Convergence issues}

Since the starting value condition required for the quadratic convergence of Newton's method cannot typically be guaranteed a priori, one must consider two regimes: (i) damped Newton updates with $\eta<1$, and (ii) pure noisy Newton updates with $\eta=1$. Indeed, even in general non-private settings, Newton’s method with step size $\eta = 1$ may or may not converge depending on the initial point chosen. Figure~\ref{fig:log_gradient_comparison}(b) displays the results of simulations using the damped version of Newton's method. The data in this simulation, $\left \{ \left (x_{i},y_{i} \right ) \right \} _{i=1}^{1000}$, are generated according to the model $y_{i}=x_{i}^{\top} \beta +\epsilon_{i}$, where $\beta= \left (1,1,1,1 \right )^{\top}$, $\epsilon_{1},\dots,\epsilon_{n} \overset{i.i.d.}{\sim} N(0,2^2)$, and the covariate vectors are given by $x_{i}=\left (1, z_{i} \right )^{\top}$, where $z_{i} \overset{i.i.d.}{\sim} N(0, 2^2 \cdot \mathbb{I}_{3})$. At each step size, 500 repetitions were performed, each initialized at $\beta^{(0)} = 0$ and $\sigma^{(0)} = 1$, and tuned for a 2-GDP privacy guarantee.
 For the same sample size, initial point, and true parameter values, we see that using damped Newton with a fixed step size $\eta$ tends to help with the divergence problem: With all else equal, the proportion of trials leading to divergent iterates tends to decrease as the step size shrinks toward zero, although at very small step sizes, the algorithm may fail to converge within the budgeted number of iterations.
\begin{figure}[h]
    \centering
    \includegraphics[width=15cm]{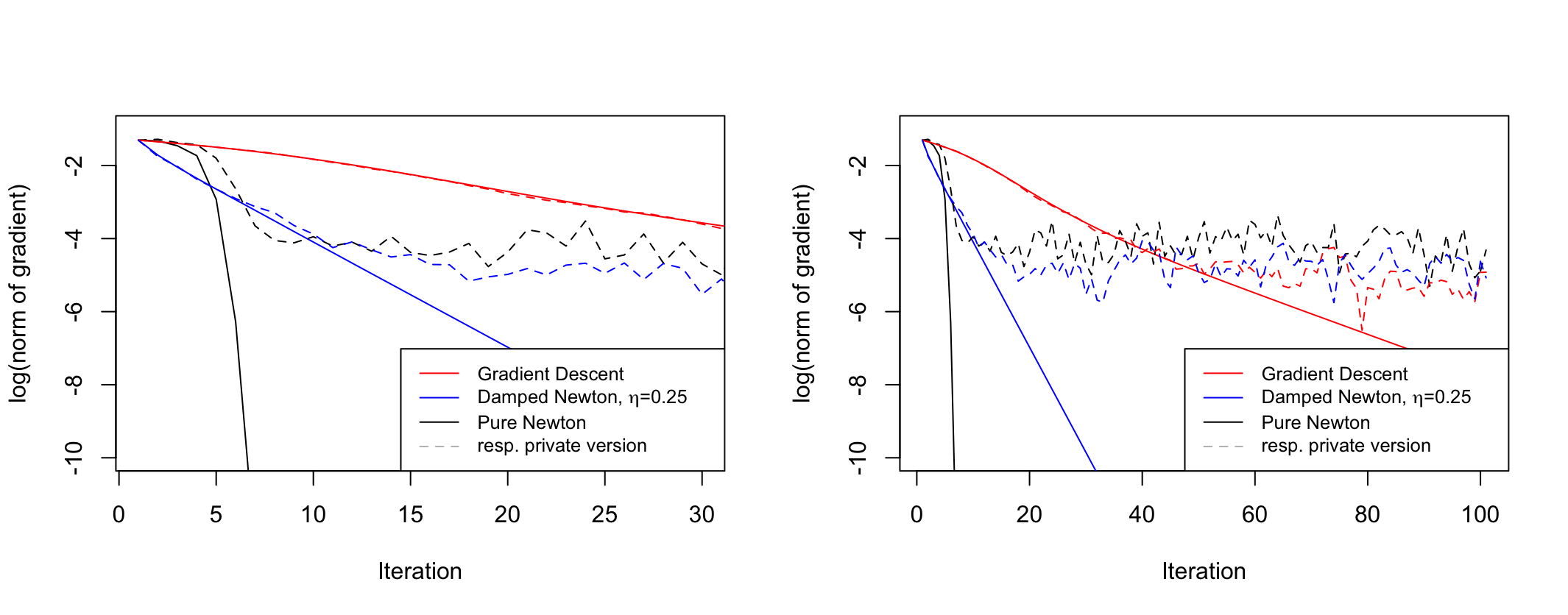}
    \caption{Trajectories of noisy gradient descent, pure Newton, and damped Newton algorithms. All optimizations were initialized at $\beta^{(0)} = 0$, with $\sigma^{2}$ assumed to be known.}
    \label{fig:damped_newton4}
\end{figure}

Although taking smaller Newton step sizes guarantees global convergence of the algorithm, it only guarantees linear convergence, as illustrated in Figure \ref{fig:damped_newton4}. 
A popular strategy employed to in order to obtain a quadratically convergent Newton algorithm is to rely on a damped version of Newton’s method, in which the update is scaled by an adaptively-chosen step size via backtracking line search \citep[Ch.\ 9.5.2]{boydandvanderberghe2004}. 
In Appendix \ref{AppBacktrack}, we outline a differentially private implementation of backtracking line search for selecting these step sizes.


\subsubsection{A private alternative to backtracking line search}

We now depart from the ideas behind backtracking and propose to privately estimate the magnitude of successive Newton steps, in order to determine whether the starting conditions of Theorems~\ref{thm:NNewton} and~\ref{thm:NNewtonSC} are met.

\noindent{\textbf{Locally strongly convex case:}}
This is particularly challenging if one relies on local strong convexity theory, as one needs to check whether  $\|\nabla\mathcal{L}_n(\theta^{(k)})\|_2\leq \frac{\tau_1^2}{L}$. In practice, one would check whether a noisy gradient meets this inequality, but one also needs to explicitly evaluate $L$ and $\tau_1$. We discuss this issue in the context of a Mallows estimator for linear regression with the smoothed Huber loss of Example \ref{ExaSmooth}
discussed in Appendix \ref{AppSCExamples}. The Lipschitz constant $L$ can be estimated as follows: The intermediate value theorem and the Cauchy-Schwarz inequality yield the upper bound $\|\nabla^2\mathcal{L}_n(\theta_1)-\nabla^2\mathcal{L}_n(\theta_2)\|_2 \le \sup_r|\psi''_{c,h}(r)|\lambda_{\max}\left( \frac{1}{n}\sum_{i=1}^nx_ix_i^\top w(x_i)\|x_i\|_2\right)\|\theta_1-\theta_2\|_2$.
 One can obtain a differentially private estimate of $\frac{1}{n}\sum_{i=1}^nx_ix_i^\top w(x_i)\|x_i\|_2$ with the matrix Gaussian mechanism with an appropriate choice of $w(x_i)$, e.g., defining $w(x_i)=\max\left\{1,\frac{1}{\|x_i\|_2^3}\right\}$. This automatically leads to a private estimate of $\lambda_{\max}\left(\frac{1}{n}\sum_{i=1}^nx_ix_i^\top w(x_i)\|x_i\|_2\right)$, and hence of $L$.
Estimating the LSC constant $\tau_1$ is trickier.
One simple approach is to first minimize the ridge regression problem 
\begin{equation}
    \label{eq:ridge}
\mathcal P_n(\theta)=\mathcal{L}_n(\theta)+\frac{\lambda}{2}\|\theta\|_2^2.
\end{equation} 
Clearly, the objective~\eqref{eq:ridge} is $\lambda$-strongly convex, and taking $\lambda=o(1)$ guarantees that its non-private minimizer will be close to the unregularized solution $\hat\theta$. In fact, the objective~\eqref{eq:ridge} has a $\frac{L+\lambda}{\lambda}$-stable Hessian, so a sequence of noisy damped Newton iterates with $\eta <  \frac{\lambda}{\lambda+L}$ will convergence geometrically; see Proposition \ref{prop:global_SC_geom} in Appendix \ref{App:global_SC_geom} for a precise statement. Furthermore, if $\lambda=o(1)$ and $\|\nabla\mathcal{P}_n(\theta^{(k)})\|_2\leq \frac{\lambda}{L+\lambda}$, then $\|\nabla\mathcal{L}_n(\theta^{(k)})\|_2=O(\lambda)$, and hence $\|\nabla\mathcal{L}_n(\theta^{(k)})\|_2\leq \frac{\tau_1^2}{L}$, because  $\tau_1$ is bounded away from $0$. This suggests a two-step procedure, where one first runs damped noisy Newton on the ridge problem and then uses this solution as a starting value for a few pure noisy Newton steps on the unregularized problem.



\noindent{\textbf{Self-concordant case:}}
In the case of a $(\gamma,2)$-self-concordant loss function such as in Examples \ref{ExaSC} and \ref{ExaLogistic} (see Appendix~\ref{AppSCExamples}), it may also be advantageous to begin Newton's method with a fixed step size $\eta$, then switch to pure Newton ($\eta = 1$) at an iteration $k$ for which $\lambda_{\min}^{-1/2} (\nabla^2 \Loss_n(\theta^{(k)})) \lambda(\theta^{(k)}) \le \frac{1}{16 \gamma}$. To preserve differential privacy, one may use private estimates of $\lambda(\theta^{(k)})$ and $\nabla^2 \Loss_n(\theta^{(k)})$. Since our method already employs private estimates of the gradient and Hessian at each iteration, privately computing the value of the Newton decrement and the minimal value of the Hessian come at no additional cost with respect to the privacy budget. However, note that the constant $C'$ in Proposition \ref{ThmNewtonGlobal} depends on unknown data-dependent quantities. Accordingly, we propose two practical alternatives. The first is to proceed in two steps, by first minimizing the objective~\eqref{eq:ridge} until one can privately check that $\lambda_{\min}^{-1/2} (\nabla^2 \Loss_n(\theta^{(k)})) \lambda(\theta^{(k)}) \le \frac{1}{16 \gamma}$. This condition will eventually be met, since by Proposition \ref{prop:global_SC_geom}, the sequence $\{\theta^{(k)}\}$ obtained from the ridge problem will approach geometrically fast to values such that $\|\nabla\mathcal{L}_n(\theta^{(k)})\|=O(\lambda)$, and then switch to a pure noisy Newton step. A second alternative is to begin with an \emph{adaptive} noisy damped Newton phase, before switching to the the pure noisy Newton phase. More specifically, one can take private data-dependent steps of size $\eta_k=\frac{1}{\beta_k}\log(1+\beta_k)$, where $\beta_k=\gamma\|\{\nabla^2\mathcal{L}_n(\theta^{(k-1)})\}^{-1}\nabla\mathcal{L}_n(\theta^{(k-1)})\|$. This is motivated by Theorem 2 in \cite{sun2019generalized}. A rigorous analysis of such an algorithm is possible using similar calculations as in this paper, but is omitted from the present paper for brevity.

In Figure \ref{fig:self_concordant_comparison}, we use noisy gradient descent and noisy damped Newton's method, together with the $(2/c,3)$-self-concordant Huber loss from Example \ref{ExaSC}, and choose an initialization far from the true parameter value ($\beta^{(0)} = 0$ vs.\ the true $\beta=(10,-11,9,-11)^{T}$) in the linear model of Section~\ref{SecLinReg}. Figure \ref{fig:self_concordant_comparison} illustrates two benefits of implementing an adaptive damped noisy Newton algorithm relative to noisy gradient descent: First, when the algorithms are initialized outside the local strong convexity region, Newton's method converges linearly to this region. This is in stark contrast with noisy gradient descent, which can only approach the local strong convexity region at a sub-linear rate. This slow rate of convergence is inherited from the rate of convergence of gradient descent \emph{in the absence of strong convexity} \citep{bubeck2015,nesterov2018}. Second, once the noisy Newton iterates are close to the global optimum, one can hope to detect this transition and obtain improved quadratic convergence by taking private pure Newton steps.
\begin{figure}[h]
    \centering
    \includegraphics[width=8cm]{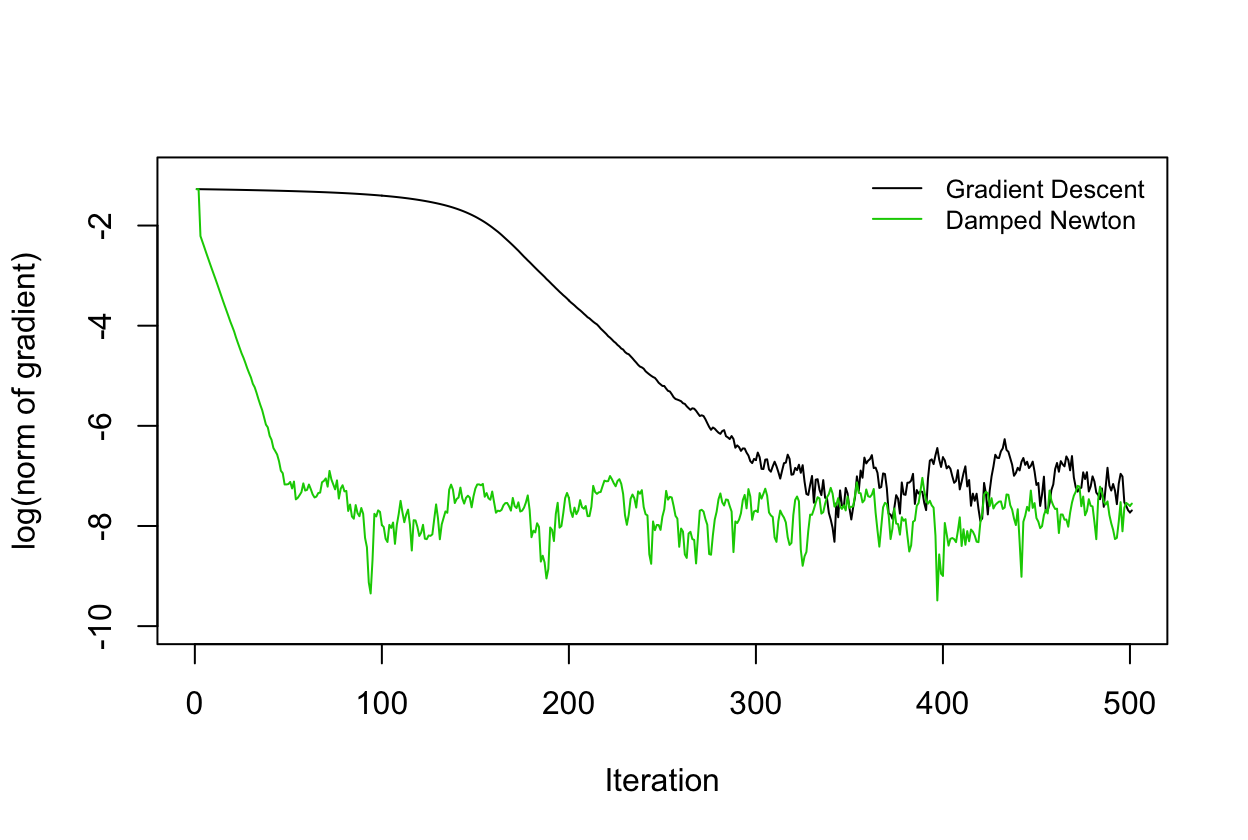}
    \caption{Trajectories of gradient descent and damped Newton algorithm. The damped Newton algorithm converges much faster in relation to gradient descent.}
    \label{fig:self_concordant_comparison}
\end{figure}

\section{Inference via noisy variance estimators}
\label{SecInference}

We now show how to leverage our results on asymptotic normality of the private estimators, obtained from either noisy gradient descent or noisy Newton's method, to construct confidence regions for $\theta_0$.

\subsection{Private sandwich formula}

The most common approach for constructing confidence intervals using an M-estimator $\hat\theta$ is to use an asymptotic pivot. In particular, under regularity conditions, we have $\sqrt{n} (\hat\theta- \theta_0) \to_d N(0 , V (\theta_0))$, where 
%
$V (\theta_0)  := M (\theta_0)^{-1} Q (\theta_0) M (\theta_0)^{-1}$, $M (\theta_0) := -\E_F[ \dot\Psi (X , \theta_0)]$, and $Q (\theta_0) := \E_F[ \Psi (X , \theta_0) \Psi (X, \theta_0)^\top]$ are the analogs of equation~\eqref{EqnSandwich}. The asymptotic variance can accordingly be estimated via the plug-in sandwich estimator 
$V_n(\hat\theta)=  M_n (\hat\theta)^{-1} Q_n (\hat\theta) M_n (\hat\theta)^{-1}$, where $M_n(\hat\theta)=\frac{1}{n}\sum_{i=1}^n\dot\Psi(X_i,\hat\theta)$ and $Q_n(\hat\theta)=\frac{1}{n}\sum_{i=1}^n\Psi(X_i,\hat\theta)\Psi(Z_i,\hat\theta)^\top$. Consistency of $\hat\theta$ and the continuous mapping theorem show that $V_n(\hat\theta)\to_p V(\theta_0)$, so  Slutsky's theorem justifies the use of the $(1-\alpha)$-confidence intervals $\mbox{CI}_{1-\alpha}(\theta_{0j})=\hat\theta_j \pm \frac{\sqrt{V_n(\hat\theta)_{jj}}}{\sqrt{n}}z_{1-\alpha/2}$,
where $z_{1-\alpha/2}$ is the $(1-\alpha/2)$-quantile of a standard normal. However, in the  differential privacy setting, this plug-in construction cannot be applied directly, as neither $M_n(\theta^{(K)})$ nor $Q_n(\theta^{(K)})$ are differentially private. Instead, we use the approach discussed in Section~\ref{SecMatrixNoise} to construct differentially private analogues of these two matrices: Assuming Condition \ref{ass:Hessian} holds, Lemma \ref{lem:GaussianMatrix} suggests using the estimates $\tilde{M}_n(\theta^{(K)})={M}_n(\theta^{(K)}) + \frac{2\bar{B}}{\mu n}G_1$ and $\tilde{Q}_n(\theta^{(K)})={Q}_n(\theta^{(K)}) + \frac{2B^2}{\mu n}G_2$, where $G_1$ and $G_2$ are i.i.d.\ symmetric random matrices whose upper-triangular elements, including the diagonals, are i.i.d.\ standard normal. The differentially private matrices  $\tilde{M}_n(\theta^{(K)})$ and $ \tilde{Q}_n(\theta^{(K)})$ can be easily projected onto
a cone of positive definite matrices $\{H: H\succeq \varepsilon I\}$ without paying a large statistical cost (cf.\ Remark~\ref{RemProjection}). Equipped with the projected matrices 
$\tilde{M}_n(\theta^{(K)})_+=\arg\min_{H \succeq \varepsilon I}\|H-\tilde{M}_n(\theta^{(K)})\|_2$ and $\tilde{Q}_n(\theta^{(K)})_+=\arg\min_{H \succeq \varepsilon I}\|H-\tilde{Q}_n(\theta^{(K)})\|_2$, we can construct the differentially private sandwich estimator
\begin{equation}
    \label{eq:sandwich}
    \tilde V_n(\theta^{(K)})= \tilde M_n (\theta^{(K)})_+^{-1} \tilde Q_n (\theta^{(K)})_+ \tilde M_n (\theta^{(K)})_+^{-1}.
\end{equation}
Essentially the same idea was suggested by \cite{wangetal2019} in the context of differentially private estimators computed via objective perturbation and output perturbation, but those techniques suffer some of the drawbacks mentioned in the introduction (e.g., assuming bounded data).

The composition property of $\mu$-GDP estimators and consistency imply the following:
\begin{proposition}
\label{prop:sandwich_consistency}
Assume the conditions of Theorems~\ref{thm:NGD} or \ref{thm:NNewton} hold. Then $ \tilde V_n(\theta^{(K)})$ is $\sqrt{3}\mu$-GDP and $\tilde V_n(\theta^{(K)})\to_p V(\theta_0)$.
\end{proposition}
We can therefore release the $\sqrt{3}\mu$-GDP intervals $\mbox{CI}_{1-\alpha}^{priv}(\theta_{0j})= \theta_j^{(K)} \pm \frac{\sqrt{\tilde V_n(\theta^{(K)})_{jj}}}{\sqrt{n}}z_{1-\alpha/2}$.
%
%
Although this construction is asymptotically valid, it tends to be too liberal in small samples. The correction proposed in the next subsection partially addresses this issue.


\subsection{A finite-sample correction}

We now discuss a correction to the noisy gradient descent and noisy Newton's method algorithms which leads to better performance in practice.

\subsubsection{Noisy gradient descent}

 The  formula     \eqref{eq:sandwich},  by construction,  underestimates the variance of $\theta^{(K)}$ in finite samples, as it fails to account for the additional variability introduced by the privacy-preserving random noise mechanism. In the case of noisy gradient descent, this can be mitigated by making the following correction to
equation~\eqref{eq:sandwich}:
\begin{equation}
\label{eq:var_cor_NGD}
\hat{V}_n(\theta^{(K)}) = \tilde V_n(\theta^{(K)}) + \frac{8\eta^2B^{2}K}{n \mu^{2}} I. 
\end{equation}
The correction \eqref{eq:var_cor_NGD} is motivated by the behavior of the noisy iterates of our algorithms at convergence. From the proof of Theorem \ref{thm:NGD}, with high probability, every iteration of the algorithm grows closer to the non-private solution $\hat\theta$, up to a privacy-preserving noisy neighborhood of radius $\frac{4\eta^2B^2K}{n\mu^2}$. Once this neighborhood is attained, the iterates approximately behave as a random walk around the boundary. Since we inject Gaussian noise with variance $\frac{4\eta^2B^2K}{n\mu^2}$, we see that $\theta^{(K)}$ will be approximately within a $\frac{8\eta^2B^2K}{n\mu^2}$-neighborhood of $\hat\theta$, which is accounted for in equation~\eqref{eq:var_cor_NGD}. This corrected variance formula yields intervals $ \mbox{CI}^{cor}_{1-\alpha}(\theta_{0j})=\theta_j^{(K)}\pm \sqrt{\frac{\tilde{V_n}(\theta^{(K)})_{jj}}{n}+2\left( \frac{2\eta B \sqrt{K}}{n \mu } \right)^{2}} z_{1-\alpha/2}$. 
%
We demonstrate the practical impact of this correction in Figure~\ref{FigCI} below, 
using noisy gradient descent on simulated data to estimate the parameters in multivariate linear regression. Data were drawn from the same model as in Section~\ref{SecLinReg}, and optimization began at the same initial point ($\beta^{(0)} = 0$ and $\sigma^{(0)} = 1$).
We see that the empirical coverage of the 95\% confidence intervals (for one particular regression coefficient) are indeed closer to the nominal 95\% level with the correction; the effect is more pronounced for smaller sample sizes. Naturally, one can also construct corrected confidence regions  given ellipsoids centered at $\theta^{(K)}$ and shaped according to equation~\eqref{eq:var_cor_NGD}. We illustrate this construction in Figure \ref{FigCI}, where data were generated from a variation of the linear model from Section~\ref{SecLinReg}, in which the covariates are correlated rather than independent.

\begin{figure}[h]
\begin{center}
\begin{tabular}{cc}
    \includegraphics[width=5.5cm]{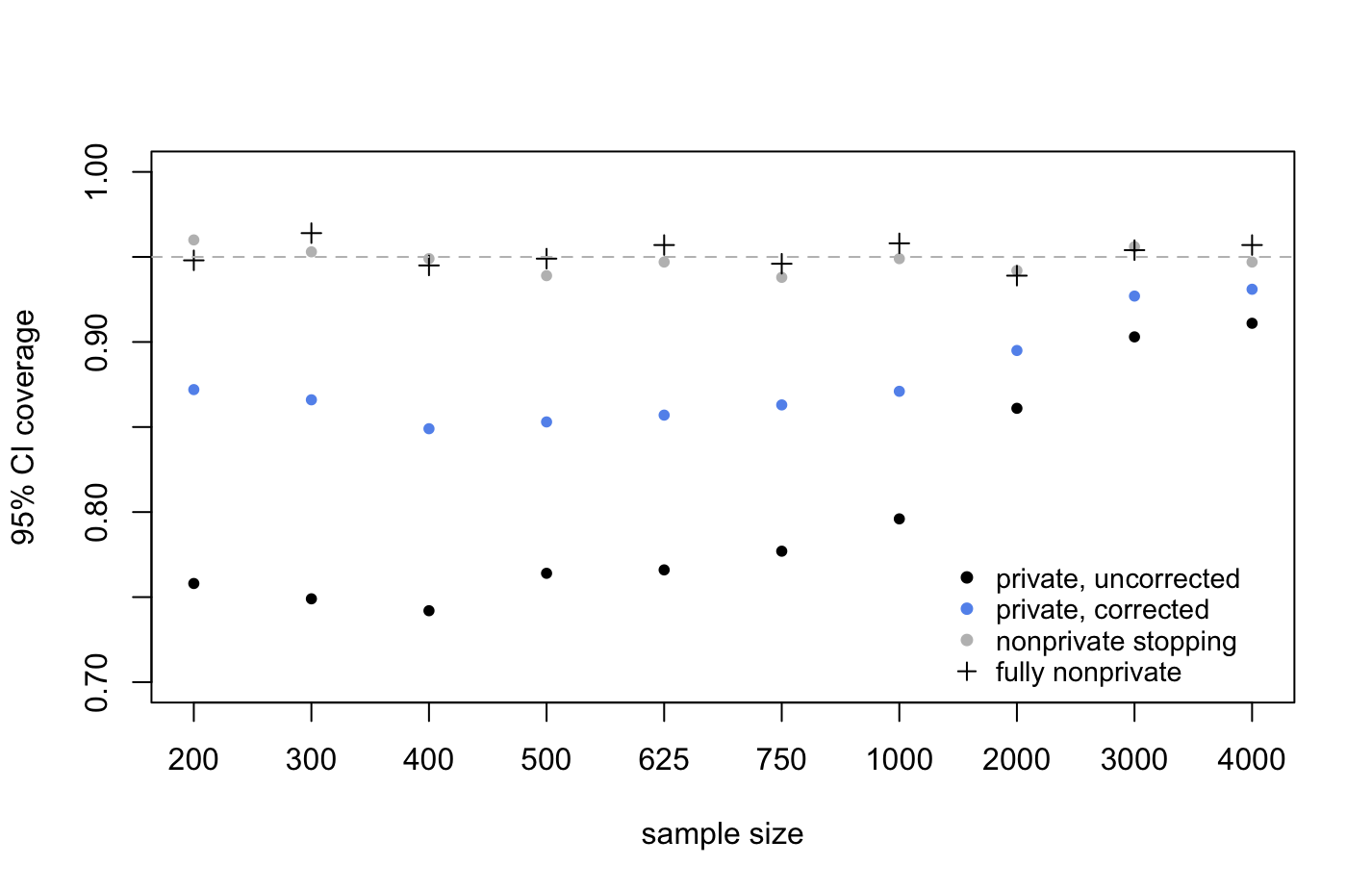} & \includegraphics[width=10cm]{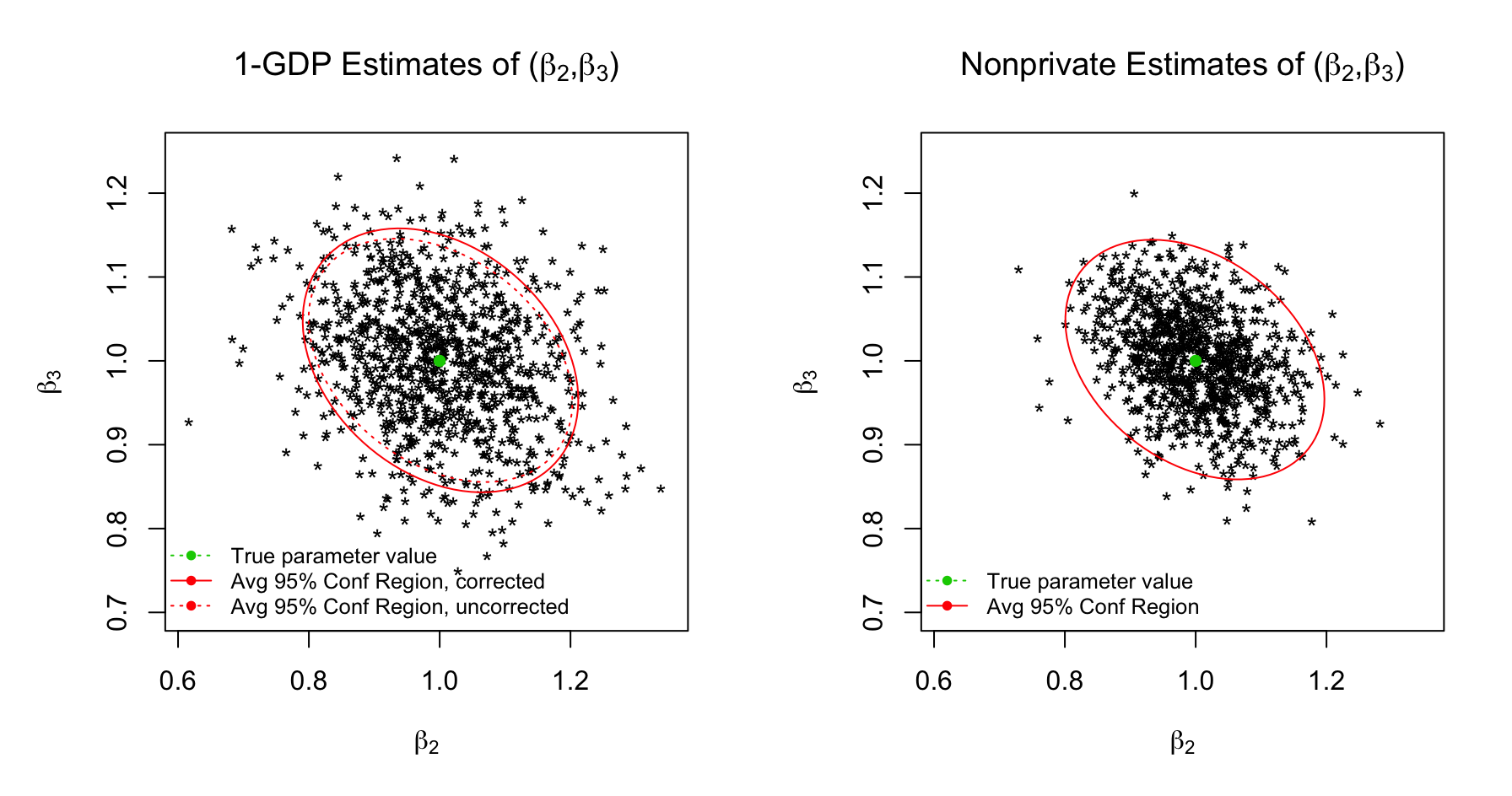} \\
    (a) & (b)
\end{tabular}
    \caption{(a) Empirical confidence interval coverage with noisy gradient descent estimators. The blue dots show the empirical coverage probabilities of the last iterate utilizing the noisy sandwich formula and the corrected formula. The gray points illustrate the coverage of the corrected formula computed using the first iterate whose gradient is smaller (in Euclidean norm) than the standard error of the noise added at each iteration. This gradient uses a non-private stopping rule. (b) 2D coverage region for a pair of parameters, showing the tighter confidence region resulting from our finite-sample correction to the sandwich formula. Nonprivate counterpart included on right.}
\label{FigCI}
\end{center}
\end{figure}
\vspace{-0.5cm}




\subsubsection{Noisy Newton's method}

 In Appendix \ref{AppThmNewton}, we see that the noisy Newton's method update can be expressed as the sum of an ordinary Newton's method update plus a noise term. Specifically, the noise term is $\eta \tilde{N}_{k}$, where  $\tilde{N}_{k}$ is an infinite sum given by equation~\eqref{EqnNoiseNewton}. We wish to account for the additional variance introduced by this noise term. We propose to retain only the first term in $\tilde{N}_{k}$, as the other terms are higher-order terms. Thus, truncating the error term to $\eta \left\{\nabla^2\mathcal{L}_n(\theta^{(k)})\right\}^{-1}\tilde{Z}_k$, the variance attributable to this term is $\eta^{2}\left\{\nabla^2\mathcal{L}_n(\theta^{(k)})\right\}^{-1} \mbox{var}(\tilde{Z}_k) \left\{\nabla^2\mathcal{L}_n(\theta^{(k)})\right\}^{-1}$. Finally, to maintain the desired privacy guarantee, we substitute the noisy estimate of the Hessian at the current iterate for $\nabla^2\mathcal{L}_n(\theta^{(k)})$.
This gives an approximate correction to the variance of the form $C_{Newton} \coloneqq \eta^{2} \left\{\nabla^2\mathcal{L}_n(\theta^{(k)})+\tilde{W}_k\right\}^{-1} \left ( \frac{2B\sqrt{2K}}{\mu n} \right )^{2} \ \left\{\nabla^2\mathcal{L}_n(\theta^{(k)})+\tilde{W}_k\right\}^{-1}$.
In the case of noisy Newton's method, the variance of $\theta^{(K)}$ can be approximated by making the following correction to equation~\eqref{eq:sandwich}: $\hat{V}_n(\theta^{(K)}) = \tilde V_n(\theta^{(K)}) + nC_{Newton}$.
Analogously to the gradient descent case above, this corrected variance formula yields $(1-\alpha)$-confidence intervals of the form $ \mbox{CI}^{cor}_{1-\alpha}(\theta_{0j})=\theta_j^{(K)}\pm \sqrt{\frac{\tilde{V}_n(\theta^{(K)})_{jj}}{n}+\left( C_{Newton} \right)_{jj} } z_{1-\alpha/2}$.
Figure \ref{fig:newton correction} demonstrates the practical impact of this correction. Simulated data for this exercise was generated from the same model as in Figure \ref{FigCI}. All optimizations were initialized at $\beta^{(0)} = 0$, with $\sigma^{2}$ assumed to be known.

\begin{figure}[h]
    \centering
    \includegraphics[width=8cm]{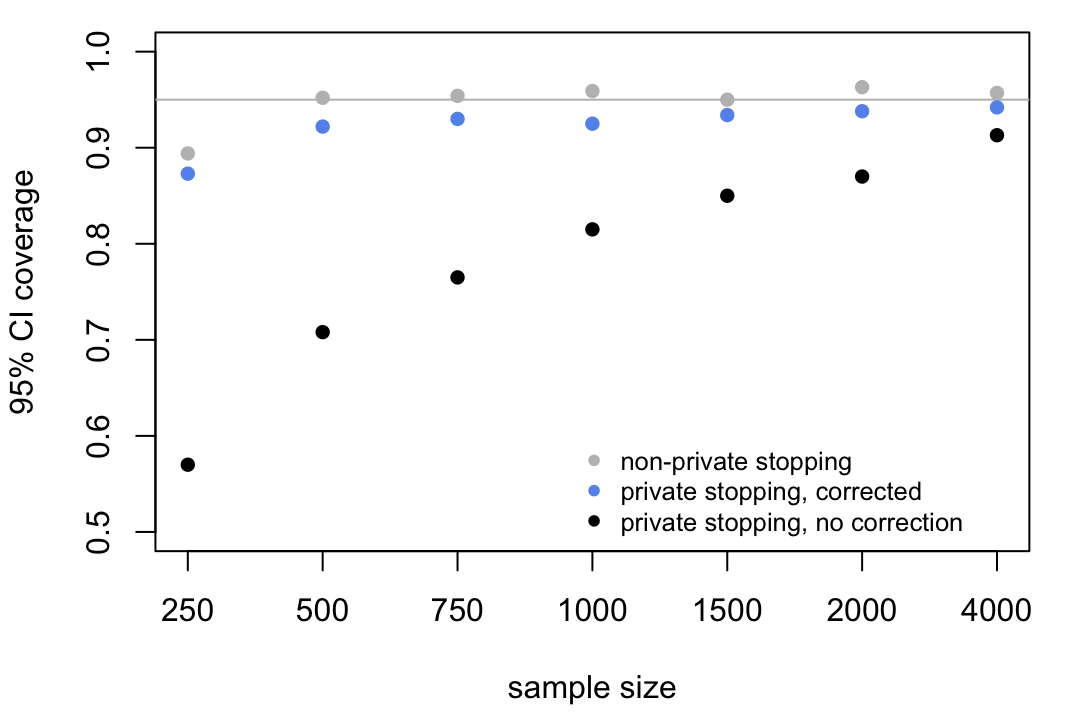}
    \caption{Empirical confidence interval coverage for confidence intervals constructed using noisy Newton's method estimators. The blue dots show the empirical coverage probabilities of the last iterate utilizing the noisy sandwich formula and the corrected formula, for an algorithm that terminates before the budgeted number of iterations if the noisy gradient is sufficiently small in magnitude. The gray dots illustrate the coverage of the corrected formula with a stopping rule based on the true gradient.}
    \label{fig:newton correction}
\end{figure}




\section{Discussion}
\label{SecDiscussion}

We have studied theoretical properties of noisy versions of gradient descent and Newton's method for computing differentially private M-estimators. We have analyzed the statistical properties of the iterates of these algorithms and provided a general approach for computing differentially private confidence regions based on asymptotic pivots.

Our theory shows that the convergence rates of our private M-estimators are nearly optimal, and many well-known convex optimization results are naturally generalized by our framework. In particular, the noisy optimizers we consider preserve the same rates of convergence as their standard non-noisy counterparts, since they converge in the same order of iterations to a neighborhood of the solution to the non-private objective function minimization problem. Our noisy algorithms exhibit several noteworthy distinct features. By construction, the iterates cannot approach the non-private M-estimators beyond a privacy-preserving neighborhood proportional to the noise injected at each step of the algorithms. They instead bounce around this neighborhood once it is attained. Furthermore, the numbers of iterations has to be scheduled in advance and directly impacts the noise added to each step of the algorithms; the more steps, the more data queries, and hence the larger the required privacy-inducing noise. 

Our analysis highlights the statistical importance of (local) strong convexity, since it justifies scheduling only $O(\log n)$ noisy iterates, which in turn leads to the nearly-optimal statistical cost of privacy of the order $O\left(\frac{ p \log n}{n\mu^2}\right)$. Without strong convexity, standard gradient-based algorithms are known to require $O(\sqrt{n})$ iterations to achieve an optimization error comparable to the statistical parametric rate $O\left(\sqrt{\frac{p}{n}}\right)$. This would be quite damaging for differentially private M-estimators, as they would lead to a statistical cost of $O\left(\sqrt{\frac{p}{n}}\frac{1}{\mu^2}\right)$,  which would in fact be the dominating term driving the asymptotic efficiency of the resulting estimators.

We have also introduced a general technique for computing confidence regions based on noisy estimates of the asymptotic variance of the M-estimator. Since any noisy iterate has additional built-in noise that is not captured by the usual asymptotic variance, we have proposed  simple corrections that account for an extra noise term. This approach significantly mitigates the systematic underestimation of the variance from the usual sandwich formula.

Two natural extensions of our work would be to study noisy stochastic gradient descent algorithms and  consider high-dimensional penalized M-estimators with noisy proximal methods. While these ideas have been explored in the literature, previous analyses share many of the limitations of the existing literature of noisy gradient descent, i.e., they rely on (restricted) strong convexity, assume bounded data and/or bounded parameter spaces, or employ truncation ideas~ \citep{jainandthakurta2014,talwaretal2015, caietal2019,caietal2020, bassilyetal2014,wangetal2017,dongetal2021}. We believe that some of the techniques used in this paper may bypass the aforementioned drawbacks and could also be used in conjunction with restricted local strong convexity and general penalty functions, as in \cite{lohandwainwright2015} and \cite{loh2017}. 

Finally, we do not address local DP  \cite{warner1965,evfimievskietal2003, kasiviswanathanetl2011,duchietal2018}, where one also has to filter the data when it is collected, in the absence of a trusted curator. Noisy gradient descent has been one of the general algorithms proposed in that setting, and it would be interesting to explore whether our techniques can lead to better understanding of the optimization problems there.

\section*{Acknowledgments}

The authors thank the associate editor and referees for suggestions which improved the quality of this manuscript.


\appendix

\section{Concentration inequalities}

The following two lemmas will be instrumental in the analysis of our noisy algorithm.  

\begin{lemma}\label{lem:maxGaussian}
Let $X\in\mathbb R^d$ be a sub-Gaussian random vector with variance proxy $\sigma^2$. For any $\alpha>0$, with probability at least $1-\alpha$,
$$\|X\|_2 \leq 4\sigma\sqrt{d}+2\sigma\sqrt{2\log(1/\alpha)}.$$
\end{lemma}

\begin{proof}
This result can be found in \cite[Theorem 1.19]{rigolletandhutter2017}.
\end{proof}

\begin{lemma}\label{lem:norm_Gaussian_matrix}
Let $W$ be a symmetric $p\times p$ random matrix whose upper triangular elements, including the diagonal, are i.i.d.\ $N(0,1)$. For any $\alpha>0$, with probability at least $1-\alpha$,
$$\|W\|_2 \leq \sqrt{2p\log(2p/\alpha)}.$$
\end{lemma}

\begin{proof}
Letting $E_{jk}$ denote the matrix which the value $1$ in the $(j,k)^{\text{th}}$ component and $0$ everywhere else, we see that 
\begin{equation*}
W=\sum_{1\leq j\leq k\leq p}Z_{jk}(E_{jk}+E_{kj}-\mathbbm{1}_{\{j=k\}}E_{jj}),
\end{equation*}
where $\{Z_{jk}\}_{1 \le j,k \le p}$ is an i.i.d.\ sequence of standard normal variables. It follows from \cite[Theorem 4.1.1]{tropp2015} that with probability at least $1-\alpha$,
\begin{equation*}
     \|W\|_2 \leq \sqrt{2 v(Z)\log(2p/\alpha)},
 \end{equation*}
 where 
 \begin{equation*}
     v(Z)=\left\|\sum_{1\leq j\leq k\leq p}(E_{jk}+E_{kj}-\mathbbm{1}_{\{j=k\}}E_{jj})^2\right\|_2=\|p I\|_2=p.
 \end{equation*}
 \end{proof}
 
 
\section{Convex analysis}
\label{AppConvex}

The following results are standard, and proofs can be found in \cite{boydandvanderberghe2004} or \cite{bubeck2015}.

\begin{lemma}
Suppose $f: \real^p \rightarrow \real$ is strongly convex with parameter $\tau_1$, i.e.,
\begin{equation*}
f(y) \ge f(x) + \langle \nabla f(x), y - x \rangle + \tau_1 \|x-y\|_2^2, \qquad \forall x, y \in \real^p.
\end{equation*}
Then
\begin{equation*}
\langle \nabla f(x) - \nabla f(y), x - y \rangle \ge 2\tau_1 \|x-y\|_2^2, \qquad \forall x, y \in \real^p,
\end{equation*}
and if $f$ is twice-differentiable, then
\begin{equation*}
\nabla^2 f(x) \succeq 2\tau_1 I, \qquad \forall x \in \real^p.
\end{equation*}
\end{lemma}

\begin{lemma}
Suppose $f: \real^p \rightarrow \real$ is convex and $\tau_2$-smooth, i.e.,
\begin{equation*}
f(y) \le f(x) + \langle \nabla f(x), y - x \rangle + \tau_2 \|x-y\|_2^2, \qquad \forall x, y \in \real^p.
\end{equation*}
Then
\begin{equation*}
\frac{1}{2\tau_2} \|\nabla f(x) - \nabla f(y)\|_2^2 \le \langle \nabla f(x) - \nabla f(y), x - y \rangle, \qquad \forall x, y \in \real^p,
\end{equation*}
and
\begin{equation*}
\|\nabla f(x) - \nabla f(y)\|_2 \le 2\tau_2\|x-y\|_2, \qquad \forall x, y \in \real^p.
\end{equation*}
If $f$ is twice-differentiable, then
\begin{equation*}
\nabla^2 f(x) \preceq 2\tau_2 I, \qquad \forall x \in \real^p.
\end{equation*}
\end{lemma}


\section{Results about self-concordant functions}
\label{AppSC}

In this appendix, we state several useful results about generalized self-concordant functions.

\subsection{Generalized self-concordant functions}
\label{AppSCLemma}

The following results are taken from \cite{sun2019generalized}.

\begin{lemma}
[Proposition 1 from \cite{sun2019generalized}]
\label{LemSCSum}
Suppose $f_i$ are $(\gamma_i, \nu)$-self-concordant functions for $1 \le i \le n$, where $\gamma_i \ge 0$ and $\nu \ge 2$. The function $f(x) := \sum_{i=1}^n \beta_i f_i(x)$, for $\beta_i > 0$, is $(\gamma, \nu)$-self-concordant with
\begin{equation*}
\gamma = \max_{1 \le i \le n}\left\{\beta_i^{1-\nu/2} \gamma_i\right\}.
\end{equation*}
\end{lemma}

\begin{lemma}
[Proposition 2 from \cite{sun2019generalized}]
\label{LemSCAffine}
Suppose $f: \real^p \rightarrow \real$ is $(\gamma, \nu)$-self-concordant with $\nu \in (0, 3]$, and consider any affine function $x \mapsto Ax+b$. Then $g(x) = f(Ax+b)$ is $(\gamma \|A\|_2^{3-\nu}, \nu)$-self-concordant.
\end{lemma}

\begin{lemma}
[Proposition 4 from \cite{sun2019generalized}]
\label{LemSC2}
Suppose $f: \real^p \rightarrow \real$ is $(\gamma, \nu)$-self-concordant with $\nu \ge 2$, and $\nabla f$ is Lipschitz continuous with Lipschitz constant $L \ge 0$ in $\ell_2$-norm. Then $f$ is $(L^{\nu/2 - 1} \gamma, 2)$-self-concordant.
\end{lemma}

\begin{lemma}
[Proposition 8 from \cite{sun2019generalized}]
\label{LemSCHess}
Suppose $f: \real^p \rightarrow \real$ is $(\gamma, 2)$-self-concordant. For any $x, y \in \real^p$, we have
\begin{equation*}
\exp\left(- \gamma \|x-y\|_2\right) \nabla^2 f(x) \preceq \nabla^2 f(y) \preceq \exp\left(\gamma \|x-y\|_2\right) \nabla^2 f(x).
\end{equation*}
\end{lemma}

\begin{lemma}
[Lemma 2 of \cite{sun2019generalized}]
\label{LemSCHmat}
Suppose $f: \real^p \rightarrow \real$ is $(\gamma, 2)$-self-concordant. For $x, y \in \real^p$, the matrix $H(x,y)$ defined by
\begin{equation*}
H(x,y) := \nabla^2 f(x)^{-1/2} \left(\int_0^1 \left(\nabla^2 f(x+t(y-x)) - \nabla^2 f(x)\right) dt\right) \nabla^2 f(x)^{-1/2}
\end{equation*}
satisfies
\begin{equation*}
\|H(x,y)\|_2 \le \left(\frac{3}{2} + \frac{\gamma \|x-y\|_2}{3}\right)\gamma \|x-y\|_2\exp\left(\gamma \|x-y\|_2\right) .
\end{equation*}
\end{lemma}

\begin{lemma}
[Proposition 9 of \cite{sun2019generalized}]
\label{LemSCgrad}
Suppose $f: \real^p \rightarrow \real$ is $(\gamma, 2)$-self-concordant. For any $x, y \in \real^p$, we have
\begin{align*}
\frac{1-\exp(-\gamma \|y-x\|_2)}{\gamma \|y-x\|^2_2} \|y - x\|^2_{\nabla^2 f(y)} & \le \langle \nabla f(y) - \nabla f(x), y - x \rangle.
\end{align*}
Furthermore, by the Cauchy-Schwarz inequality, the right-hand expression is upper-bounded by $\|\nabla f(y) - \nabla f(x)\|_{\nabla^2 f(y)^{-1}} \|y-x\|_{\nabla^2 f(y)}$, so
\begin{align*}
\frac{1-\exp(-\gamma \|y-x\|_2)}{\gamma \|y-x\|^2_2} \|y - x\|_{\nabla^2 f(y)} & \le \|\nabla f(y) - \nabla f(x)\|_{\nabla^2 f(y)^{-1}}.
\end{align*}
\end{lemma}

\begin{lemma}
[Proposition 10 of \cite{sun2019generalized}]
\label{LemSCfunc}
Suppose $f: \real^p \rightarrow \real$ is $(\gamma, 2)$-self-concordant. For any $x, y \in \real^p$, we have
\begin{equation*}
\omega(-\gamma \|x-y\|_2) \cdot \|x-y\|^2_{\nabla^2 f(x)} \le f(y) - f(x) - \langle \nabla f(x), y-x \rangle \le \omega(\gamma \|x-y\|_2) \cdot \|x-y\|^2_{\nabla^2 f(x)},
\end{equation*}
where $\omega(t) := \frac{\exp(t) - t - 1}{t^2}$.
\end{lemma}

We now provide the proofs of Lemmas~\ref{LemMallows} and~\ref{LemSchweppe}, stated in the main text.

\begin{proof} [Proof of Lemma~\ref{LemMallows}]
Note that Lemma~\ref{LemSCAffine} implies that $\rho(y_i, x_i^T \theta)$ is $(\gamma \|x_i\|_2, 2)$-self-concordant. Then Lemma~\ref{LemSCSum} implies that $\Loss_n$ is $(\gamma \max_i \|x_i\|_2, 2)$-self-concordant.
\end{proof}

\begin{proof} [Proof of Lemma~\ref{LemSchweppe}]
For (i), note that Lemma~\ref{LemSCAffine} implies that $\rho((y_i - x_i^T \theta)v(x_i))$ is $(\gamma \|v(x_i) x_i\|_2, 2)$-self-concordant. Then Lemma~\ref{LemSCSum} implies that $\Loss_n(\theta)$ is $(\gamma \max_i \|v(x_i) x_i\|_2, 2)$-self-concordant.

For (ii), note that Lemma~\ref{LemSCAffine} implies that $\rho((y_i-x_i^T \theta)v(x_i))$ is $(\gamma, 3)$-self-concordant. Note that by assumption, each function $\rho'((y_i - x_i^T \theta) v(x_i)) v(x_i) x_i$ is $L$-Lipschitz continuous with $L = C^2 \|\rho''\|_\infty$. Hence, Lemma~\ref{LemSC2} then implies that $\rho((y_i-x_i^T \theta)v(x_i))$ is also $(\gamma L^{1/2}, 2)$-self-concordant. Hence, Lemma~\ref{LemSCSum} implies that $\Loss_n(\theta)$ is $(\gamma L^{1/2}, 2)$-self-concordant.
\end{proof}


\subsection{Stability of $(\gamma, 2)$-self-concordant functions}
\label{AppStable}

The following Hessian stability condition was introduced in \cite{karimireddyetal2018}, where it was used to prove global convergence of the iterates in (non-noisy) Newton's method for generalized self-concordant functions.

\begin{condition}
[Hessian stability]
\label{ass:stable_Hessian}
Suppose $\Theta_0 \subseteq \real^p$. For any $\theta_1,\theta_2\in\Theta_0$ and $\theta_1\neq\theta_2$, assume $\|\theta_1-\theta_2\|_{\nabla^2\mathcal{L}_n(\theta_1)}>0$ and there exists a constant $\tau_0\geq 1$ such that
$$\tau_0 := \sup_{\theta_1,\theta_2\in\Theta_0}\frac{\|\theta_1-\theta_2\|_{\nabla^2\mathcal{L}_n(\theta_2)}^2}{\|\theta_1-\theta_2\|_{\nabla^2\mathcal{L}_n(\theta_1)}^2} < \infty.$$
\end{condition}
In practice, we will take $\Theta_0=\mathcal{B}_{R}(\thetahat)$, where $R$ is a suitably chosen radius such that $\{\theta^{(k)}\} \subseteq \Theta_0$.
Condition \ref{ass:stable_Hessian} is a stability assumption that allows us to show global convergence results without assuming local strong convexity. 
An important consequence is that it implies upper and lower bounds that are similar to local strong convexity and smoothness, as we can see in the following lemma:

\begin{lemma}
\label{LemA}
Given Condition \ref{ass:stable_Hessian}, for any $\theta_1, \theta_2\in\Theta_0$, we have the following upper and lower bounds:
\begin{align}
    \mathcal{L}_n(\theta_1)&\leq \mathcal{L}_n(\theta_2)+\langle \nabla\mathcal{L}_n(\theta_2),\theta_1-\theta_2\rangle+\frac{\tau_0}{2}\|\theta_1-\theta_2\|_{\nabla^2\mathcal{L}_n(\theta_2)}^2, \label{UBstable}\\
    \mathcal{L}_n(\theta_1)&\geq \mathcal{L}_n(\theta_2)+\langle \nabla\mathcal{L}_n(\theta_2),\theta_1-\theta_2\rangle+\frac{1}{2\tau_0}\|\theta_1-\theta_2\|_{\nabla^2\mathcal{L}_n(\theta_2)}^2. \label{LBstable}
\end{align}
\end{lemma}

\begin{proof}
We note that this corresponds to Lemma 2 in \cite{karimireddyetal2018}, but their arguments have a mistake which we have corrected here. A second-order Taylor expansion shows that 
\begin{equation}
    \mathcal{L}_n(\theta_1)=\mathcal{L}_n(\theta_1)+\langle \nabla \mathcal{L}_n(\theta_2),\theta_1-\theta_2\rangle +\frac{1}{2}\|\theta_1-\theta_2\|^2_{\int_0^1\nabla^2\mathcal{L}_n(t\theta_1+(1-t)\theta_2)\mathrm{d}t}. 
    \label{eq:Lem2_1stable}
\end{equation}
Let $\bar{\theta}_t=t\theta_1+(1-t)\theta_2$, and note that by convexity of $\mathcal{L}_n$, we have $\bar{\theta}_t\in\Theta_0$, for all $t\in[0,1]$. Therefore, Condition~\ref{ass:stable_Hessian} ensures that 
\begin{equation*}
    t^2 \|\theta_1-\theta_2\|^2_{\nabla^2\mathcal{L}_n(\bar{\theta}_t)}=\|\bar{\theta}_t-\theta_2\|^2_{\nabla^2\mathcal{L}_n(\bar{\theta}_t)} \le \tau_0\|\bar{\theta}_t-\theta_2\|^2_{\nabla^2\mathcal{L}_n(\theta_2)} = \tau_0t^2\|\theta_1-\theta_2\|^2_{\nabla^2\mathcal{L}_n(\theta_2)},
\end{equation*}
so
\begin{equation}
    \label{eq:Lem2_2stable}
   \|\theta_1-\theta_2\|^2_{\int_0^1\nabla^2\mathcal{L}_n(t\theta_1+(1-t)\theta_2)\mathrm{d}t}\leq \sup_{t\in[0,1]}\|\theta_1-\theta_2\|^2_{\nabla^2\mathcal{L}_n(\bar{\theta}_t)}\leq \tau_0\|\theta_1-\theta_2\|^2_{\nabla^2\mathcal{L}_n(\theta_2)}.
\end{equation}
Combining inequalities~\eqref{eq:Lem2_1stable} and~\eqref{eq:Lem2_2stable} yield the desired upper bound \eqref{UBstable}. An analogous argument establishes the lower bound. Indeed, Condition \ref{ass:stable_Hessian} also ensures that
\begin{equation*}
\tau_0t^2\|\theta_1-\theta_2\|^2_{\nabla^2\mathcal{L}_n(\bar{\theta}_t)}=\tau_0\|\bar\theta_t-\theta_2\|^2_{\nabla^2\mathcal{L}_n(\bar{\theta}_t)} \geq \|\bar{\theta}_t-\theta_2\|^2_{\nabla^2\mathcal{L}_n(\theta_2)} = t^2\|\theta_1-\theta_2\|^2_{\nabla^2\mathcal{L}_n(\theta_2)},
\end{equation*}
so
\begin{equation}
            \label{eq:Lem2_3stable}
\|\theta_1-\theta_2\|^2_{\int_0^1\nabla^2\mathcal{L}_n(t\theta_1+(1-t)\theta_2)\mathrm{d}t}\geq \inf_{t\in[0,1]}\|\theta_1-\theta_2\|^2_{\nabla^2\mathcal{L}_n(\bar{\theta}_t)}\geq \frac{1}{\tau_0}\|\theta_1-\theta_2\|^2_{\nabla^2\mathcal{L}_n(\theta_2)}.
\end{equation}
Combining inequalities~\eqref{eq:Lem2_1stable} and~\eqref{eq:Lem2_3stable} proves the lower bound \eqref{LBstable}.
\end{proof}

We now show that Condition~\ref{ass:stable_Hessian} is implied by self-concordance:

\begin{lemma}
\label{LemSCMult}
Suppose $\Loss_n: \real^p \rightarrow \real$ is $(\gamma, 2)$-self-concordant.
\begin{itemize}
\item[(i)] For any $r > 0$, we have
\begin{equation*}
\tau_{1,r} := \inf_{\theta \in \mathcal{B}_r(\theta_0)} \lambda_{\min}(\nabla^2 \Loss_n(\theta)) \ge \exp(-\gamma r) \lambda_{\min}(\nabla^2 \Loss_n(\theta_0)).
\end{equation*}
\item[(ii)] Suppose $\nabla^2 \Loss_n(\theta_0) \succ 0$ for some $\theta_0 \in \Theta_0$, and in addition, $\mathrm{diam}(\Theta_0) \le D$. Then $\mathcal{L}_n$ satisfies Condition~\ref{ass:stable_Hessian} with $\tau_0 = \exp(\gamma D)$.
\end{itemize}
\end{lemma}

\begin{proof}
Suppose $\theta \in \mathcal{B}_r(\theta_0)$. By Lemma~\ref{LemSCHess}, we have
\begin{equation*}
\nabla^2 \Loss_n(\theta) \succeq \exp(-\gamma r) \nabla^2 \Loss_n(\theta_0),
\end{equation*}
which immediately implies (i).

For (ii), note that Lemma~\ref{LemSCHess} also implies that for $\theta_1, \theta_2 \in \Theta_0$, we have
\begin{equation}
\label{EqnFish}
\nabla^2 \Loss_n(\theta_2) \preceq \exp(\gamma D) \nabla^2 \Loss_n(\theta_1).
\end{equation}
Hence, left- and right-multiplying by $(\theta_1 - \theta_2)$ and rearranging gives the upper bound
\begin{equation*}
\frac{\|\theta_1 - \theta_2\|^2_{\nabla^2 \Loss_n(\theta_2)}}{\|\theta_1 - \theta_2\|^2_{\nabla^2 \Loss_n(\theta_1)}} \le \exp(\gamma D).
\end{equation*}
Further note that for $\theta_1 \neq \theta_0$, we have $\|\theta_1 - \theta_2\|_{\nabla^2 \Loss_n(\theta_1)} > 0$, since
\begin{equation*}
\nabla^2 \Loss_n(\theta_1) \succeq \exp(-\gamma D)\nabla^2 \Loss_n(\theta_0) \succ 0,
\end{equation*}
by inequality~\eqref{EqnFish}.
\end{proof}

\subsection{Examples of self-concordant losses}
\label{AppSCExamples}

\begin{example}
[Smoothed Huber loss]
\label{ExaSmooth}

An intuitive way to circumvent the fact that Huber's $\psi_c$-function is not differentiable at $\{-c,c\}$ is to smooth out the corners where differentiability is violated. In particular, one could consider the following smooth approximation to Huber's score function:
 \begin{equation*}
     \psi_{c,h}(t)=
     \begin{cases} 
     t & \mbox{ for } |t|\leq c, \\
     P_4(t) & \mbox{ for } c<|t|<c+h, \\
     c' & \mbox{ for }  |t|\geq c+h  ,            
     \end{cases}
 \end{equation*} 
 where $c'>c$ and $P_4(t)$ is a piecewise fourth-degree polynomial, ensuring that  $\psi_{c,h(t)}$ is twice-differentiable everywhere. Hence, by construction,
 \begin{equation*}
     \psi_{c,h}'(t)=
     \begin{cases} 
     1 & \mbox{ for } |t|\leq c, \\
     P_4'(t) & \mbox{ for } c<|t|<c+h,\\
     0 & \mbox{ for }  |t|\geq c+h            
     \end{cases}
     \quad\mbox{and}\quad
     \psi_{c,h}''(t)=
     \begin{cases} 
     0 & \mbox{ for } |t|\leq c \mbox{ and }|t|\geq c+h , \\
     P_4''(t) & \mbox{ for } c<|t|<c+h.
     \end{cases}
 \end{equation*} 
This smoothed Huber function and related ideas have been discussed in the robust statistics literature \citep{fraimanetal2001,hampeletal2011}.
The proposed $\psi_{c,h}(t)$ can be used to define a Mallows loss function that meets the conditions of Theorem \ref{thm:NNewton}. Indeed, it is easy to see that for all $t_1,t_2\in[-c,c]$  and some $\bar t$ between $t_1$ and $t_2,$ we have 
$$\left(\psi_{c,h}(t_1)-\psi_{c,h}(t_2)\right)(t_1-t_2)\geq \psi_{c,h}'(\bar t)(t_1-t_2)^2=(t_1-t_2)^2,$$
which implies that $\psi_{c,h}(t)$ is $\frac{1}{2}$-locally strongly convex. This in turn can be used to establish that the objective function satisfies local strong convexity, with high probability, in a straightforward fashion.
Clearly, the objective function is also $\frac{1}{2}$-smooth, since $|\psi_{c,h}'(t)|\leq 1$ for all $t$.
However, this smoothed Huber loss is not $(\gamma,\nu)$-self-concordant, as we cannot find a constant $\gamma$ such that $ |\psi_{c,h}''(t)|\leq 2\gamma \{\psi_{c,h}'(t)\}^{\nu/2}$ for all $|t|\in(c,c+h)$. 
\end{example}
 
\begin{figure}[H]
    \centering
    \includegraphics[width=16cm]{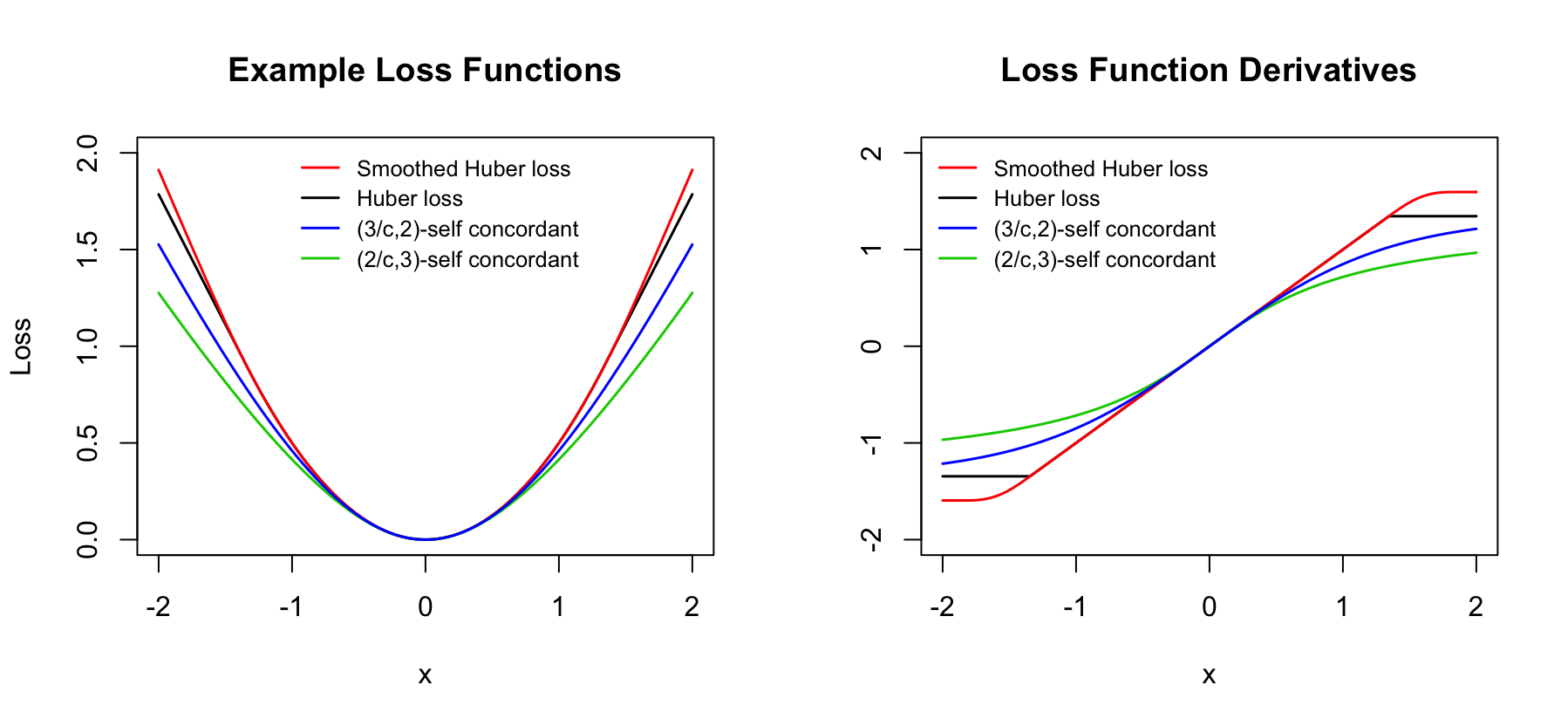} \\
    \caption{Univariate loss functions corresponding to variants of the Huber loss. For more details, see Examples~\ref{ExaSmooth} and~\ref{ExaSC}.}
    \label{fig:loss_functions}
\end{figure}



 
\begin{example}
[Self-concordant Huber regression  with Schweppe weights]
\label{ExaSC}

In the setting of Lemma~\ref{LemSchweppe}(i), we can choose the univariate $(3/c,2)$-self-concordant Huber functions 
$$\phi_c(t)=c^2\log(\cosh(t/c)) \quad \mbox{ and }\quad \phi_c(t)=c^2(\sqrt{1+(t/c)^2}-1).$$
 Alternatively, we could consider the following $(2/c,3)$-self-concordant Huber loss which, for $t \ne 0$, is defined as 
 $$\phi_c(t)=\frac{c^2}{2} \left[ \sqrt{1+4 (t/c)^{2}} -1 + \log{\left( \frac{\sqrt{1+4 (t/c)^{2}}-1}{2 (t/c)^{2}} \right)} \right].$$
The corresponding values at $t=0$ are defined to be $\phi_c(0)=0$. Figure~\ref{fig:loss_functions} illustrates the various proposals for smoothed and self-concordant versions of the Huber loss, as well as their derivatives. We note that these univariate self-concordant Huber losses were discussed by \cite{ostrovskii2021} in the context of a finite-sample theory for M-estimators.
These functions can also be shown to be locally strongly convex, since for all $t_1,t_2\in[-c,c]$, we have 
\begin{equation*}
\left(\phi(t_1)-\phi(t_2)\right)(t_1-t_2)\geq \min_{- c\leq t\leq c}\phi_{c}'( t)(t_1-t_2)^2
\end{equation*}
and $\min_{- c\leq t\leq c}\phi_{c}'( t)>0$. Assuming a linear model $y_i=x_i^\top\theta_0+u_i$ for $i=1,\dots,n$, and assuming that $\|v(x)x\|_2\leq C$ and $|v(x)|\leq 1$ for all $x\in\mathbb{R}^p$, it is also easy to check that the self-concordant Huber regression loss with Schweppe weights is also locally strongly convex with parameter $\tau_1=\frac{1}{2}\lambda_{\min}\left(\frac{1}{n}\sum_{i=1}^nv^2(x_i)x_ix_i^\top\psi'_c(|u_i|+Cr)\right)$  within $\mathcal{B}_r(\theta_0)$.  We note that this estimate holds for a fixed data set. One can give a deterministic $\tau_1$ that holds with high probability by the argument of Proposition 2 in \cite{loh2017}. 
We also note that this loss function is $\tau_2$-smooth with $\tau_2=\frac{1}{2}C^2\psi_c'(0)$.
\end{example}


\begin{example}
[Logistic regression]
\label{ExaLogistic}

Assuming that $\max_i\|x_i\|_2\leq C$, \cite{bach2010} showed that the logistic regression loss is $(C,2)$-self-concordant. Indeed, defining $\phi(t)=\log(e^{-t/2}+e^{t/2})$,  we see that $\phi'''(t)\leq |\phi''(t)|$, so $\phi$ is $(1, 2)$-self-concordant. Hence, by Lemma~\ref{LemMallows}, the logistic regression loss is $(C, 2)$-self-concordant. It is also clear that $\phi(t)$ is locally strongly convex, since for all $t_1,t_2\in[-c,c]$, we have
\begin{equation*}
\left(\phi(t_1)-\phi(t_2)\right)(t_1-t_2)\geq \left(\min_{- r\leq t\leq r}\phi'( t)\right)(t_1-t_2)^2=\frac{e^r}{(1+e^r)^2}(t_1-t_2)^2.
\end{equation*}
Assuming again that $\max_i\|x_i\|_2\leq C$, we can then show that the logistic regression loss is locally strongly convex with parameter $\tau_1=\frac{1}{2}\lambda_{\min}\left(\frac{1}{n}\sum_{i=1}^nx_ix_i^\top\right)\frac{\exp\{C(\|\theta_0\|_2+r)\}}{[1+\exp\{C(\|\theta_0\|_2+r)\}]^2}$ within $\mathcal{B}_r(\theta_0)$. It is also easy to see that see that the logistic regression loss is $\tau_2$-smooth where $\tau_2=\frac{1}{4}\lambda_{\max}\left(\frac{1}{n}\sum_{i=1}^nx_ix_i^\top\right)$.
 
%
%
It is not very obvious to construct a $(\gamma,\nu)$-self-concordant loss for binary regression such that the parameter $\gamma$ does not depend on the data. Indeed, a Mallows-type estimator such as the one considered in \cite{cantoniandronchetti2001} runs into the problem discussed in the remark following Lemma~\ref{LemMallows}. Furthermore, the Schweppe estimator for logistic regression proposed by \cite{kunschetal1989} is not twice-differentiable.
\end{example}
 

\section{Proofs for noisy gradient descent}

In this appendix, we provide the proofs of the two main results on the convergence of noisy gradient descent, Theorem~\ref{thm:NGD} and Proposition~\ref{prop:NGD_bad_starting_value}.


\subsection{Proof of Theorem \ref{thm:NGD}}
\label{AppThmNGD}

We first present the main argument, followed by statements and proofs of supporting lemmas in succeeding subsections.

\subsubsection{Main argument}

It is easy to see that the Gaussian mechanism and post-processing guarantee that every iteration of the algorithm is $\frac{\mu}{\sqrt{K}}$-GDP: Fixing $\theta^{(k)}$, the global sensitivity for the gradient iterates $\theta^{(k+1)}$ in equation~\eqref{eq:GD} is clearly bounded by $\frac{2\eta B}{n}$, and then we simply apply the Gaussian mechanism (Theorem~\ref{thm:DP-GS}) to obtain the noisy gradient descent iterates~\eqref{eq:NGD}. It follows from Corollary 2 in \cite{dongetal2021} that the entire algorithm is $\mu$-GDP after $K$ iterations. This proves (i).

Let us now turn to the proof of (ii). We denote $N_k=\frac{B\sqrt{K}}{\mu n}Z_{k}$, so the noisy gradient updates can be rewritten as $\theta^{(k+1)}=\theta^{(k)}-\eta\nabla\mathcal{L}_n(\theta^{(k)})+\eta N_k$.
From Lemma \ref{lem:maxGaussian} and a union bound, we have that with probability at least $1-\xi$,
\begin{equation}
    \label{eq:thm2.2}
    \|N_k\|_2\leq \frac{\{4\sqrt{p}+2\sqrt{2\log(K/\xi)}\}B\sqrt{K}}{\mu n} := r_{priv},
\end{equation}
for all $k < K$. For the remainder of the argument, we will assume the bound~\eqref{eq:thm2.2} holds, and prove that the desired conditions hold deterministically.

From Lemmas~\ref{lemma1_alternative} and~\ref{lemma2} and the local strong convexity assumption, we have
\begin{align}
\label{eq:thm2.1}
\nonumber \|\theta^{(k)}-\hat{\theta}\|_2^2  &\leq \frac{1}{\tau_1}\kappa^{k}\Delta_0+\frac{3rr_{priv}}{2(1-\kappa)\tau_1}\\
    & \leq \frac{3 r r_{priv}}{(1-\kappa)\tau_1},
\end{align}
where $r_{priv}=\frac{\{4\sqrt{p}+2\sqrt{2\log(K/\xi)}\}B\sqrt{K}}{\mu n}$ and the last inequality holds as long as $k\geq k_0=\lceil k_0'\rceil$, where
\begin{equation*}
k_0'= \frac{\log(1/\Delta_0)+\log\left(\frac{3r\{4\sqrt{p}+2\sqrt{2\log(K/\xi)}\}B\sqrt{K}}{2(1-\kappa)\mu n}\right)}{\log(\kappa)}.
\end{equation*}
We note that the bound \eqref{eq:thm2.1} is looser than the desired result. The rest of our argument will improve this estimate for later iterates $k > K$ (with $K > k_0$) by  refining the argument of Lemma \ref{lemma2}, leveraging inequality~\eqref{eq:thm2.1}.

As derived in inequality~\eqref{lem2.3} in the proof of Lemma \ref{lemma2}, we can show that
%
\begin{equation}
\label{EqnGF}
\Delta_{k+1} \leq (1-\gamma)\Delta_k+\|N_k\|_2\left(\|\hat\theta-\theta^{(k)}\|_2+ \|\theta^{(k+1)}-\theta^{(k)}\|_2\right),
\end{equation}
where $\Delta_k=\mathcal{L}_n(\theta^{(k)})-\mathcal{L}_n(\hat\theta)$ and $\gamma = 2\eta \tau_1$.
Furthermore, note that smoothness, inequality~\eqref{eq:thm2.1}, and $k\geq k_0$ imply that
%
\begin{equation*}
    \|\nabla\mathcal{L}_n(\theta^{(k)})\|_2= \|\nabla\mathcal{L}_n(\theta^{(k)})-\nabla\mathcal{L}_n(\hat\theta)\|_2\leq 2\tau_2\|\theta^{(k)}-\hat\theta\|_2 \leq 2\tau_2 \sqrt{\frac{3 r r_{priv}}{(1-\kappa)\tau_1}},
\end{equation*}
Thus, the triangle inequality and inequality~\eqref{eq:thm2.2} give
\begin{equation}
    \label{eq:thm2.5}
    \|\theta^{(k+1)}-\theta^{(k)}\|_2= \eta  \|\nabla\mathcal{L}_n(\theta^{(k)})+N_k\|_2\leq \eta r_{priv}+2\tau_2\eta\sqrt{\frac{3 r }{(1-\kappa)\tau_1}}r_{priv}^{1/2}.
\end{equation}
Consequently, from inequalities~\eqref{eq:thm2.1}, \eqref{EqnGF}, and \eqref{eq:thm2.5} and the fact that $k\geq k_0$, we obtain
\begin{equation*}
  \Delta_{k+1}\leq (1-\gamma)\Delta_k +r_{priv}\left(\eta r_{priv}+(2\tau_2\eta+1)\sqrt{\frac{3 r }{(1-\kappa)\tau_1}}r_{priv}^{1/2}\right).
\end{equation*}
Recalling that $\kappa = 1 - \gamma$, we have that for $j=k_0+1,\dots,K$, 
\begin{align}
    \label{eq:thm2.6}
\nonumber \Delta_{j+k_0} &\leq \kappa \Delta_{j-1+k_0} + (2\tau_2\eta+1)\sqrt{\frac{3 r }{(1-\kappa)\tau_1}}r_{priv}^{3/2} +\eta r_{priv}^{2} \\
 &\leq \kappa^j\Delta_{k_0} + \left(2\tau_2\eta+\frac{3}{2}\right)\sqrt{\frac{3 r }{(1-\kappa)\tau_1}}r_{priv}^{3/2},    
\end{align}
where the last inequality used the fact that $n\geq \frac{\{4\sqrt{p}+2\sqrt{2\log(K/\xi)}\}B\sqrt{K}}{\mu \cdot \frac{1}{4\eta^2}\frac{3r}{(1-\kappa)\tau_1}} \iff \eta r_{priv}^{1/2}\leq \frac{1}{2}\sqrt{\frac{3r}{(1-\kappa)\tau_1}}$. Taking $k=j+k_0$, local strong convexity and inequality~\eqref{eq:thm2.6} show that
%
\begin{align}
   \label{eq:thm2.7}
   \|\theta^{(k)}-\hat\theta\|_2^2& \leq \frac{1}{\tau_1}\kappa^j\Delta_{k_0}+\left(2\tau_2\eta+\frac{3}{2}\right)\sqrt{\frac{3 r }{(1-\kappa)\tau_1^3}}r_{priv}^{3/2} \notag \\
    &\leq (2\tau_2\eta+2)\sqrt{\frac{3 r }{(1-\kappa)\tau_1^3}}r_{priv}^{3/2},
\end{align}
where the last inequality holds as long as $k\geq k_1+k_0$ with  $k_1=\lceil k_1'\rceil$ and 
\begin{equation*}
k_1'= \frac{\log(1/\Delta_{k_0})+\log\left(\frac{1}{2}\sqrt{\frac{3 r }{(1-\kappa)\tau_1}}r_{priv}^{3/2}\right)}{\log(\kappa)}.
\end{equation*}

Thus, we see $k_1+k_0$  iterations of the algorithm improve the convergence rate of the iterates from $O\left(r_{priv}^{1/2}\right)$, obtained in \eqref{eq:thm2.1} after $k_0$ iterations, to $O\left(r_{priv}^{3/4}\right)$. Repeating the same argument $m$ times, we will show that successive iterates of the algorithm in fact  improve the estimation error bound successively as $O\left(r_{priv}^{1/2}\right), O\left(r_{priv}^{3/4}\right), O\left(r_{priv}^{7/8}\right), \dots, O\left(r_{priv}^{1-1/2^m}\right)$. This will be enough to prove the desired result, since we can take $m=\log_2 n$, which gives the rate $O\left(r_{priv}^{1-1/n}\right)$. 

Letting $C_0=\sqrt{\frac{3 r }{(1-\kappa)\tau_1}}$, we see that inequality~\eqref{eq:thm2.1} says that for $k>k_0$, we have $\|\theta^{(k)}-\hat{\theta}\|_2\leq C_0 r_{priv}^{1/2}$. We then saw that for $k>k_1+k_0$ this bound can be further refined to 
$$\|\theta^{(k)}-\hat{\theta}\|_2\leq \sqrt{\frac{2\tau_2\eta+2}{\tau_1}C_0} r_{priv}^{3/4} :=C_1r_{priv}^{3/4}.$$
The same arguments used to obtain inequality~\eqref{eq:thm2.6} show that for $j=k_0+k_1+1,\dots,K$, we have
\begin{align}
    \label{eq:thm2.6*}
\nonumber \Delta_{j+k_0+k_1} &\leq \kappa\Delta_{j-1+k_0} +(2\tau_2\eta+1)C_1r_{priv}^{1+3/4} +\eta r_{priv}^{2} \\
 &\leq \kappa^{j+k_1}\Delta_{k_0} + \left(2\tau_2\eta+\frac{3}{2}\right)C_1r_{priv}^{7/4},    
\end{align}
where the last inequality uses $n\geq \frac{\{4\sqrt{p}+2\sqrt{2\log(K/\xi)}\}B\sqrt{K}}{\mu \frac{1}{2^4}C_1^4}\eta^4$, which is implied by the minimum sample size assumption $n\geq \frac{\{4\sqrt{p}+2\sqrt{2\log(K/\xi)}\}B\sqrt{K}}{\mu \cdot  \frac{1}{2^2}C_0^2}\eta^2$ used to show inequality~\eqref{eq:thm2.6}, since $\frac{1}{2^2\eta^2}C_0^2<\frac{1}{2^4\eta^4}C_1^4$.  
Similar to inequality~\eqref{eq:thm2.7}, taking $k=j+k_0+k_1$, local strong convexity and inequality~\eqref{eq:thm2.6*} show that
\begin{align}
   \label{eq:thm2.7*}
   \|\theta^{(k)}-\hat\theta\|_2^2 & \leq \frac{1}{\tau_1}\kappa^{j}\Delta_{k_0+k_1}+\frac{2\tau_2\eta+3/2}{\tau_1}C_1r_{priv}^{7/4} \notag \\
    &\leq \frac{2\tau_2\eta+2}{\tau_1}C_1r_{priv}^{7/4}, 
\end{align}
where the last inequality holds as long as $k\geq k_2+k_1+k_0$ with  $k_2=\lceil k_2'\rceil$ and 
\begin{equation*}
k_2'= \frac{\log(1/\Delta_{k_0+k_1})+\log\left((2\tau_2\eta+2)C_1r_{priv}^{7/4}\right)}{\log(\kappa)}.
\end{equation*}
We conclude from inequality~\eqref{eq:thm2.7*} that for $k>k_0+k_1+k_2$ we have 
$$\|\theta^{(k)}-\hat\theta\|_2\leq \sqrt{\frac{2\tau_2\eta+2}{\tau_1}C_1}r_{priv}^{7/8}=C_2r_{priv}^{1-1/2^3}. $$
Iterating this argument, a tedious but straightforward calculation shows that we can obtain the following recurrence: Let $C_0=\sqrt{\frac{3 r }{(1-\kappa)\tau_1}}$, and define $C_m:=\sqrt{\frac{2\tau_2\eta+2}{\tau_1}C_{m-1}}$, $k_m:=\lceil k_m'\rceil$, and
\begin{equation*}
k_m'= \frac{\log(1/\Delta_{\sum_{j=0}^{m-1}k_j})+\log\left((2\tau_2\eta+2)C_{m-1}r_{priv}^{2-1/2^m}\right)}{\log(\kappa)}.
\end{equation*} 
Note that we also have 
\begin{equation}
\label{eq:Cm}
    C_m = \left(\frac{2\tau_2\eta+2}{\tau_1}\right)^{\sum_{j=1}^m1/2^j}C_0^{1/2^m} 
   =\left(\frac{2\tau_2\eta+2}{\tau_1}\right)^{1-1/2^m}C_0^{1/2^m}.
\end{equation}
Then, taking $k>\sum_{j=0}^mk_j$ and $m=\log_2n -1$, we have 
\begin{equation}
       \label{eq:thm2.8}
\|\theta^{(k)}-\hat\theta\|_2\leq C_mr_{priv}^{1-1/2^{m+1}}= \left(\frac{2\tau_2\eta+2}{\tau_1}\right)^{1-2/n}C_0^{2/n}r_{priv}^{1-1/n}\leq C r_{priv},
\end{equation} 
where $C$ is some positive constant. We note that \eqref{eq:thm2.8} implicitly used the fact that the minimum sample size requirement $n\geq \frac{\{4\sqrt{p}+2\sqrt{2\log(K/\xi)}\}B\sqrt{K}}{\mu \cdot \frac{1}{(2\eta)^2}C_0^2}$ also implies that $n\geq \frac{\{4\sqrt{p}+2\sqrt{2\log(K/\xi)}\}B\sqrt{K}}{\mu \cdot \frac{1}{(2\eta)^{2^{m+1}}}C_{m-1}^{2^m}}$ when $\eta\leq \frac{1}{2}$. To see this, it  suffices to check that $C_0^2\leq \frac{1}{(2\eta)^{2^{m}}}C_{m-1}^{2^m}$. The latter holds true since inequality~\eqref{eq:Cm} and $\tau_2\geq \tau_1$ show that $
  C_{m-1} \geq (2\eta)^{1-1/2^{m-1}}C_0^{1/2^{m-1}}
$. 
This last inequality implies the desired inequality, since clearly,
$$C_0^2\leq \frac{1}{(2\eta)^{2^{m}}}\left((2\eta)^{1-1/2^{m-1}}C_0^{1/2^{m-1}}\right)^{2^m}= \frac{1}{(2\eta)^{2}}C_0^2.$$
This completes the proof.


\subsubsection{Supporting lemmas}

%
\begin{lemma} 
\label{lem:LSCball}
Suppose $\mathcal{L}_n$ is convex in $\Theta\subseteq\mathbb{R}^p$ and locally $\tau_1$-strongly convex in $\mathcal{B}_r(\theta_0)$.
If $\hat\theta\in\mathcal{B}_{r/2}(\theta_0)$ and  $\mathcal{L}_n(\theta)-\mathcal{L}_n(\hat\theta)\leq \frac{r^2}{4} \tau_1$, then $\|\theta - \hat\theta\|_2 \le \frac{r}{2}$ and $\theta \in \mathcal{B}_r(\theta_0)$.
\end{lemma}

\begin{proof}
Define 
\begin{equation*}
\Delta := \inf_{\theta \not\in \mathcal{B}_{r/2}(\hat{\theta})} \mathcal{L}_n(\theta) - \mathcal{L}_n(\hat{\theta}).
\end{equation*}
By convexity of $\mathcal{L}_n$, the infimum must be achieved at some point $\theta^*_r$ on the boundary of the ball $\mathcal{B}_{r/2}(\hat{\theta})$. Therefore, for any parameter $\theta$ such that $\mathcal{L}_n(\theta) - \mathcal{L}_n(\hat{\theta}) < \Delta$, we must have $\|\theta - \hat{\theta}\|_2 \le \frac{r}{2}$, and hence also $\theta \in \mathcal{B}_r(\theta_0)$.

We now claim that $\Delta \ge \frac{r^2}{4} \tau_1$, from which the desired result follows. By the triangle inequality, we have $\theta^*_r \in \mathcal{B}_r(\theta_0)$. Thus, by strong convexity,
\begin{equation*}
\Delta = \mathcal{L}_n(\theta^*_r) - \mathcal{L}_n(\hat{\theta}) \ge \tau_1 \|\theta^*_r - \hat{\theta}\|_2^2 = \frac{r^2}{4} \tau_1,
\end{equation*}
as claimed.
\end{proof}

\begin{lemma}
\label{lemma1_alternative}
Suppose $\mathcal{L}_n(\theta)$ is twice-differentiable almost everywhere in $\mathcal{B}_r(\theta_0)$ and satisfies LSC and strong smoothness with parameters $\tau_1$ and $\tau_2 \le \frac{1}{2\eta}$. Also suppose the sample size satisfies $n = \Omega\left(\frac{\sqrt{K \log(K/\xi)}}{\mu}\right)$ and inequality~\eqref{eq:thm2.2} holds, and  suppose $\hat\theta\in\mathcal{B}_{r/2}(\theta_0)$ and $\Loss_n(\theta^{(0)}) \le \Loss_n(\hat\theta) + \tau_1 \frac{r^2}{4}$.
Then
\begin{equation*}
\mathcal{L}_n(\theta^{(k)})\leq \mathcal{L}_n(\thetahat) + \tau_1 \frac{r^2}{4} \quad \text{and} \quad \|\theta^{(k)}-\hat\theta\|_2\leq \frac{r}{2}, \quad \forall k \le K.
\end{equation*}
\end{lemma}

\begin{proof}

We will induct on $k$. Note that by Lemma~\ref{lem:LSCball}, we have $\|\theta^{(0)}-\hat\theta\|_2\leq \frac{r}{2}$.

For the inductive step, suppose $\Loss_n(\theta^{(k)}) \le \Loss_n(\thetahat) + \tau_1 \frac{r^2}{4}$ and $\|\theta^{(k)}-\hat\theta\|_2\leq \frac{r}{2}$ for some $0 \le k <  n^{\alpha}$. Since $\mathcal{L}_n$ is almost everywhere twice-differentiable, we have
\begin{equation}
\label{eq:Taylor0}
    \mathcal{L}_n(\theta-\eta\Delta)= \mathcal{L}_n(\theta)-\eta\langle\nabla\mathcal{L}_n(\theta),\Delta\rangle+\frac{\eta^2}{2}\langle\textstyle{\int_0^1\nabla^2\mathcal{L}_n(\theta-t\eta\Delta)\mathrm{d}t}\Delta,\Delta\rangle,
\end{equation}
for any $\Delta\in\mathbb R^p$.
Recalling that $N_k=\frac{B\sqrt{K}}{\mu n}Z_{k}$ and $\theta^{(k+1)}=\theta^{(k)}-\eta\nabla\mathcal{L}_n(\theta^{(k)})+\eta N_k$, applying equation~\eqref{eq:Taylor0} with $\theta = \theta^{(k)}$ and $\Delta = \nabla \Loss_n(\theta^{(k)}) - N_k$ leads to
\begin{align}
\label{eq:Taylor}
\mathcal{L}_n(\theta^{(k+1)}) & = \mathcal{L}_n(\theta^{(k)})-\eta\|\nabla\mathcal{L}_n(\theta^{(k)})\|_2^2 +\eta\langle \nabla\mathcal{L}_n(\theta^{(k)}),N_k\rangle \notag \\
& \qquad +\frac{\eta^2}{2}\|\nabla\mathcal{L}_n(\theta^{(k)})\|_{H_k}^2-\eta^2\langle N_k,H_k\nabla\mathcal{L}_n(\theta^{(k)})\rangle +\frac{\eta^2}{2}\| N_k\|_{H_k}^2 \notag \\
& \le \mathcal{L}_n(\theta^{(k)})-\eta\|\nabla\mathcal{L}_n(\theta^{(k)})\|_2^2 +\eta \| \nabla\mathcal{L}_n(\theta^{(k)})\|_2 \|N_k\|_2 \notag \\
& \qquad +\frac{\eta^2}{2}\|\nabla\mathcal{L}_n(\theta^{(k)})\|_{H_k}^2 + \eta^2\| N_k\|_2 \|H_k\|_2 \|\nabla\mathcal{L}_n(\theta^{(k)})\|_2 +\frac{\eta^2}{2}\| N_k\|_{H_k}^2,
\end{align}
where $H_k=\int_0^1\nabla^2\mathcal{L}_n\left(\theta^{(k)}-t\eta\nabla\mathcal{L}_n(\theta^{(k)}) + t\eta N_k\right)\mathrm{d}t$. By inequality~\eqref{eq:Taylor} and the Cauchy-Schwarz inequality, we can guarantee that $\mathcal{L}_n(\theta^{(k+1)})\leq \mathcal{L}_n(\theta^{(k)})$ holds if the following inequality is satisfied:
\begin{equation}
    \label{eq:sufficient1}
    \|\nabla\mathcal{L}_n(\theta^{(k)})\|^2_{I-\frac{\eta}{2}H_k} \ge \| \nabla\mathcal{L}_n(\theta^{(k)})\|_2\|N_k\|_2+ \eta\|N_k\|_2\|H_k\|_2\|\nabla\mathcal{L}_n(\theta^{(k)})\|_2 +\frac{\eta}{2}\| N_k\|_2^2\|H_k\|_2. 
\end{equation}
The descent condition $\mathcal{L}_n(\theta^{(k+1)})\leq \mathcal{L}_n(\theta^{(k)})$, together with the inductive hypothesis, would then
imply that
\begin{equation*}
\Loss_n(\theta^{(k+1)}) - \Loss_n(\hat\theta) \le \Loss_n(\theta^{(k)}) - \Loss_n(\hat\theta) \le \tau_1 \frac{r^2}{4},
\end{equation*}
so from Lemma~\ref{lem:LSCball}, we could conclude that $\theta^{(k+1)} \in \mathcal{B}_{r/2}(\hat\theta)$, as wanted.

Returning to inequality~\eqref{eq:sufficient1}, note that $\|\nabla\mathcal{L}_n(\theta^{(k)})\|^2_{I-\frac{\eta}{2}H_k}\geq \|\nabla\mathcal{L}_n(\theta^{(k)})\|^2_2(1-\eta\tau_2)$, since by $\tau_2$-smoothness, we have $\lambda_{\max}(H_k)\leq 2\tau_2$. Using this bound on the left-hand side of inequality~\eqref{eq:sufficient1} and the upper bounds $\sup_\theta\|\nabla\mathcal{L}_n(\theta)\|_2\leq B$ and $\lambda_{\max}(H_k)\leq 2\tau_2$ on the right-hand side of inequality~\eqref{eq:sufficient1}, we have the stronger sufficient condition
\begin{equation*}
    \|\nabla\mathcal{L}_n(\theta^{(k)})\|^2_2(1-\eta\tau_2) \ge \|N_k\|_2B\left(1+ 2 \eta \tau_2\right) +\eta\tau_2\| N_k\|_2^2,
\end{equation*}
which is guaranteed to be satisfied when 
\begin{equation*}
        \|N_k\|  \leq \frac{-B(1+2\eta \tau_2)+\sqrt{B^2(1+2\eta \tau_2)^2+4\eta\tau_2(1-\eta\tau_2)\|\nabla\mathcal{L}_n(\theta^{(k)})\|^2_2}}{2\eta\tau_2},\\
\end{equation*}
or equivalently,
\begin{align}
        \label{eq:sufficient3}
\nonumber     \|\nabla\mathcal{L}_n(\theta^{(k)})\|_2 & \geq \sqrt{\frac{\left(2\eta\tau_2\|N_k\|_2+B(1+2\eta \tau_2)\right)^2-B^2(1+2\eta \tau_2)^2}{4\eta\tau_2(1-\eta\tau_2)}}\\
     &=\sqrt{\frac{\eta\tau_2\|N_k\|_2^2+B(1+2\eta \tau_2)\|N_k\|_2}{1-\eta\tau_2}}.
\end{align}
%
%
%
Combining inequality~\eqref{eq:sufficient3} with the upper bound~\eqref{eq:thm2.2} on $\|N_k\|_2$, we see that inequality~\eqref{eq:sufficient1} holds with high probability, provided 
\begin{equation}
\label{eq:sufficient4}
\|\nabla\mathcal{L}_{n}(\theta^{(k)})\|_2 \geq \sqrt{\frac{\eta\tau_2r_{priv}^2+B(1+2 \eta\tau_2)r_{priv}}{1-\eta\tau_2}}=\bar r_{priv}.
\end{equation}
%

We now argue that if $\|\nabla\mathcal{L}_n(\theta^{(k)})\|_2< \bar{r}_{priv}$, we still have $\|\theta^{(k+1)} - \thetahat\|_2 \le \frac{r}{2}$ and $\mathcal{L}_n(\theta^{(k+1)})\leq \mathcal{L}_n(\thetahat) + \tau_1 \frac{r^2}{4}$. Indeed, since $\theta^{(k)} \in \mathcal{B}_{r} (\theta_0)$ by assumption, local strong convexity implies that
\begin{equation*}
\tau_1 \|\theta^{(k)} - \thetahat\|_2^2 \le \Loss_n(\thetahat) - \Loss_n(\theta^{(k)}) - \langle \nabla \Loss_n(\theta^{(k)}), \thetahat - \theta^{(k)} \rangle \le \|\nabla \Loss_n(\theta^{(k)})\|_2 \|\theta^{(k)} - \thetahat\|_2,
\end{equation*}
so
\begin{equation*}
\bar{r}_{priv} > \|\nabla \Loss_n(\theta^{(k)})\|_2 \ge \tau_1 \|\theta^{(k)} - \thetahat\|_2.
\end{equation*}
Hence, by the triangle inequality and inequality~\eqref{eq:thm2.2}, we have
\begin{align}
    \nonumber\|\theta^{(k+1)}-\hat\theta\|_2&\leq \|\theta^{(k)}-\hat\theta\|_2+\|\theta^{(k+1)}-\theta^{(k)}\|_2\\
    \nonumber&\leq \frac{\bar{r}_{priv}}{\tau_1}+\eta\|\nabla\mathcal{L}_n(\theta^{(k)})+N_k\|_2\\
    \nonumber&\leq \left(\frac{1}{\tau_1}+\eta\right)\bar{r}_{priv}+\eta r_{priv}\\
    &\leq \left(\frac{\tau_1}{\tau_2}\right)^{1/2} \frac{r}{2} \le \frac{r}{2} \label{eq:radius_iterate}
\end{align}
with high probability, when $n$ is sufficiently large. Consequently, $\|\nabla\mathcal{L}_n(\theta^{(k)})\|_2< \bar{r}_{priv}$ also implies that $\theta^{(k+1)}\in\mathcal{B}_{r/2}(\hat\theta)$; furthermore, by smoothness, we also have
\begin{equation*}
\Loss_n(\theta^{(k+1)}) - \Loss_n(\thetahat) \le \tau_2 \|\theta^{(k+1)} - \thetahat\|_2^2 \le \tau_2 \cdot \frac{\tau_1}{\tau_2} \frac{r^2}{4} \le \tau_1 \frac{r^2}{4}.
\end{equation*}
Note that the penultimate inequality in the chain~\eqref{eq:radius_iterate} holds under the sample size condition $n = \Omega\left(\frac{\sqrt{K \log(K/\xi)}}{\mu}\right)$.
Indeed, $\eta r_{priv}\leq \frac{r}{4} \sqrt{\frac{\tau_1}{\tau_2}}$ if $n\geq \eta\frac{4\{4\sqrt{p}+2\sqrt{2\log(K/\xi)}\}B\sqrt{K}}{r\mu\sqrt{\tau_1/\tau_2}}$ and $\left(\frac{1}{\tau_1}+\eta\right)\bar{r}_{priv}\leq\frac{r}{4}\sqrt{\frac{\tau_1}{\tau_2}}$ if 
 \begin{equation*}
     r_{priv}\leq \frac{\frac{- B(1+2\eta \tau_2)}{1-\eta\tau_2}+\sqrt{\frac{B^2(1+2\eta\tau_2)^2}{(1-\eta\tau_2)^2}+\frac{r^2}{4}\frac{\eta\tau_1}{(1-\eta\tau_2)(\frac{1}{\tau_1}+\eta)^2}}}{\frac{2\eta\tau_2}{1-\eta\tau_2}}.
 \end{equation*}
\end{proof}


In fact, as the following result shows, the assumptions in Lemma~\ref{lemma1_alternative} also imply that the sub-optimality gap of successive iterates decreases at a geometric rate (up to an error term).

\begin{lemma}
\label{lemma2}
Let $\Delta_k=\mathcal{L}_n(\theta^{(k)})-\mathcal{L}_n(\hat\theta)$, and suppose the conditions of Lemma \ref{lemma1_alternative} hold. Then
$$\Delta_{k}=\kappa^{k}\Delta_{0}+\frac{3r\{4\sqrt{p}+2\sqrt{2\log(K/\xi)}\}B\sqrt{K}}{2(1-\kappa)\mu n}, \quad \forall k \le K,  $$
where $\kappa = 1 - 2\eta \tau_1 \in [0,1)$.
\end{lemma}
 
\begin{proof}
Note that $\theta^{(k+1)}$ can be viewed as the minimizer of 
\begin{equation*}
    \label{eq:NGG_objective}
    Q_k(\theta):=\mathcal{L}_n(\theta^{(k)})+\langle \nabla\mathcal{L}_n(\theta^{(k)})-N_k,\theta-\theta^{(k)} \rangle + \frac{1}{2\eta}\|\theta-\theta^{(k)}\|_2^2.
\end{equation*}
Denote $\theta_{\gamma}=\gamma\hat\theta+(1-\gamma)\theta^{(k)}$, for a parameter $\gamma\in(0,1)$ to be chosen later. By optimality of $\theta^{(k+1)}$ and local strong convexity,  we have 
\begin{align}
    \label{lem2.1}
\nonumber       Q_k(\theta^{(k+1)})&\leq Q_k(\theta_\gamma) \\
 \nonumber      & = \mathcal{L}_n(\theta^{(k)})+\gamma\langle\nabla\mathcal{L}_n(\theta^{(k)})-N_k,\hat\theta-\theta^{(k)}\rangle+\frac{\gamma^2}{2\eta}\|\theta^{(k)}-\hat\theta\|_2^2\\
\nonumber       &\leq \mathcal{L}_n(\theta^{(k)})-\gamma\left(\mathcal{L}_n(\theta^{(k)})-\mathcal{L}_n(\hat\theta)\right)+\left(\frac{\gamma^2}{2\eta}-\tau_1\right)\|\theta^{(k)}-\hat\theta\|_2^2-\gamma\langle N_k, \hat\theta-\theta^{(k)}\rangle \\
       &= \mathcal{L}_n(\theta^{(k)})-\gamma\Delta_k+\left(\frac{\gamma^2}{2\eta}-\tau_1\right)\|\theta^{(k)}-\hat\theta\|_2^2-\gamma\langle N_k, \hat\theta-\theta^{(k)}\rangle. 
\end{align}
Furthermore, by local smoothness and the fact that $\tau_1\leq \tau_2\leq \frac{1}{2\eta}$, we have
\begin{align}
    \label{lem2.2}
 \mathcal{L}_n(\theta^{(k+1)})   \nonumber  &\leq \mathcal{L}_n(\theta^{(k)})+\langle \nabla\mathcal{L}_n(\theta^{(k)}), \theta^{(k+1)}-\theta^{(k)}\rangle+\tau_2\|\theta^{(k+1)}-\theta^{(k)}\|_2^2 \\
     &\leq Q_k(\theta^{(k+1)})+\langle N_k, \theta^{(k+1)}-\theta^{(k)}\rangle.
\end{align}
Combining inequalities~\eqref{lem2.1} and \eqref{lem2.2}, we obtain 
\begin{align}
\label{lem2.2+}
 \mathcal{L}_n(\theta^{(k+1)}) &\leq \mathcal{L}_n(\theta^{(k)})-\gamma\Delta_k+\left(\frac{\gamma^2}{2\eta}-\tau_1\right)\|    \theta^{(k)}-\hat\theta\|_2^2-\gamma\langle N_k, \hat\theta-\theta^{(k)}\rangle+\langle N_k, \theta^{(k+1)}-\theta^{(k)}\rangle.
\end{align}
Furthermore, subtracting $\mathcal{L}_n(\hat\theta)$ from both sides of inequality~\eqref{lem2.2+}, we see that local strong convexity, taking $\gamma = 2\eta \tau_1$ so that $\gamma^2 \le  2\eta\tau_1$,
and using the bound~\eqref{eq:thm2.2} implies that
%
\begin{align}
    \label{lem2.3}
    \nonumber \Delta_{k+1} &\leq \Delta_k(1-\gamma) +\left(\frac{\gamma^2}{2\eta}-\tau_1\right)\|    \theta^{(k)}-\hat\theta\|_2^2-\gamma\langle N_k, \hat\theta-\theta^{(k)}\rangle+\langle N_k, \theta^{(k+1)}-\theta^{(k)}\rangle \\
    \nonumber  &\leq (1-\gamma)\Delta_k-\gamma\langle N_k, \hat\theta-\theta^{(k)}\rangle+\langle N_k, \theta^{(k+1)}-\theta^{(k)}\rangle \\
    \nonumber  &\leq (1-\gamma)\Delta_k+\|N_k\|_2\left(\|\hat\theta-\theta^{(k)}\|_2+ \|\theta^{(k+1)}-\theta^{(k)}\|_2\right) \\
     &\leq (1-\gamma)\Delta_k+r_{priv}\frac{3r}{2},
\end{align}
where we have used the conclusion of Lemma~\ref{lemma1_alternative} in the last inequality. Iterating the bound~\eqref{lem2.3} and writing $\kappa=1-\gamma$, we see that 
\begin{align*}
     \Delta_{k}&\leq \kappa^k\Delta_{0}+\frac{3r_{priv}r}{2}(1+\kappa+\kappa^2+\dots+\kappa^{k-1})\\
     &\leq \kappa^k\Delta_{0}+\frac{3r_{priv}r}{2(1-\kappa)}.
\end{align*}
This shows the desired result.
\end{proof}


\subsection{Proof of Proposition \ref{prop:NGD_bad_starting_value}}
\label{App:prop:NGD_bad_starting_value}

We again begin by presenting the main argument, followed by statements and proofs of supporting lemmas.

\subsubsection{Main argument}

Let $N_k = \frac{B\sqrt{K}}{\mu n} Z_k$. By Lemma \ref{lem:maxGaussian} and a union bound, we have
\begin{equation}
\label{EqnNkBd}
\max_{k < K_0} \|N_k\|_2\leq \frac{\{4\sqrt{p}+2\sqrt{2\log(K_0/\xi_0)}\}B\sqrt{K_0}}{\mu n} := r_{priv},
\end{equation}
with probability at least $1-\xi_0$. We will assume this inequality holds for the rest of the proof, and argue that the desired conditions hold deterministically.

Recall the form of the updates~\eqref{eq:NGD}. By strong smoothness and the Cauchy-Schwarz inequality, we have 
\begin{align*}
\mathcal{L}_n(\theta^{(k+1)}) &\leq \mathcal{L}_n(\theta^{(k)})-\eta\langle \nabla\mathcal{L}_n(\theta^{(k)}),\nabla\mathcal{L}_n(\theta^{(k)})-N_k\rangle+\tau_2\eta^2\|\nabla\mathcal{L}_n(\theta^{(k)})-N_k\|_2^2 \\
&\leq \mathcal{L}_n(\theta^{(k)})-\eta\left(1-\tau_2\eta\right)\|\nabla\mathcal{L}_n(\theta^{(k)})\|_2^2 +\eta\left(1+2\tau_2\eta\right)\|N_2\|_2\|\nabla\mathcal{L}_n(\theta^{(k)})\|_2+\tau_2\eta^2\|N_k\|_2^2.
\end{align*}
Using the bound~\eqref{EqnNkBd} and the fact that $\eta\leq\frac{1}{2\tau_2}$, we then have
%
\begin{equation}
    \label{eq:prop1.1}
    \mathcal{L}_n(\theta^{(k+1)})\leq \mathcal{L}_n(\theta^{(k)})- \frac{\eta}{2}\|\nabla\mathcal{L}_n(\theta^{(k)})\|_2^2+Cr_{priv},
\end{equation}
for all $k\leq K_0$,
where $C \ge 2\eta B+\frac{\eta}{2}r_{priv}$.
Furthermore, convexity implies that 
\begin{equation}
    \label{eq:prop1.2}
    \mathcal{L}_n(\theta^{(k)})\leq \mathcal{L}_n(\hat\theta)+\langle\nabla\mathcal{L}_n(\theta^{(k)}),\theta^{(k)}-\hat\theta\rangle.
\end{equation}
Combining inequalities~\eqref{eq:prop1.1} and \eqref{eq:prop1.2}, we obtain
\begin{equation*}
    \begin{split}
    \mathcal{L}_n(\theta^{(k+1)})    & \leq \mathcal{L}_n(\hat\theta)+\langle\nabla\mathcal{L}_n(\theta^{(k)}),\theta^{(k)}-\hat\theta\rangle-\frac{\eta}{2}\|\nabla\mathcal{L}_n(\theta^{(k)})\|_2^2+C r_{priv}.
    \end{split}
\end{equation*}
This in turn shows that 
\begin{align*}
\mathcal{L}_n(\theta^{(k+1)})  -\mathcal{L}_n(\hat\theta) &\leq \frac{1}{2\eta}\left(2\eta\langle\nabla\mathcal L_n(\theta^{(k)}),\theta^{(k)}-\hat\theta\rangle -\eta^2\|\nabla\mathcal{L}_n(\theta^{(k)})\|_2^2\right)+C r_{priv} \notag \\
   \nonumber &=\frac{1}{2\eta}\left(\|\theta^{(k)}-\hat\theta\|_2^2-\|\theta^{(k)}-\eta\nabla\mathcal{L}_n(\theta^{(k)})-\hat\theta\|_2^2\right)+C r_{priv}\\
      \nonumber &=\frac{1}{2\eta}\left(\|\theta^{(k)}-\hat\theta\|_2^2-\|\theta^{(k+1)}-\hat\theta-\eta N_k\|_2^2\right)+C r_{priv}\\
 \nonumber&=\frac{1}{2\eta}\left(\|\theta^{(k)}-\hat\theta\|_2^2-\|\theta^{(k+1)}-\hat\theta\|_2^2-\eta^2\|N_k\|_2^2+2\eta\langle N_k,\theta^{(k+1)}-\hat\theta\rangle\right)+C r_{priv}\\
    &\leq \frac{1}{2\eta}\left(\|\theta^{(k)}-\hat\theta\|_2^2-\|\theta^{(k+1)}-\hat\theta\|_2^2\right)+\|N_k\|_2\|\theta^{(k+1)}-\hat\theta\|_2+C r_{priv}.
\end{align*}
Summing over $k$, we obtain
\begin{align}
\sum_{k=1}^{K_0}\left(\mathcal{L}_n(\theta^{(k)})-\mathcal{L}_n(\hat\theta)\right) & \leq \frac{1}{2\eta}\left(\|\theta^{(0)}-\hat\theta\|_2^2-\|\theta^{(K_0)}-\hat\theta\|_2^2\right)+K_0\max_{k < K_0}\|N_k\|_2\|\theta^{(k+1)}-\hat\theta\|_2+K_0Cr_{priv} \notag \\
     &\leq \frac{1}{2\eta}\|\theta^{(0)}-\hat\theta\|_2^2+K_0 r_{priv}(\|\theta^{(0)}-\hat\theta\|_2+ 1)+K_0C r_{priv}, \label{eq:prop1.4}
\end{align}
where the last inequality is a consequence of  Lemma \ref{lem:improved_param_NGD}. 

By inequality~\eqref{eq:prop1.1}, we have
\begin{equation*}
\Loss_n(\theta^{(K_0)}) \le \Loss_n(\theta^{(k)}) + CK_0 r_{priv}
\end{equation*}
for all $k < K_0$, so
\begin{equation*}
K_0\left(\mathcal{L}_n(\theta^{(K_0)}) - \mathcal{L}_n(\hat{\theta})\right) \le \sum_{k=1}^{K_0} \left(\mathcal{L}_n(\theta^{(k)}) - \mathcal{L}_n(\hat{\theta})\right) + CK_0^2 r_{priv}.
\end{equation*}
Combined with inequality~\eqref{eq:prop1.4}, this gives
\begin{align*}
\mathcal{L}_n(\theta^{(K_0)}) - \mathcal{L}_n(\hat{\theta}) & \le \frac{1}{K_0} \sum_{k=1}^{K_0} \left(\mathcal{L}_n(\theta^{(k)}) - \mathcal{L}_n(\hat{\theta})\right) + CK_0 r_{priv}\\
& \le \frac{1}{K_0} \frac{1}{2\eta} \|\theta^{(0)} - \hat{\theta}\|_2^2 + r_{priv} (\|\theta^{(0)} - \hat{\theta}\|_2+1) + (K_0+1)C r_{priv} \\
& \le \frac{1}{K_0} \frac{1}{2\eta} R^2 + (R + C(K_0+1) + 1) r_{priv} \\
& \leq \Delta,
\end{align*}
where we have taken $K_0=\frac{R^2}{\eta \Delta}$ and
\begin{equation*}
n\geq \frac{2(R+C(K_0+1)+1)\{4\sqrt{p}+2\sqrt{2\log(K_0/\xi_0)}\}B\sqrt{K_0}}{\Delta\mu }.
\end{equation*}
%
%
This last inequality establishes the desired result.

\subsubsection{Supporting lemmas}

\begin{lemma}
\label{lem:improved_param_NGD}
Suppose $\mathcal{L}_n$ is convex and $\tau_2$-smooth and let $\theta^{(k)}$ denote the updates \eqref{eq:NGD}, with $\eta \le \frac{1}{2\tau_2}$. Let $\|\theta^{(0)}-\hat\theta\|_2=R$. Also suppose inequality~\eqref{EqnNkBd} holds and $ \left(K_0 C'+2(K_0-1)\eta\right)r_{priv} \leq 1$, where
$C'=2\eta R+2\eta^2(B+1)$.
Then
\begin{equation*}
      \|\theta^{(k+1)}-\hat\theta\|_2^2 \leq R^2 +\left((k+1)C'r_{priv}+2k\eta\right)r_{priv} \le R^2 + 1,
\end{equation*}
for all $k < K_0$.
\end{lemma}
\begin{proof}
We first note that since $\mathcal{L}_n$ is convex and $\tau_2$-smooth, we have 
\begin{equation}
\label{eq:lem_prop1.1}
\langle \nabla\mathcal{L}_n(\theta_1)-\nabla\mathcal{L}_n(\theta_2),\theta_1-\theta_2\rangle \geq \frac{1}{2\tau_2}\|\nabla\mathcal{L}_n(\theta_1)-\nabla\mathcal{L}_n(\theta_2)\|_2^2,\quad \forall \theta_1,\theta_2\in\Theta\subseteq\mathbb{R}^p.    
\end{equation}
Applying inequality~\eqref{eq:lem_prop1.1} to the pair $(\theta^{(k)}, \hat{\theta})$, we then have
\begin{equation*}
\frac{1}{2\tau_2} \|\nabla \Loss_n(\theta^{(k)})\|_2^2 \le \langle \nabla \Loss_n(\theta^{(k)}), \theta^{(k)} - \thetahat \rangle.
\end{equation*}
Then using the Cauchy-Schwarz inequality and the fact that $\eta\leq \frac{1}{2\tau_2}$, we obtain
\begin{align}
    \nonumber
    \|\theta^{(k+1)}-\hat\theta\|_2^2&= \|\theta^{(k)}-\eta\nabla\mathcal{L}_n(\theta^{(k)})+\eta N_k-\hat\theta\|_2^2\\
    \nonumber& =  \|\theta^{(k)}-\hat\theta\|_2^2-2\eta\langle \theta^{(k)}-\hat\theta,\nabla\mathcal{L}_n(\theta^{(k)})- N_k\rangle+\eta^2\|\nabla\mathcal{L}_n(\theta^{(k)})- N_k\|_2^2\\
    \nonumber &\leq \|\theta^{(k)}-\hat\theta\|_2^2-\frac{\eta}{\tau_2}\|\nabla\mathcal{L}_n(\theta^{(k)})\|_2^2+2\eta\| \theta^{(k)}-\hat\theta\|\|N_k\|_2+\eta^2\|\nabla\mathcal{L}_n(\theta^{(k)})- N_k\|_2^2\\
     \nonumber& \leq \|\theta^{(k)}-\hat\theta\|_2^2-\frac{\eta}{\tau_2}\|\nabla\mathcal{L}_n(\theta^{(k)})\|_2^2+2\eta\| \theta^{(k)}-\hat\theta\|\|N_k\|_2\\
    \nonumber & \quad \quad +\eta^2\|\nabla\mathcal{L}_n(\theta^{(k)})\|_2^2+
     2\eta^2\|\nabla\mathcal{L}_n(\theta^{(k)})\|_2\|N_k\|_2+\eta^2\| N_k\|_2^2\\
     &\leq\|\theta^{(k)}-\hat\theta\|_2^2+2\eta\| \theta^{(k)}-\hat\theta\|_2 \|N_k\|_2-\frac{\eta}{2\tau_2}\|\nabla\mathcal{L}_n(\theta^{(k)})\|_2^2+2\eta^2 B\|N_k\|_2+\eta^2\|N_k\|_2^2.
     \label{eq:lem_prop1.2}
\end{align}
%
It is easy to see from the condition~\eqref{EqnNkBd} and inequality~\eqref{eq:lem_prop1.2} that
%
\begin{equation}
\label{eq:lem_prop1.3}
    \|\theta^{(1)}-\hat\theta\|_2^2\leq R^2+2\eta Rr_{priv}+2\eta^2Br_{priv}+\eta^2r_{priv}^2\leq R^2+C'r_{priv},
\end{equation}
where $C'= 2\eta R+2\eta^2 (B+1)\geq 2\eta R+2\eta^2 B+\eta^2r_{priv}$.
Furthermore, inequalities~\eqref{eq:lem_prop1.2} and~\eqref{eq:lem_prop1.3} also imply that
\begin{equation*}
    \begin{split}
        \|\theta^{(2)}-\hat\theta\|_2^2&\leq  \|\theta^{(1)}-\hat\theta\|_2^2+2\eta  \|\theta^{(1)}-\hat\theta\|_2r_{priv}+2\eta^2Br_{priv}+\eta^2r_{priv}^2\\
        &\leq (R^2+C'r_{priv})+2\eta r_{priv}(R+\sqrt{C'r_{priv}})+2\eta^2Br_{priv}+\eta^2r_{priv}^2\\
        &\leq R^2+2C' r_{priv}+2\eta r_{priv}\sqrt{C'r_{priv}}\\
        &\leq R^2+(2C'+2\eta)r_{priv},
    \end{split}
\end{equation*}
where the last inequality uses the fact that $C' r_{priv}\leq 1$. A similar argument shows that
\begin{equation*}
    \begin{split}
        \|\theta^{(3)}-\hat\theta\|_2^2&\leq  \|\theta^{(2)}-\hat\theta\|_2^2+2\eta  \|\theta^{(2)}-\hat\theta\|_2 r_{priv}+2\eta^2B r_{priv}+\eta^2 r_{priv}^2\\
        &\leq (R^2+2C' r_{priv})+2\eta r_{priv}+2\eta  r_{priv}(R+\sqrt{(2C'+2\eta) r_{priv}})+2\eta^2B r_{priv}+\eta^2 r_{priv}^2\\
        &\leq R^2 +3C' r_{priv}+2\eta r_{priv}(1+\sqrt{(2C'+2\eta) r_{priv}})\\
        &\leq R^2 +3C' r_{priv}+4\eta r_{priv},
    \end{split}
\end{equation*}
where the last inequality uses the assumption $ (2C'+2\eta) r_{priv} \leq 1$. More generally, one can recursively verify that as long as $ \left((k+1)C'+2k\eta\right) r_{priv} \leq 1$, we also have
\begin{equation*}
      \|\theta^{(k+1)}-\hat\theta\|_2^2 \leq R^2 +\left((k+1)C' r_{priv}+2k\eta\right) r_{priv}.
\end{equation*}
\end{proof}


\section{Proofs for noisy Newton's method}

In this appendix, we provide proofs of Theorems~\ref{thm:NNewton} and~\ref{thm:NNewtonSC} and Proposition~\ref{ThmNewtonGlobal}.


\subsection{Proof of Theorem \ref{thm:NNewton}}
\label{AppThmNewton}

We begin by presenting the main argument, followed by statements and proofs of supporting lemmas.


\subsubsection{Main argument}

The arguments for part (i) are very similar to those presented for Theorem \ref{thm:NGD}. In particular,  every iteration of the algorithm is $\frac{\mu}{\sqrt{K}}$-GDP by the Gaussian mechanism, post-processing, and the composition theorem of \cite[Corollary 1]{dongetal2021}. The same composition result shows that the whole algorithm is $\mu$-GDP after $K$ iterations. This proves (i).

We now turn to the proof of (ii). Let $\tilde{W}_k=\frac{2\bar{B}\sqrt{2K}}{\mu n}W_k$ and $\tilde{Z}_k=\frac{2B\sqrt{2K}}{\mu n}Z_{k}$, and note that Lemmas~\ref{lem:maxGaussian} and~\ref{lem:norm_Gaussian_matrix} and a union bound imply that with probability at least $1-\xi$, we have  $\max_{k < K} \|Z_k\|_2\leq 4\sqrt{p}+2\sqrt{2\log(2K/\xi)}$ and $\max_{k < K} \|W_k\|_2\leq \sqrt{2p\log(4Kp/\xi)}$. We will assume that these inequalities hold and argue deterministically for the rest of the proof. 

For a sufficiently large $n$ the Neumann series formula is well defined and leads to the identity 
\begin{equation*}
    \left\{\nabla^2\mathcal{L}_n(\theta^{(k)})+\tilde{W}_k\right\}^{-1}=  \left\{\nabla^2\mathcal{L}_n(\theta^{(k)})\right\}^{-1}\sum_{j=0}^\infty  \left[-\tilde{W}_k \left\{\nabla^2\mathcal{L}_n(\theta^{(k)})\right\}^{-1}\right]^{j}.
\end{equation*}
Hence, 
\begin{equation*}
    \begin{split}
        \theta^{(k+1)}&=\theta^{(k)}-\eta \left\{\nabla^2\mathcal{L}_n(\theta^{(k)})\right\}^{-1}\sum_{j=0}^\infty  \left[-\tilde{W}_k \left\{\nabla^2\mathcal{L}_n(\theta^{(k)})\right\}^{-1}\right]^{j}\left\{\nabla\mathcal{L}_n(\theta^{(k)})+\tilde{Z}_k\right\}\\
        &=\theta^{(k)}- \eta\left\{\nabla^2\mathcal{L}_n(\theta^{(k)})\right\}^{-1}\nabla\mathcal{L}_n(\theta^{(k)})+ \eta\tilde{N}_k,
    \end{split}
\end{equation*}
where 
\begin{equation}
\label{EqnNoiseNewton}
    \tilde{N}_k=\left\{\nabla^2\mathcal{L}_n(\theta^{(k)})\right\}^{-1}\left(\tilde{Z}_k+\sum_{j=1}^\infty  \left[-\tilde{W}_k \left\{\nabla^2\mathcal{L}_n(\theta^{(k)})\right\}^{-1}\right]^{j}\left\{\nabla\mathcal{L}_n(\theta^{(k)})+\tilde{Z}_k\right\}\right).
\end{equation}

Lemma~\ref{LemInduct} is our main workhorse. The bound~\eqref{EqnGradQuad}, together with the starting value condition $\|\nabla\mathcal{L}_n(\theta^{(0)})\|_2\leq \frac{\tau_1^2}{L}$, implies that
\begin{equation}
\label{EqnNewtonContraction}
\frac{L}{2\tau_1^2}\|\nabla\mathcal{L}_n(\theta^{(k)})\|_2 < \left(\frac{L}{2\tau_1^2}\|\nabla\mathcal{L}_n(\theta^{(0)})\|_2\right)^{2^k}+3C\tilde{r}_{priv},
\end{equation}
for all $k \le K$, $\tilde{r}_{priv} := \frac{\sqrt{Kp\log(Kp/\xi)}}{\mu n}$. Indeed, we can prove inequality~\eqref{EqnNewtonContraction} by induction: The base case $k = 1$ holds by inequality~\eqref{EqnGradQuad}, and if we assume the bound holds for some $k \ge 1$, then
\begin{align*}
    \frac{L}{2\tau_1^2}\|\nabla\mathcal{L}_n(\theta^{(k+1)})\|_2&\leq \left(\frac{L}{2\tau_1^2}\|\nabla\mathcal{L}_n(\theta^{(k)})\|_2\right)^2+C\tilde{r}_{priv} \\
    & \leq  \left\{\left(\frac{L}{2\tau_1^2}\|\nabla\mathcal{L}_n(\theta^{(0)})\|_2\right)^{2^k}+3C\tilde{r}_{priv}\right\}^2+C\tilde{r}_{priv} \\
    &\leq \left(\frac{L}{2\tau_1^2}\|\nabla\mathcal{L}_n(\theta^{(0)})\|_2\right)^{2^{k+1}}+C\tilde{r}_{priv}\left(\frac{3}{2}+9C\tilde{r}_{priv}\right)+C\tilde{r}_{priv}\\
& < \left(\frac{L}{2\tau_1^2}\|\nabla\mathcal{L}_n(\theta^{(0)})\|_2\right)^{2^{k+1}}+3C\tilde{r}_{priv},
\end{align*}
where the first inequality comes from inequality~\eqref{EqnGradQuad}; the second inequality follows by the induction hypothesis; and the last inequality uses the fact that the sample size condition $n=\Omega\left(\frac{\sqrt{Kp \log(Kp/\xi)}}{\mu}\right)$ guarantees that $C\tilde{r}_{priv} < \frac{1}{18}$. This completes the inductive step.


Note that inequality~\eqref{EqnNewtonContraction} implies that
\begin{equation*}
\frac{L}{2\tau_1^2}\|\nabla\mathcal{L}_n(\theta^{(K)})\|_2 < \left(\frac{1}{2}\right)^{2^K}+3C\tilde{r}_{priv}.
\end{equation*}
Hence, for $K\geq \frac{1}{\log (2)}\log\left(\frac{\log(C\tilde{r}_{priv})}{\log(1/2)}\right)$, we have
\begin{equation*}
4C\tilde{r}_{priv} \ge \frac{L}{2\tau_1^2}\|\nabla\mathcal{L}_n(\theta^{(K)})\|_2 \ge \frac{L}{2\tau_1^2} \cdot 2\tau_1 \|\theta^{(K)} - \thetahat\|_2,
\end{equation*}
where the last inequality holds by LSC. Rearranging yields the desired result.

\subsubsection{Supporting lemmas}

In the statements and proofs of these lemmas, we work under the assumption that $\max_{k < K} \|Z_k\|_2\leq 4\sqrt{p}+2\sqrt{2\log(2K/\xi)}$ and $\max_{k < K} \|W_k\|_2\leq \sqrt{2p\log(4Kp/\xi)}$, in addition to the assumptions stated in Theorem~\ref{thm:NNewton}.

\begin{lemma}
\label{LemInduct}
For all $0 \le k < K$, we have
\begin{itemize}
\item[(a)] $\|\nabla \Loss_n(\theta^{(k)})\|_2 \le \min\left\{\tau_1 r, \frac{\tau_1^2}{L}\right\}$,
\item[(b)] $\|\theta^{(k)} - \thetahat\|_2 \le r$, and
\item[(c)]
\begin{equation}
\label{EqnGradQuad}
\frac{L}{2\tau_1^2} \|\nabla \Loss_n(\theta^{(k+1)}\|_2 \le \left(\frac{L}{2\tau_1^2} \|\nabla \Loss_n(\theta^{(k)})\|_2\right)^2 + C\tilde{r}_{priv},
\end{equation}
where $\tilde{r}_{priv} := \frac{\sqrt{Kp\log(Kp/\xi)}}{\mu n}$ and $C$ depends on $\tau_1$, $B$ and  $\bar{B}$.
\end{itemize}
In addition, $\|\theta^{(K)} - \thetahat\|_2 \le r$.
\end{lemma}

\begin{proof}
We induct on $k$. For the inductive hypothesis, consider the case $k = 0$. Note that (a) holds by assumption. Lemma~\ref{LemSmallGrad} then implies that (b) holds, and then Lemma~\ref{LemNoiseNewton} implies that (c) holds, as well.

Turning to the inductive step, assume that (a), (b), and (c) all hold for some $k \ge 0$. We will show that $\|\nabla \Loss_n(\theta^{(k+1)})\|_2 \le  \min\left\{\tau_1 r, \frac{\tau_1^2}{L}\right\}$; then Lemmas~\ref{LemSmallGrad} and~\ref{LemNoiseNewton} imply that statements (b) and (c) hold, as well. Note that by the inductive hypothesis, we have
\begin{equation*}
\frac{L}{2\tau_1^2} \|\nabla \Loss_n(\theta^{(k)})\|_2 \le \frac{1}{2}.
\end{equation*}
Then by (a) and (c), we have
\begin{align}
\label{EqnDuck}
\frac{L}{2\tau_1^2} \|\nabla \Loss_n(\theta^{(k+1)}\|_2 & \le \left(\frac{L}{2\tau_1^2} \|\nabla \Loss_n(\theta^{(k)})\|_2\right)^2 + C\tilde{r}_{priv} \notag \\
& \le \frac{1}{2} \left(\frac{L}{2\tau_1^2} \|\nabla \Loss_n(\theta^{(k)})\|_2\right) + C\tilde{r}_{priv} \notag \\
& \le \frac{1}{2} \cdot \frac{L}{2\tau_1^2} \min\left\{\tau_1 r, \frac{\tau_1^2}{L}\right\} + C\tilde{r}_{priv} \notag \\
& \le \frac{L}{2\tau_1^2} \min\left\{\tau_1 r, \frac{\tau_1^2}{L}\right\},
\end{align}
using the sample size assumption $n = \Omega\left(\frac{\sqrt{Kp \log(Kp/\xi)}}{\mu}\right)$. This completes the inductive step.

Finally, note that the argument used to establish inequality~\eqref{EqnDuck} above shows that we can also deduce statements (a) and (b) for $k = K$.
\end{proof}

\begin{lemma}
\label{LemSmallGrad}
Suppose $\|\nabla \Loss_n(\theta)\|_2 < 2\tau_1 r$. Then $\|\theta - \thetahat\|_2 \le r$.
\end{lemma}

\begin{proof}
We prove the contrapositive. Suppose $\|\theta - \thetahat\|_2 > r$, and let $\thetatil$ denote the point on the boundary of $\mathcal{B}_r(\thetahat)$ lying on the segment between $\theta$ and $\thetahat$. By the LSC condition, we have
\begin{equation*}
\langle \nabla \Loss_n(\thetatil), \thetatil - \thetahat \rangle \ge 2\tau_1 \|\thetatil - \thetahat\|_2^2.
\end{equation*}
Rearranging and defining $v = \frac{\thetatil - \thetahat}{\|\thetatil - \thetahat\|_2}$, we then have
\begin{equation*}
\langle \nabla \Loss_n(\thetatil), v \rangle \ge 2\tau_1 r.
\end{equation*}
Further note that if we define $f(t) = \langle \nabla \Loss_n(\thetahat + tv), v\rangle$ for $t \ge 0$, then $f'(t) = v^T \nabla^2 \Loss_n(\thetahat + tv) v \ge 0$, so $f(t)$ is increasing. This implies that
\begin{equation*}
\|\nabla \Loss_n(\theta)\|_2 \ge \langle \nabla \Loss_n(\theta), v \rangle \ge \langle \nabla \Loss_n(\thetatil), v \rangle \ge 2\tau_1 r,
\end{equation*}
giving the desired result.
\end{proof}

\begin{lemma}
\label{LemNoiseNewton}
Suppose $\theta^{(k)} \in \mathcal{B}_r(\thetahat)$. Then inequality~\eqref{EqnGradQuad} holds.
\end{lemma}

\begin{proof}
We begin by establishing a bound on $\|\tilde{N}_k\|_2$.
By LSC, we have $\|\left\{\nabla^2\mathcal{L}_n(\theta^{(k)})\right\}^{-1}\|_2 \le \frac{1}{2\tau_1}$. Furthermore, by our sample size assumption, we have $\max_{k < K} \|\tilde{W}_k\|_2\leq \tau_1$. Then the series expansion~\eqref{EqnNoiseNewton} and the bounds on $\|Z_k\|_2$ and $\|W_k\|_2$ imply that for all $k < K$, we have
\begin{align}
\label{EqnNk}
    \|\tilde{N}_k\|_2 & \leq \frac{1}{2\tau_1}\|\tilde{Z}_k\|_2+\sum_{j=1}^\infty  \frac{\|\tilde{W}_k\|^j_2}{(2\tau_1)^{(j+1)}}\left\{\|\nabla\mathcal{L}_n(\theta^{(k)})\|_2+\|\tilde{Z}_k\|_2\right\} \notag \\
& \le \frac{1}{2\tau_1}\|\tilde{Z}_k\|_2+\frac{1}{2\tau_1}\frac{\|\tilde{W}_k\|_2}{2\tau_1-\|\tilde{W}_k\|_2}\left\{\|\nabla\mathcal{L}_n(\theta^{(k)})\|_2+\|\tilde{Z}_k\|_2\right\} \notag \\
& \le \frac{2B\sqrt{2K}}{2\mu n\tau_1}\|Z_k\|_2+\frac{1}{2\tau_1}\frac{\|\tilde{W}_k\|_2}{2\tau_1 - \|\tilde{W}_k\|_2}\left\{B+\frac{2B\sqrt{2K}}{\mu n}\|Z_k\|_2\right\} \notag \\
& \le \frac{B\sqrt{2K}(4\sqrt{p}+2\sqrt{2\log(2K/\xi)})}{\mu n\tau_1} \notag \\
& \quad\quad+\frac{2\bar{B}\sqrt{2K}}{2\tau_1\mu n}\frac{\|W_k\|_2}{\tau_1}\left(B+\frac{2B\sqrt{2K}(4\sqrt{p}+2\sqrt{2\log(2K/\xi)})}{\mu n}\right) \notag \\
& \le C_0 \frac{B\sqrt{K}\{\sqrt{p}+\sqrt{\log(2K/\xi)}\}+\bar{B} B\tau_1^{-1}\sqrt{Kp\log(Kp/\xi)}}{\mu n\tau_1} = C'\tilde{r}_{priv},
\end{align}
where $C_0 > 0$ is a constant.

To establish inequality~\eqref{EqnGradQuad}, we write $\theta^{k+1}=\theta^{(k)}+\Delta\theta^{(k)}+\tilde{N}_k$, where
\begin{equation*}
\Delta\theta^{(k)}=- \left\{\nabla^2\mathcal{L}_n(\theta^{(k)})\right\}^{-1}\nabla\mathcal{L}_n(\theta^{(k)}).
\end{equation*}
Using Condition~\ref{ass:Lipschitz_Hessian}, the LSC condition, and the triangle inequality, we then obtain
\begin{equation*}
\begin{split}
&\|\nabla\mathcal{L}_n(\theta^{(k+1)})\|_2 \\
&=  \left\|\nabla\mathcal{L}_n\left(\theta^{(k)}+\Delta\theta^{(k)}+\tilde{N}_k\right)\right\|_2\\
&= \bigg\|\nabla\mathcal{L}_n\left(\theta^{(k)}+\Delta\theta^{(k)}+\tilde{N}_k\right)-\nabla\mathcal{L}_n(\theta^{(k)})-\nabla^2\mathcal{L}_n(\theta^{(k)})\left(\Delta\theta^{(k)}-\tilde{N}_k+\tilde{N}_k\right)\bigg\|_2\\
&= \left\|\int_0^1\left\{\nabla^2\mathcal{L}_n\left(\theta^{(k)}+t(\Delta\theta^{(k)}+\tilde{N}_k)\right)-\nabla^2\mathcal{L}_n(\theta^{(k)})\right\}(\Delta\theta^{(k)}+\tilde{N}_k)\mathrm{d}t+\nabla^2\mathcal{L}_n(\theta^{(k)})\tilde{N}_k\right\|_2\\
& \le \left\|\int_0^1 Lt\|\Delta \theta^{(k)} - \tilde{N}_k\|_2^2 \mathrm{d}t +\nabla^2\mathcal{L}_n(\theta^{(k)})\tilde{N}_k\right\|_2\\
&\leq \frac{L}{2}\|\Delta\theta^{(k)}+\tilde{N}_k\|_2^2+\|\nabla^2\mathcal{L}_n(\theta^{(k)})\tilde{N}_k\|_2\\
&\leq \frac{L}{2}\|\Delta\theta^{(k)}\|_2^2+\left(L\|\Delta\theta^{(k)}\|_2+\bar{B}\right)\|\tilde{N}_k\|_2+\frac{L}{2}\|\tilde{N}_k\|_2^2\\
&\leq  \frac{L}{2\tau_1^2}\|\nabla\mathcal{L}_n(\theta^{(k)})\|_2^2+\left(L\|\Delta\theta^{(k)}\|_2+\bar{B}\right)\|\tilde{N}_k\|_2+\frac{L}{2}\|\tilde{N}_k\|_2^2.
\end{split}
\end{equation*}
Since LSC guarantees that $\|\Delta\theta^{(k)}\|_2\leq \frac{B}{\tau_1}$, the desired result follows from the bound~\eqref{EqnNk} on $\|\tilde{N}_k\|_2$ and the sample size condition.
\end{proof}

\subsection{Proof of Theorem~\ref{thm:NNewtonSC}}
\label{AppThmNewtonSC}

We begin by presenting the main argument, followed by the proofs of the supporting lemmas.

\subsubsection{Main argument}

It is easy to see that (i) follows by the same arguments as in Theorem~\ref{thm:NNewton}.
For (ii), we borrow arguments from the proof of Theorem 3 of \cite{sun2019generalized} and follow a similar approach as in our analysis of noisy Newton's method under LSC in Theorem~\ref{thm:NNewton}. We relegate the statements and proofs of the key auxiliary lemmas to Appendix~\ref{AppLambda}.

As in the proof of Theorem~\ref{thm:NNewton}, we first control the additive noise introduced in the first $K$ steps via a union bound, and then argue deterministically for the rest of the proof. Let $\tilde{W}_k=\frac{2\bar{B}\sqrt{2K}}{\mu n}W_k$ and $\tilde{Z}_k=\frac{2B\sqrt{2K}}{\mu n}Z_{k}$. Then Lemmas~\ref{lem:maxGaussian} and~\ref{lem:norm_Gaussian_matrix} and a union bound imply that with probability at least $1-\xi$, we have  $\max_{k < K} \|Z_k\|_2\leq 4\sqrt{p}+2\sqrt{2\log(2K/\xi)}$ and $\max_{k < K} \|W_k\|_2\leq \sqrt{2p\log(4Kp/\xi)}$. We will assume that these inequalities hold for the rest of the proof.



Define
\begin{equation*}
\lambda_k := \lambda_{\min}^{-1/2} \left(\nabla^2 \Loss_n(\theta^{(k)})\right) \lambda(\theta^{(k)}) = \lambda_{\min}^{-1/2} \left(\nabla^2 \Loss_n(\theta^{(k)})\right) \|\nabla \Loss_n(\theta^{(k)})\|_{\nabla^2 \Loss_n(\theta^{(k)})^{-1}}.
\end{equation*}
Lemma~\ref{LemLambdaExp} shows that the sequence $\{\lambda_k\}$ (approximately) converges at a quadratic rate. We now show that this implies a similar convergence statement about the iterates $\|\theta^{(k)} - \thetahat\|_2$. Applying Lemma~\ref{LemSCgrad} with $x = \thetahat$ and $y = \theta^{(k)}$, we have
\begin{equation*}
\frac{1-\exp(-\gamma \|\theta^{(k)} - \thetahat\|_2)}{\gamma \|\theta^{(k)} - \thetahat\|^2_2} \|\thetahat^{(k)} - \thetahat\|_{\nabla^2 \Loss_n(\theta^{(k)})} \le \|\nabla \Loss_n(\theta^{(k)})\|_{\nabla^2 \Loss_n(\theta^{(k)})^{-1}} = \lambda(\theta^{(k)}).
\end{equation*}
By Lemma~\ref{LemInductSC}, we have $\|\theta^{(k)} - \thetahat\|_2 \le \frac{1}{\gamma}$. Thus, the left-hand expression is lower-bounded by $\frac{1}{2} \|\thetahat^{(k)} - \thetahat\|_{\nabla^2 \Loss_n(\theta^{(k)})}$, implying that
\begin{equation*}
\lambda_{\min}\left(\nabla^2 \Loss_n(\theta^{(k)})\right)^{1/2} \frac{\|\thetahat^{(k)} - \thetahat\|_2}{2} \le \frac{1}{2}\|\thetahat^{(k)} - \thetahat\|_{\nabla^2 \Loss_n(\theta^{(k)})} \le \lambda(\theta^{(k)}).
\end{equation*}
Hence, by Lemma~\ref{LemLambdaExp}, we have
\begin{equation*}
\|\thetahat^{(k)} - \thetahat\|_2 \le 2\lambda_k \le \frac{1}{\gamma} \left((2\gamma \lambda_0)^{2^k} + 8\gamma \zeta \sqrt{\frac{\bar{B}}{\tau_{1,2/\gamma}}}\right) \le \frac{1}{\gamma} \left(\frac{1}{8}\right)^{2^k} + 8 \zeta \sqrt{\frac{\bar{B}}{\tau_{1,2/\gamma}}},
\end{equation*}
for all $k \le K$, where $\zeta = C'\frac{\bar{B} B\sqrt{Kp\log(Kp/\xi)}}{\mu n\tau_{1,2/\gamma}^2}$ is as defined in Lemma~\ref{LemInductSC}. (Note that the condition $8\gamma \zeta \sqrt{\frac{\bar{B}}{\tau_{1,2/\gamma}}} \le \frac{1}{4}$ holds by our sample size assumption.) Finally, taking $K = \Omega(\log \log n)$ yields the desired error bound.

\subsubsection{Supporting lemmas}
\label{AppLambda}

In the statements and proofs of the following lemmas, we are working under the
assumption that $\max_{k < K} \|Z_k\|_2\leq 4\sqrt{p}+2\sqrt{2\log(2K/\xi)}$ and $\max_{k < K} \|W_k\|_2\leq \sqrt{2p\log(4Kp/\xi)}$, in addition to the assumptions stated in Theorem~\ref{thm:NNewtonSC}.

\begin{lemma}
\label{LemInductSC}
For all $0 \le k < K$, we have
\begin{itemize}
\item[(a)] $\lambda_k \le \frac{1}{16\gamma}$,
\item[(b)] $\|\theta^{(k)} - \thetahat\|_2 \le \frac{1}{\gamma}$, and
\item[(c)]
\begin{equation}
\label{EqnLambdaKIts}
\lambda_{k+1} \le 2\gamma \left(\lambda_k + \zeta \sqrt{\frac{\bar{B}}{\tau_{1,2/\gamma}}}\right)^2 + \exp\left(\frac{1}{8}\right) \zeta \sqrt{\frac{\bar{B}}{\tau_{1,2/\gamma}}},
\end{equation}
where $\zeta := C'\frac{\bar{B} B\sqrt{Kp\log(Kp/\xi)}}{\mu n\tau_{1,2/\gamma}^2}$.
\end{itemize}
In addition, $\|\theta^{(K)} - \thetahat\|_2 \le \frac{1}{\gamma}$.
\end{lemma}

\begin{proof}
We induct on $k$. For the inductive hypothesis, consider the case $k = 0$. Note that (a) holds by assumption. Lemma~\ref{LemNoiseNewtonSC} then implies that (b) and (c) hold, as well.

Turning to the inductive step, assume that (a), (b), and (c) all hold for some $k \ge 0$. It clearly suffices to show that $\lambda_{k+1} \le \frac{1}{16\gamma}$; then Lemma~\ref{LemNoiseNewtonSC} implies that statements (b) and (c) also hold. Note that by our sample size assumption $n = \Omega\left(\frac{\sqrt{Kp \log(Kp/\xi)}}{\mu}\right)$, we can make $\zeta \sqrt{\frac{\bar{B}}{\tau_{1,2/\gamma}}} \le \frac{1}{32\gamma}$,
so statement (c) of the inductive hypothesis implies that
\begin{equation*}
\lambda_{k+1} \le 2\gamma \left(\frac{1}{8\gamma}\right)^2 + \frac{1}{32\gamma} \exp\left(\frac{1}{8}\right) \le \frac{1}{16\gamma},
\end{equation*}
as wanted.

Finally, note that the argument above shows that we can also deduce statement (a) for $k = K$, and statement (b) then follows by Lemma~\ref{LemNoiseNewtonSC}.
\end{proof}

\begin{lemma}
\label{LemNoiseNewtonSC}
Suppose $\lambda_k \le \frac{1}{16\gamma}$. Then $\|\theta^{(k)} - \thetahat\|_2 \le \frac{1}{\gamma}$ and inequality~\eqref{EqnLambdaKIts} holds.
\end{lemma}

\begin{proof}
We first show that $\theta^{(k)} \in \mathcal{B}_{1/\gamma}(\thetahat)$. By Lemma~\ref{LemSCgrad}, we have
\begin{align*}
\frac{1-\exp(-\gamma \|\theta^{(k)} - \thetahat\|_2)}{\gamma \|\theta^{(k)} - \thetahat\|_2} \cdot \lambda_{\min}(\theta^{(k)})^{1/2} \|\theta^{(k)} - \thetahat\|_2 & \le \frac{1-\exp(-\gamma \|\theta^{(k)} - \thetahat\|_2)}{\gamma \|\theta^{(k)} - \thetahat\|_2} \cdot \|\theta^{(k)} - \thetahat\|_{\nabla^2 \Loss_n(\theta^{(k)})} \\
& \le \|\nabla \Loss_n(\theta^{(k)})\|_{\nabla^2 \Loss_n(\theta^{(k)})^{-1}},
\end{align*}
implying that
\begin{equation*}
\frac{1-\exp(-\gamma \|\theta^{(k)} - \thetahat\|_2)}{\gamma} \le \lambda_k \le \frac{1}{16 \gamma}.
\end{equation*}
Rearranging, we see that
\begin{equation*}
\gamma \|\theta^{(k)} - \thetahat\|_2 \le \log \left(\frac{16}{15}\right) \le 1,
\end{equation*}
implying the desired result.

Turning to inequality~\eqref{EqnLambdaKIts}, let $v = \theta^{(k+1)} - \theta^{(k)}$. We claim that under the assumptions, we have
\begin{equation}
\label{EqnHk1}
\|v\|_{\nabla^2 \Loss_n(\theta^{(k)})} \le \lambda(\theta^{(k)}) + \sqrt{\bar{B}} \zeta.
\end{equation}
Indeed, the same logic applied in the proof of Lemma~\ref{LemNoiseNewton}, together with the fact that $\theta^{(k)} \in \mathcal{B}_{1/\gamma}(\thetahat)$ and the lower bound on the Hessian guaranteed by Lemma~\ref{LemSCMult}, implies that $\|\tilde{N}_k\|_2 \le \zeta$, where
\begin{equation*}
\tilde{N}_k = v + \left\{\nabla^2\mathcal{L}_n(\theta^{(k)})\right\}^{-1}\nabla\mathcal{L}_n(\theta^{(k)}).
\end{equation*}
Hence, by the triangle inequality,
\begin{equation*}
\|v\|_{\nabla^2 \Loss_n(\theta^{(k)})} \le \lambda(\theta^{(k)}) + \sqrt{\bar{B}} \|\tilde{N}_k\|_2 \le \lambda(\theta^{(k)}) + \sqrt{\bar{B}} \zeta.
\end{equation*}
Now let
\begin{align*}
G_k & := \int_0^1 \left(\nabla^2 \Loss_n(\theta^{(k)} + tv) - \nabla^2 \Loss_n(\theta^{(k)})\right) dt, \\
H_k & := \nabla^2 \Loss_n(\theta^{(k)})^{-1/2} G_k \nabla^2 \Loss_n(\theta^{(k)})^{-1/2}.
\end{align*}
By Lemma~\ref{LemSCHmat}, we have
\begin{equation}
\label{EqnHk2}
\|H_k\|_2 \le \left(\frac{3}{2} + \frac{\gamma \|v\|_2}{3}\right)\gamma \|v\|_2\exp\left(\gamma \|v\|_2\right).
\end{equation}
Furthermore,
\begin{align*}
\nabla \Loss_n(\theta^{(k+1)}) & = \nabla \Loss_n(\theta^{(k+1)}) - \nabla \Loss_n(\theta^{(k)}) - \nabla^2 \Loss_n(\theta^{(k)}) (v - \tilde{N}_k) \\
& = G_k v + \nabla^2 \Loss_n(\theta^{(k)}) \tilde{N}_k.
\end{align*}
Hence,
\begin{align}
\label{EqnHk3}
\|\nabla \Loss_n(\theta^{(k+1)})\|_{\nabla^2 \Loss_n(\theta^{(k)})^{-1}} & \le \left(v^T G_k \nabla^2 \Loss_n(\theta^{(k)})^{-1} G_k v\right)^{1/2} + \sqrt{\bar{B}} \zeta \notag \\
& = \left(v^T \nabla^2 \Loss_n(\theta^{(k)})^{1/2} H_k^2 \nabla^2 \Loss_n(\theta^{(k)})^{1/2} v \right)^{1/2} + \sqrt{\bar{B}} \zeta \notag \\
& \le \|H_k\|_2 \|v\|_{\nabla^2 \Loss_n(\theta^{(k)})} + \sqrt{\bar{B}} \zeta.
%
\end{align}
Now note that by Lemma~\ref{LemSCHess}, with $x = \theta^{(k)}$ and $y = \theta^{(k+1)}$, we have
\begin{equation*}
\lambda(\theta^{(k+1)})^2 \le \exp(\gamma \|v\|_2) \|\nabla \Loss_n(\theta^{(k+1)})\|_{\nabla^2 \Loss_n(\theta^{(k)})}^2.
\end{equation*}
Combined with inequalities~\eqref{EqnHk1}, \eqref{EqnHk2}, and~\eqref{EqnHk3}, this implies that
\begin{align}
\label{EqnLambda1}
\lambda(\theta^{(k+1)}) & \le \exp\left(\frac{\gamma}{2} \|v\|_2\right) \left(\|H_k\|_2 \|v\|_{\nabla^2 \Loss_n(\theta^{(k)})} + \sqrt{\bar{B}} \zeta\right) \notag \\
& \le \exp\left(\frac{\gamma}{2} \|v\|_2\right) \left(\frac{3}{2} + \frac{\gamma \|v\|_2}{3}\right)\gamma \|v\|_2\exp\left(\gamma \|v\|_2\right) \left(\lambda(\theta^{(k)}) + \sqrt{\bar{B}} \zeta\right) + \exp\left(\frac{\gamma}{2} \|v\|_2\right) \sqrt{\bar{B}} \zeta.
\end{align}
By inequality~\eqref{EqnHk1}, we also have
\begin{equation*}
\lambda_{\min}^{1/2} \left(\nabla^2 \Loss_n(\theta^{(k)})\right) \|v\|_2 \le \lambda(\theta^{(k)}) + \sqrt{\bar{B}} \zeta,
\end{equation*}
so rearranging,
\begin{equation}
\label{EqnLambda2}
\|v\|_2 \le \lambda_k + \lambda_{\min}^{-1/2} \left(\nabla^2 \Loss_n(\theta^{(k)})\right) \sqrt{\bar{B}} \zeta.
\end{equation}
Furthermore, Lemma~\ref{LemSCHess} also implies that
\begin{equation*}
\lambda_{\min}\left(\nabla^2 \Loss_n(\theta^{(k)})\right) \le \lambda_{\min}\left(\nabla^2 \Loss_n(\theta^{(k+1)})\right) \exp(\gamma \|v\|_2),
\end{equation*}
so
\begin{equation*}
\lambda_{\min}\left(\nabla^2 \Loss_n(\theta^{(k+1)})\right)^{-1/2} \le \lambda_{\min}\left(\nabla^2 \Loss_n(\theta^{(k)})\right)^{-1/2} \exp\left(\frac{\gamma}{2} \|v\|_2\right).
\end{equation*}
Combined with inequalities~\eqref{EqnLambda1} and~\eqref{EqnLambda2}, it follows that
\begin{align*}
\lambda_{k+1} & \le \exp\left(\frac{\gamma}{2} \|v\|_2\right) \left(\frac{3}{2} + \frac{\gamma \|v\|_2}{3}\right)\gamma \exp\left(\gamma \|v\|_2\right) \left(\lambda(\theta^{(k)}) + \sqrt{\bar{B}} \zeta\right) \\
& \qquad \qquad \cdot \left(\lambda_k + \lambda_{\min}^{-1/2} \left(\nabla^2 \Loss_n(\theta^{(k)})\right) \sqrt{\bar{B}} \zeta\right) \\
& \qquad \qquad \cdot \lambda_{\min}\left(\nabla^2 \Loss_n(\theta^{(k)})\right)^{-1/2} \exp\left(\frac{\gamma}{2} \|v\|_2\right) \\
& \qquad + \lambda_{\min}\left(\nabla^2 \Loss_n(\theta^{(k)})\right)^{-1/2} \exp\left(\gamma \|v\|_2\right) \sqrt{\bar{B}} \zeta \\
& = \gamma \exp\left(2\gamma \|v\|_2\right) \left(\frac{3}{2} + \frac{\gamma \|v\|_2}{3}\right) \left(\lambda_k + \lambda_{\min}\left(\nabla^2 \Loss_n(\theta^{(k)})\right)^{-1/2} \sqrt{\bar{B}} \zeta\right)^2 \\
& \qquad + \lambda_{\min}\left(\nabla^2 \Loss_n(\theta^{(k)})\right)^{-1/2} \exp\left(\gamma \|v\|_2\right) \sqrt{\bar{B}} \zeta.
\end{align*}
We can check that if $\gamma \|v\|_2 \le \frac{1}{8}$, then
\begin{equation*}
\exp(2\gamma \|v\|_2) \left(\frac{3}{2} + \frac{\gamma \|v\|_2}{3}\right) \le 2,
\end{equation*}
implying that
\begin{equation}
\label{EqnGinger}
\lambda_{k+1} \le 2\gamma \left(\lambda_k + \lambda_{\min}\left(\nabla^2 \Loss_n(\theta^{(k)})\right)^{-1/2} \sqrt{\bar{B}} \zeta\right)^2 + \exp\left(\frac{1}{8}\right) \lambda_{\min}\left(\nabla^2 \Loss_n(\theta^{(k)})\right)^{-1/2} \sqrt{\bar{B}} \zeta.
\end{equation}
Note that since $\|\theta^{(k)} - \thetahat\|_2 \le \frac{1}{\gamma}$ and $\hat\theta \in \mathcal{B}_{1/\gamma}(\theta_0)$ by assumption, we have $\lambda_{\min}\left(\nabla^2 \Loss_n(\theta^{(k)})\right) \ge \tau_{1, 2/\gamma}$. Combining this with inequality~\eqref{EqnGinger} yields inequality~\eqref{EqnLambdaKIts}.

Finally, note that if $\lambda_k \le \frac{1}{16\gamma}$, then by inequality~\eqref{EqnLambda2} and the sample size condition, we indeed have $\|v\|_2 \le \frac{1}{8\gamma}$.
\end{proof}

\begin{lemma}
\label{LemLambdaExp}
Suppose $8\gamma \zeta \sqrt{\frac{\bar{B}}{\tau_{1,2/\gamma}}} \le \frac{1}{4}$ and $\lambda_0 \le \frac{1}{16\gamma}$, with $\zeta$ as defined in Lemma~\ref{LemInductSC}. Then for $k \le K$, we have
\begin{equation}
\label{EqnLambdaKIter}
2\gamma \lambda_k \le (2\gamma \lambda_0)^{2^k} + 8\gamma \zeta \sqrt{\frac{\bar{B}}{\tau_{1,2/\gamma}}}.
\end{equation}
\end{lemma}

\begin{proof}
We will use induction to show that inequality~\eqref{EqnLambdaKIter} holds
for all $k \le K$. The base case, $k = 0$, is obvious. For the inductive step, assume the inequality holds for some $k \ge 0$.

By Lemma~\ref{LemInductSC}, we have
\begin{align}
\label{EqnLambdaKBd}
\lambda_{k+1} & \le 2\gamma \left(\lambda_k + \zeta \sqrt{\frac{\bar{B}}{\tau_{1,2/\gamma}}} \right)^2 + \exp\left(\frac{1}{8}\right) \zeta \sqrt{\frac{\bar{B}}{\tau_{1,2/\gamma}}} \\
& = 2\gamma \lambda_k^2 + 4\gamma \lambda_k \zeta \sqrt{\frac{\bar{B}}{\tau_{1,2/\gamma}}} + 2\gamma\zeta^2 \frac{\bar{B}}{\tau_{1,2/\gamma}} + \exp\left(\frac{1}{8}\right) \zeta \sqrt{\frac{\bar{B}}{\tau_{1,2/\gamma}}} \notag \\
& \le 2\gamma \lambda_k^2 + \zeta \sqrt{\frac{\bar{B}}{\tau_{1,2/\gamma}}} \left(\frac{1}{4} + \frac{1}{4} + \exp\left(\frac{1}{8}\right)\right) \notag \\
& = 2\gamma \lambda_k^2 + 2\zeta\sqrt{\frac{\bar{B}}{\tau_{1,2/\gamma}}},
\end{align}
using the fact that $\lambda_k \le \frac{1}{16\gamma}$ from Lemma~\ref{LemInductSC} and the assumption $2\gamma \zeta \sqrt{\frac{\bar{B}}{\tau_{1,2/\gamma}}} \le \frac{1}{4}$.

Hence,
\begin{align*}
2\gamma \lambda_{k+1} & \le 2\gamma\left(2\gamma \lambda^2_k + 2 \zeta \sqrt{\frac{\bar{B}}{\tau_{1,2/\gamma}}}\right) \\
& = (2\gamma \lambda_k)^2 + 4 \gamma \zeta \sqrt{\frac{\bar{B}}{\tau_{1,2/\gamma}}} \\
& \le \left(\left(2\gamma \lambda_0\right)^{2^k} + 8\gamma \zeta \sqrt{\frac{\bar{B}}{\tau_{1,2/\gamma}}}\right)^2 + 4\gamma \zeta \sqrt{\frac{\bar{B}}{\tau_{1,2/\gamma}}} \\
& = \left(2\gamma \lambda_0\right)^{2^{k+1}} + 8\gamma \zeta \sqrt{\frac{\bar{B}}{\tau_{1,2/\gamma}}} \left(8\gamma \zeta \sqrt{\frac{\bar{B}}{\tau_{1,2/\gamma}}} + 2\left(2\gamma \lambda_0\right)^{2^{k}}\right) +  4 \gamma \zeta \sqrt{\frac{\bar{B}}{\tau_{1,2/\gamma}}} \\
& \le \left(2\gamma \lambda_0\right)^{2^{k+1}} + 8\gamma \zeta \sqrt{\frac{\bar{B}}{\tau_{1,2/\gamma}}},
\end{align*}
where the second inequality comes from the induction hypothesis and the final inequality uses the facts that $8\gamma \zeta \sqrt{\frac{\bar{B}}{\tau_{1,2/\gamma}}} \le \frac{1}{4}$ and $2\gamma \lambda_0 \le \frac{1}{8}$.
\end{proof}

\subsection{Proof of Proposition \ref{ThmNewtonGlobal}}
\label{AppThmNewtonGlobal}

We first present the main argument, followed by statements and proofs of supporting lemmas.

\subsubsection{Main argument}


By Lemma~\ref{lem:Newton_update_self-concordance}, we know that with probability at least $1-\xi$, all iterates $\{\theta^{(k)}\}_{k=1}^K$ lie in $\Theta_0 = \mathcal{B}_{R_0}(\thetahat)$. Thus, we may apply Lemma~\ref{LemA} to the pair $(\theta^{(k)}, \theta^{(k+1)})$
to obtain
\begin{align*}
\mathcal{L}_n(\theta^{(k+1)}) & \le \mathcal{L}_n(\theta^{(k)}) + \langle \nabla \mathcal{L}_n(\theta^{(k)}), \Delta \theta^{(k)} \rangle + \frac{\tau_0}{2} \|\Delta \theta^{(k)}\|^2_{\nabla^2 \mathcal{L}_n(\theta^{(k)})},
\end{align*}
where $\tau_0 = \exp(2\gamma R_0)$ (by Lemma~\ref{LemSCMult}), $\Delta \theta^{(k)} = \theta^{(k+1)} - \theta^{(k)} = - \eta \left\{\nabla^2 \mathcal{L}_n(\theta^{(k)})\right\}^{-1} \nabla \mathcal{L}_n(\theta^{(k)}) + \eta \tilde{N}_k$, and $\tilde{N}_k$ is as defined in equation~\eqref{EqnNoiseNewton}.
Note that
\begin{align*}
\langle \nabla \mathcal{L}_n(\theta^{(k)}), \Delta \theta^{(k)} \rangle & = - \eta \langle \nabla \Loss_n(\theta^{(k)}), \{\nabla^2 \Loss_n(\theta^{(k)})\}^{-1} \nabla \Loss_n(\theta^{(k)}) \rangle + \eta \langle \nabla \Loss_n(\theta^{(k)}), \tilde{N}_k \rangle \\
& = - \eta \|\{\nabla^2 \mathcal{L}_n(\theta^{(k)})\}^{-1}\nabla \mathcal{L}_n(\theta^{(k)})\|^2_{\nabla^2 \mathcal{L}_n(\theta^{(k)})} + \eta \langle \nabla \Loss_n(\theta^{(k)}), \tilde{N}_k \rangle
\end{align*}
and
\begin{equation*}
\|\Delta \theta^{(k)}\|^2_{\nabla^2 \mathcal{L}_n(\theta^{(k)})} \le 2\eta^2 \left(\|\{\nabla^2 \mathcal{L}_n(\theta^{(k)})\}^{-1}\nabla \mathcal{L}_n(\theta^{(k)})\|^2_{\nabla^2 \mathcal{L}_n(\theta^{(k)})} + \|\tilde N_k\|^2_{\nabla^2 \mathcal{L}_n(\theta^{(k)})}\right).
\end{equation*}
Thus, we obtain
\begin{align}
\label{EqnSleepy}
\Loss_n(\theta^{(k+1)}) & \le \mathcal{L}_n(\theta^{(k)}) - \left(\eta - \eta^2 \tau_0\right) \|\{\nabla^2 \mathcal{L}_n(\theta^{(k)})\}^{-1}\nabla \mathcal{L}_n(\theta^{(k)})\|^2_{\nabla^2 \mathcal{L}_n(\theta^{(k)})} \notag \\
& \qquad + \eta \|\nabla \Loss_n(\theta^{(k)})\|_2 \|\tilde{N}_k\|_2 + \eta^2 \tau_0 \|\tilde N_k\|^2_{\nabla^2 \mathcal{L}_n(\theta^{(k)})}.
\end{align}
Furthermore, minimizing both sides of the lower bound~\eqref{LBstable} in Lemma~\ref{LemA} with respect to $\theta_1$ gives
\begin{equation*}
\mathcal{L}_n(\hat{\theta}) \ge \mathcal{L}_n(\theta_2) - \frac{\tau_0}{2}  \|\{\nabla^2 \mathcal{L}_n(\theta_2)\}^{-1}\nabla \mathcal{L}_n(\theta_2)\|^2_{\nabla^2 \mathcal{L}_n(\theta_2)}.
\end{equation*}
Plugging in  $\theta_2 = \theta^{(k)}$, we then obtain
\begin{equation*}
\|\{\nabla^2 \mathcal{L}_n(\theta^{(k)})\}^{-1}\nabla \mathcal{L}_n(\theta^{(k)})\|^2_{\nabla^2 \mathcal{L}_n(\theta^{(k)})} \ge \frac{2}{\tau_0} \left(\mathcal{L}_n(\theta^{(k)}) - \mathcal{L}_n(\hat{\theta})\right).
\end{equation*}
%
Combining this bound with inequality~\eqref{EqnSleepy} and using the fact that $\eta \le \frac{1}{2\tau_0}$ then gives
\begin{equation*}
\mathcal{L}_n(\theta^{(k+1)}) \le \mathcal{L}_n(\theta^{(k)}) - \frac{\eta}{\tau_0} \left(\mathcal{L}_n(\theta^{(k)}) - \mathcal{L}_n(\hat{\theta})\right) + \eta \|\nabla \Loss_n(\theta^{(k)})\|_2 \|\tilde{N}_k\|_2 + \eta^2 \tau_0 \|\tilde N_k\|^2_{\nabla^2 \mathcal{L}_n(\theta^{(k)})}.
\end{equation*}
Hence, using the fact that $\max_{k \le K} \|\tilde N_k\|_2 \le \tilde{r}_{priv}$ from 
Lemma~\ref{lem:Newton_update_self-concordance}, we have
\begin{align}
\label{EqnBum}
\mathcal{L}_n(\theta^{(k+1)}) - \mathcal{L}_n(\hat{\theta}) & \le \left(\mathcal{L}_n(\theta^{(k)}) - \mathcal{L}_n(\hat{\theta})\right) - \frac{\eta}{\tau_0} \left(\mathcal{L}_n(\theta^{(k)}) - \mathcal{L}_n(\hat{\theta})\right) + \rho \notag \\
& = \left(1 - \frac{\eta}{\tau_0}\right) \left(\mathcal{L}_n(\theta^{(k)}) - \mathcal{L}_n(\hat{\theta})\right) + \rho,
\end{align}
where $\rho = \eta B \tilde{r}_{priv} + \eta^2 \bar{B} \tau_0 \tilde{r}_{priv}^2$. (Note that $\tau_0 \ge 1$, so $\eta \le \frac{1}{\tau_0}$ implies that $\eta \le \tau_0$, as well.)
Iterating the bound~\eqref{EqnBum}, we obtain
\begin{equation*}
\mathcal{L}_n(\theta^{(k+1)}) - \mathcal{L}_n(\hat{\theta}) \le \left(1 - \frac{\eta}{\tau_0}\right)^k \left(\mathcal{L}_n(\theta^{(0)}) - \mathcal{L}_n(\hat{\theta})\right) + \frac{\rho \tau_0}{\eta}\left(1 - \left(1   -\frac{\eta}{\tau_0}\right)\right)^k,
\end{equation*}
as claimed.


\subsubsection{Supporting lemmas}

The following result is an analogue of Lemma~\ref{lem:LSCball} for self-concordant functions, showing that a parameter $\theta$ must be close to $\thetahat$ if the sub-optimality gap is small.

\begin{lemma}
\label{LemB4}
Suppose $\mathcal{L}_n$ is $(\gamma,2)$-self-concordant. For $\theta \in \real$, define $\Delta = \Loss_n(\theta) - \Loss_n(\thetahat)$. Then
\begin{equation*}
\|\theta - \thetahat\|_2 \le g^{-1}\left(\frac{\gamma^2 \Delta}{\lambda_{\min}(\nabla^2 \Loss_n(\thetahat))}\right),
\end{equation*}
where $g(t) := e^{-t} + t - 1$ for $t > 0$.
\end{lemma}

\begin{proof}
Let $R = \|\thetahat - \theta\|_2$. Applying Lemma~\ref{LemSCfunc} with $x = \thetahat$ and $y = \theta$, we have
\begin{equation*}
\frac{\exp(-\gamma R) + \gamma R - 1}{(\gamma R)^2} \cdot R^2 \lambda_{\min}(\nabla^2 \Loss_n(\thetahat)) \le \Loss_n(\theta) - \Loss_n(\thetahat).
\end{equation*}
Rearranging then gives
\begin{equation*}
\exp(-\gamma R) + \gamma R - 1 \le \frac{\gamma^2 (\Loss_n(\theta) - \Loss_n(\thetahat))}{\lambda_{\min}(\nabla^2 \Loss_n(\thetahat))},
\end{equation*}
from which the result clearly follows by applying $g^{-1}$.
\end{proof}

\begin{lemma}
\label{lem:Newton_update_self-concordance}
Suppose $\Loss_n$ is $(\gamma, 2)$-self-concordant, and let $\Delta_0$ and $R_0$ be as defined in the statement of Proposition~\ref{ThmNewtonGlobal}.
Also suppose $\eta \le \frac{4\tau_{1,R_0}^2}{\bar{B}^2}$ and the sample size satisfies $n = \Omega\left(\frac{\sqrt{Kp \log(Kp/\xi)}}{\mu}\right)$.
With probability at least $1-\xi$, the noisy Newton updates~\eqref{eq:NNewton} can be rewritten as 
\begin{equation*}
\label{EqnNewtonUpdate}
    \theta^{(k+1)}=\theta^{(k)}- \eta \left\{\nabla^2\mathcal{L}_n(\theta^{(k)}))\right\}^{-1}\nabla\mathcal{L}_n(\theta^{(k)})+ \eta \tilde{N}_k,
\end{equation*}
where
\begin{equation*}
\label{EqnNewtonErr}
\max_{k < K} \|\tilde N_k\|_2 \leq C \frac{\bar{B} B\sqrt{Kp\log(Kp/\xi)}}{\mu n\tau_{1,R_0}^2} := \tilde{r}_{priv},
\end{equation*}
and $C$ is a positive constant. Furthermore, $\theta^{(k)} \in \mathcal{B}_{R_0}(\hat\theta)$ for all $k \le K$.
\end{lemma}

\begin{proof}


Using a similar argument as in Lemma~\ref{lemma1_alternative}, we will prove by induction that $\|\theta^{(k)} - \hat\theta\|_2 \le R_0$ and $\Loss_n(\theta^{(k)}) \le \Loss_n(\thetahat) + \Delta_0$, for all $k \le K$. First note that by Lemmas~\ref{lem:maxGaussian} and~\ref{lem:norm_Gaussian_matrix} and a union bound, we have $\max_{k < K} \|Z_k\|_2 \le 4\sqrt{p} + 2\sqrt{2\log(2K/\xi)}$ and $\max_{k < K} \|W_k\|_2 \le \sqrt{2p\log(4Kp/\xi)}$, with probability at least $1-\xi$. We will assume these bounds hold and argue deterministically for the rest of the proof.

The base case $k=0$ holds by Lemma~\ref{LemB4}. For the inductive step, suppose the bounds hold for $\{1, \dots, k\}$. Note that $\|\left\{\nabla^2\mathcal{L}_n(\theta^{(k)})\right\}^{-1}\|_2 \le \frac{1}{2\tau_{1,R_0}}$, since we have local strong convexity guaranteed by Lemma~\ref{LemSCMult}. As in the proof of inequality~\eqref{EqnNk} in Lemma~\ref{LemNoiseNewton}, the series expansion~\eqref{EqnNoiseNewton} and the bounds on $\|Z_k\|_2$ and $\|W_k\|_2$ imply that for all $k < K$, we have
\begin{align}
\label{EqnNk2}
    \|\tilde{N}_k\|_2 \le C_0 \frac{B\sqrt{K}\{\sqrt{p}+\sqrt{\log(2K/\xi)}\}+\bar{B} B\tau_{1,R_0}^{-1}\sqrt{Kp\log(Kp/\xi)}}{\mu n\tau_{1,R_0}} = \tilde{r}_{priv}.
\end{align}
Using smoothness, we now write
\begin{align}
\label{EqnPork}
\mathcal{L}_n(\theta^{(k+1)}) & = \mathcal{L}_n\left(\theta^{(k)} - \eta \{\nabla^2 \Loss_n(\theta^{(k)})\}^{-1} \nabla \mathcal{L}_n(\theta^{(k)}) + \eta \tilde{N}_k\right) \notag \\
& \le \mathcal{L}_n(\theta^{(k)}) - \eta \langle \nabla \mathcal{L}_n(\theta^{(k)}), \{\nabla^2 \Loss_n(\theta^{(k)})\}^{-1} \nabla \mathcal{L}_n(\theta^{(k)}) - \tilde{N}_k\rangle \notag \\
& \qquad + \frac{\bar{B} \eta^2}{2} \|\{\nabla^2 \Loss_n(\theta^{(k)})\}^{-1} \nabla \mathcal{L}_n(\theta^{(k)}) - \tilde N_k\|_2^2 \notag \\
& = \mathcal{L}_n(\theta^{(k)}) + \frac{\bar{B} \eta^2}{2} \|\tilde{N}_k\|_2^2 - \|\nabla \mathcal{L}_n(\theta^{(k)})\|^2_{\eta \{\nabla^2 \Loss_n(\theta^{(k)})\}^{-1} - \frac{\bar{B} \eta^2}{2} \{\nabla^2 \Loss_n(\theta^{(k)})\}^{-2}} \notag \\
& \qquad - \eta \langle \nabla \Loss_n(\theta^{(k)}) + \bar{B} \eta \{\nabla^2 \Loss_n(\theta^{(k)})\}^{-1} \nabla \mathcal{L}_n(\theta^{(k)}), \tilde{N}_k \rangle.
\end{align}
Note that by the assumption $\eta \le \frac{4\tau_{1,R_0}^2}{\bar{B}^2}$ and the facts that $\|\nabla^2 \Loss_n(\theta^{(k)})\|_2 \le \bar{B}$ and $\|\{\nabla^2 \Loss_n(\theta^{(k)})\}^{-1}\|_2 \le \frac{1}{2\tau_{1,R_0}}$, we have
\begin{align*}
& \lambda_{\min}\left(\eta \{\nabla^2 \Loss_n(\theta^{(k)})\}^{-1} - \frac{\bar{B} \eta^2}{2} \{\nabla^2 \Loss_n(\theta^{(k)})\}^{-2}\right) \\
& \qquad \ge \eta \lambda_{\min}\left(\{\nabla^2 \Loss_n(\theta^{(k)})\}^{-1}\right) - \lambda_{\max}\left(\frac{\bar{B} \eta^2}{2} \{\nabla^2 \Loss_n(\theta^{(k)})\}^{-2}\right) \\
& \qquad \ge \frac{\eta}{\bar{B}} - \frac{\bar{B} \eta^2}{2} \cdot \frac{1}{4\tau_{1,R_0}^2} \\
& \qquad \ge \frac{\eta}{2\bar{B}},
\end{align*}
implying that
\begin{align*}
\|\nabla \mathcal{L}_n(\theta^{(k)})\|^2_{\eta \{\nabla^2 \Loss_n(\theta^{(k)})\}^{-1} - \frac{\bar{B} \eta^2}{2} \{\nabla^2 \Loss_n(\theta^{(k)})\}^{-2}}
%
%
\ge \frac{\eta}{2\bar{B}} \|\nabla \mathcal{L}_n(\theta^{(k)})\|^2_2.
\end{align*}
Furthermore, we have
\begin{align*}
& \left|\langle \nabla \Loss_n(\theta^{(k)}) + \bar{B} \eta \{\nabla^2 \Loss_n(\theta^{(k)})\}^{-1} \nabla \mathcal{L}_n(\theta^{(k)}), \tilde{N}_k \rangle\right| \\
& \qquad \le \|\tilde{N}_k\|_2 \left\|\nabla \Loss_n(\theta^{(k)}) + \bar{B} \eta \{\nabla^2 \Loss_n(\theta^{(k)})\}^{-1} \nabla \mathcal{L}_n(\theta^{(k)})\right\|_2 \\
& \qquad \le \|\tilde{N}_k\|_2 \left(B + \frac{\bar{B}\eta B}{2\tau_{1,R_0}}\right),
\end{align*}
Hence, by inequality~\eqref{EqnPork}, the descent condition $\mathcal{L}_n(\theta^{(k+1)}) \le \mathcal{L}_n(\theta^{(k)})$ can be guaranteed as long as
\begin{equation*}
B\eta \left(1+ \frac{\bar{B} \eta}{2\tau_{1,R_0}}\right)\|\tilde N_k\|_2 + \frac{\bar{B} \eta^2}{2} \|\tilde N_k\|_2^2 \le \frac{\eta}{2\bar{B}} \|\nabla\mathcal{L}_n(\theta^{(k)})\|^2_2,
\end{equation*}
which is in turn guaranteed if
\begin{equation*}
        \|\tilde N_k\|  \leq \frac{-B(1+\frac{\eta\bar{B}}{2\tau_{1,R_0}}) +\sqrt{B^2(1+\frac{\eta\bar{B}}{2\tau_{1,R_0}})^2 + \eta\|\nabla\mathcal{L}_n(\theta^{(k)})\|^2_2}}{\eta\bar{B}},
\end{equation*}
or equivalently,
\begin{align*}
\|\nabla\mathcal{L}_n(\theta^{(k)})\|_2 & \geq \sqrt{\frac{\left(\eta \bar{B}\|\tilde N_k\|_2 + B\left(1 + \frac{\eta \bar{B}}{2\tau_{1,R_0}}\right)\right)^2 - B^2\left(1+\frac{\eta \bar{B}}{2\tau_{1,R_0}}\right)^2}{\eta}}\\
&=  \sqrt{\eta\bar{B}^2\|\tilde N_k\|_2^2 + 2\bar{B} B\left(1+\frac{\eta\bar{B}}{2\tau_{1,R_0}}\right)\|\tilde N_k\|_2}.
\end{align*}
By inequality~\eqref{EqnNk2}, this last bound is ensured if
\begin{equation}
    \label{eq:suff_Newton3}
    \|\nabla\mathcal{L}_n(\theta^{(k)})\|_2 \geq \sqrt{\eta\bar{B}^2\tilde r_{priv}^2+2\bar{B} B\left(1+\frac{\eta\bar{B}}{2\tau_{1,R_0}}\right)\tilde{r}_{priv}} := \check{r}_{priv}.
\end{equation}
In summary, if $\|\nabla\mathcal{L}_n(\theta^{(k)})\|_2\geq \check{r}_{priv}$, then $\Loss_n(\theta^{(k+1)}) - \Loss_n(\thetahat) \le \Loss_n(\theta^{(k)}) - \Loss_n(\thetahat) \le \Delta_0$, so also $\theta^{(k+1)} \in \mathcal{B}_{R_0}(\hat\theta)$ by Lemma~\ref{LemB4}.

If instead $\|\nabla\mathcal{L}_n(\theta^{(k)})\|_2< \check{r}_{priv}$, we use an alternative argument to conclude that $\theta^{(k+1)}\in\mathcal{B}_{R_0}(\hat\theta)$: Note that $\|\nabla\mathcal{L}_n(\theta^{(k)})\|_2\geq 2\tau_{1,R_0}\|\theta^{(k)}-\hat\theta\|_2$, so the triangle inequality and a sufficiently large $n$ (to be specified below) show that
\begin{align}
\nonumber
    \|\theta^{(k+1)}-\hat\theta\|_2&\leq \nonumber\|\theta^{(k)}-\hat\theta\|_2+\|\theta^{(k+1)}-\theta^{(k)}\|_2\\
\nonumber    &\leq \frac{1}{2\tau_{1,R_0}}\check{r}_{priv} + \eta \|\{\nabla^2\mathcal{L}_n(\theta^{(k)})\}^{-1}\nabla\mathcal{L}_n(\theta^{(k)}) - \tilde N_k\|_2\\
\nonumber & \le \frac{1}{2\tau_{1,R_0}}\check{r}_{priv} + \frac{\eta \check{r}_{priv}}{2\tau_{1, R_0}} + \eta \tilde{r}_{priv} \\
    &\leq \min\left\{\sqrt{\frac{\Delta_0}{\tau_2}}, R_0\right\}. \label{eq:radius_iterates_NNewton}
\end{align}
Therefore, we conclude that if $ \|\nabla\mathcal{L}_n(\theta^{(k)})\|_2< \check{r}_{priv}$, then $\theta^{(k+1)}\in\mathcal{B}_{R_0}(\hat\theta)$, as well. Furthermore, by smoothness, inequality~\eqref{eq:radius_iterates_NNewton} also gives
\begin{equation*}
\Loss_n(\theta^{(k+1)}) - \Loss_n(\thetahat) \le \tau_2 \|\theta^{(k+1)} - \thetahat\|_2^2 \le \tau_2 \cdot \frac{\Delta_0}{\tau_2} \le \Delta_0.
\end{equation*}
From the definitions of $\tilde{r}_{priv}$ and $\check{r}_{priv}$ in inequalities~\eqref{EqnNk2} and~\eqref{eq:suff_Newton3}, we can check that the last inequality in the chain~\eqref{eq:radius_iterates_NNewton} can be ensured via the
assumed minimum sample size condition $n = \Omega\left(\frac{\sqrt{Kp \log(Kp/\xi)}}{\mu}\right)$.
This completes the inductive step.
\end{proof}

\subsection{Global convergence of damped Newton under strong convexity }\label{App:global_SC_geom}

The following proposition is the strong convexity counterpart of Proposition \ref{ThmNewtonGlobal}. The proof structure is the same with some minor adaptations that we work out for completeness. The argument is in fact greatly simplified, as we do not need to localize the iterates in a ball centered at $\hat\theta$ and show that they remain within that ball (as in Lemma \ref{lem:Newton_update_self-concordance}) before proceeding to the main contraction argument. This is reflected in the simpler step size condition.
\begin{proposition}\label{prop:global_SC_geom}
    Suppose $\Loss_n$ is a $\gamma$-strongly convex and $L$-smooth loss function satisfying Condition~\ref{ass:Hessian}. Let $\Delta_0 = \Loss_n(\theta^{(0)}) - \Loss_n(\thetahat)$ and $K \ge 1$, and suppose the step size satisfies $\eta \le  \gamma$ and the sample size satisfies $n = \Omega\left(\frac{\sqrt{Kp \log(Kp/\xi)}}{\mu}\right)$. Then with probability at least $1-\xi$, the Newton iterates satisfy
\begin{equation*}
\mathcal{L}_n(\theta^{(K)}) - \mathcal{L}_n(\hat{\theta}) \le \left(1 - \frac{\eta}{\gamma }\right)^K \left(\mathcal{L}_n(\theta^{(0)}) - \mathcal{L}_n(\hat{\theta})\right) + r_{priv},
\end{equation*}
where $r_{priv} = C' \frac{\sqrt{Kp\log(Kp/\xi)}}{\mu n\gamma}$, and $C' > 0$ is a constant depending on $\gamma$, $B$, and $\bar{B}$.
\end{proposition}
\begin{proof}
We first recall that by Lemmas~\ref{lem:maxGaussian} and~\ref{lem:norm_Gaussian_matrix} and a union bound, we have $\max_{k < K} \|Z_k\|_2 \le 4\sqrt{p} + 2\sqrt{2\log(2K/\xi)}$ and $\max_{k < K} \|W_k\|_2 \le \sqrt{2p\log(4Kp/\xi)}$, with probability at least $1-\xi$.  
Consequently, with probability $1-\xi$, the noisy Newton iterates can be written as 
\begin{equation*}
\label{EqnNewtonUpdate}
    \theta^{(k+1)}=\theta^{(k)}- \eta \left\{\nabla^2\mathcal{L}_n(\theta^{(k)}))\right\}^{-1}\nabla\mathcal{L}_n(\theta^{(k)})+ \eta \tilde{N}_k,
\end{equation*}
where $\tilde{N}_k$ is defined as in 
equation~\eqref{EqnNoiseNewton} and
 \begin{equation*}
\label{EqnNewtonErr}
\max_{k < K} \|\tilde N_k\|_2 \leq C \frac{\bar{B} B\sqrt{Kp\log(Kp/\xi)}}{\mu n\gamma^2} := \tilde{r}_{priv},
\end{equation*}
and $C$ is a positive constant. Note that  $\gamma$-strong convexity and $L$-smoothness implies $\frac{L}{\gamma}$-Hessian stability. We can therefore follow the same argument used for Proposition \ref{ThmNewtonGlobal} to obtain inequality~\eqref{EqnBum}, where now $\tau_0=\gamma$ and $\rho=\eta B\tilde{r}_{priv}+\eta^2\bar{B}\gamma\tilde{r}_{priv}^2$. Iterating this inequality leads to 
\begin{equation*}
\mathcal{L}_n(\theta^{(k+1)}) - \mathcal{L}_n(\hat{\theta}) \le \left(1 - \frac{\eta}{\gamma}\right)^k \left(\mathcal{L}_n(\theta^{(0)}) - \mathcal{L}_n(\hat{\theta})\right) + \frac{\rho \gamma}{\eta}\left(1 - \left(1   -\frac{\eta}{\gamma}\right)\right)^k.
\end{equation*}
The desired claim follows from this last inequality.
 \end{proof}
 
 \section{Numerical illustrations}
\label{SecExperiments}

In this section, we explore the performance of our algorithms on several real data sets.

\subsection{Linear regression}
 
To further explore the performance of our proposed methods, we consider fitting a linear regression model to a housing price data set utilized in \cite{lei2011}. The data consist of price, square footage, year of sale, and county for 348,189 homes sold in the San Francisco Bay Area from 2003 to 2006. We excluded all records with missing values, and combined several counties such that the number of county categories was reduced from 9 to 6. The remaining data set contains 286,537 records. Additionally, price, square footage, and year were standardized by subtracting their medians and dividing by median absolute deviation. (Note that differentially private estimates of the median and MAD of these variables could also be returned. This comes at an additional cost in terms of privacy budget, but may be desirable for interpretability or prediction purposes.)
 
With this data set, we used noisy gradient descent with loss function as in equation~\eqref{EqnLinearLoss} to estimate a linear regression model for predicting the home price from the remaining variables. The total privacy budget for estimation and corresponding inference was $\mu=0.25$, the number of iterations $K$ was 100, and optimization was initialized at $\beta^{(0)} = 0$ and $\sigma^{(0)} = 1$. The estimated coefficients of this regression model are in Table \ref{table:ngd_california}.

\begin{table}[h]
\centering
\begin{tabular}{ c c c c c } 
              &    Value & Std. Error  &  z value  &p value\\
(Intercept)  &-0.1835649 &0.002066234  &-88.84036      & 0\\
bsqft        & 0.6434390 &0.001411104  &455.98285      & 0\\
date         & 0.4578594 &0.001696891  &269.82252      & 0\\
Contra Costa &-0.2743499 &0.003084991  &-88.93053      & 0\\
MSS          & 0.8731473 &0.004160866  &209.84747      & 0\\
NS           &-0.1559515 &0.003732295  &-41.78435      & 0\\
Santa Clara  & 0.2388637 &0.002969860  & 80.42928      & 0\\
Solano       &-0.6718810 &0.003007074  &-223.43350     & 0\\
\end{tabular}
\caption{Coefficients estimated via noisy gradient descent ($\mu=0.25$)}
\label{table:ngd_california}
\end{table}
The inference procedures outlined in
Section \ref{SecInference} indicate that all regression coefficients are statistically significant at a $\alpha=0.05$ threshold. To obtain an approximate notion of performance on smaller data sets, we repeatedly applied noisy gradient descent to random subsamples of the original data set. In each implementation, the total privacy budget was $\mu=0.25$, the number of iterations $K$ was 100, and optimization was initialized at $\beta^{(0)} = 0$ and $\sigma^{(0)} = 1$.

\begin{figure}[h]
    \centering
    \includegraphics[width=12cm]{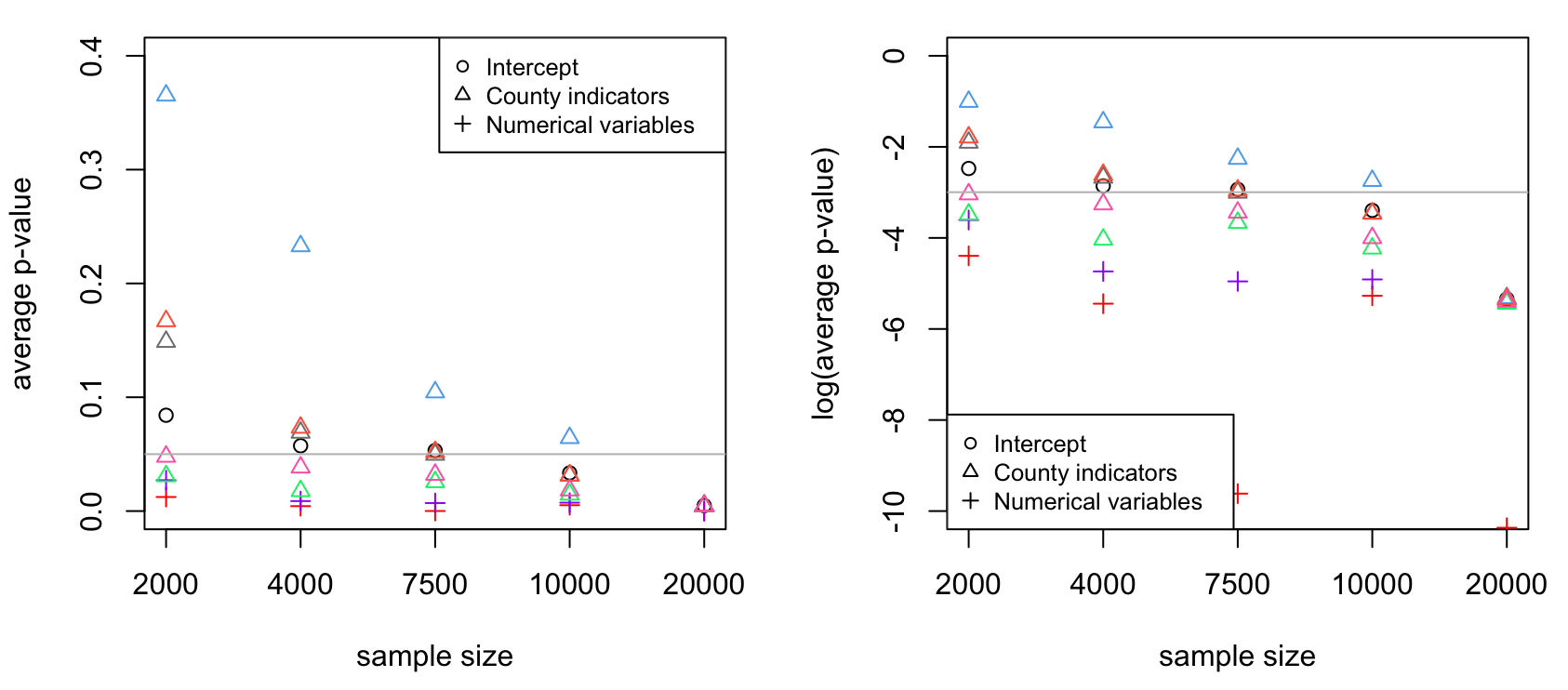}
    \caption{Noisy gradient descent applied to random subsamples of the housing data set.}
    \label{fig:housing subsamples3}
\end{figure}

In these subsamples, we see that for sample size $n=2000$, only 4 of the 8 regression coefficients have p-values averaging less than 0.05 across 200 repetitions. As we increase the sample size to 10,000, 7 of the 8 coefficients have average p-values below that significance threshold, and by a sample size 20,000, all coefficients are detected as significant.



\begin{center}
\begin{figure}
\begin{tabular}{cc}
\includegraphics[width=6cm]{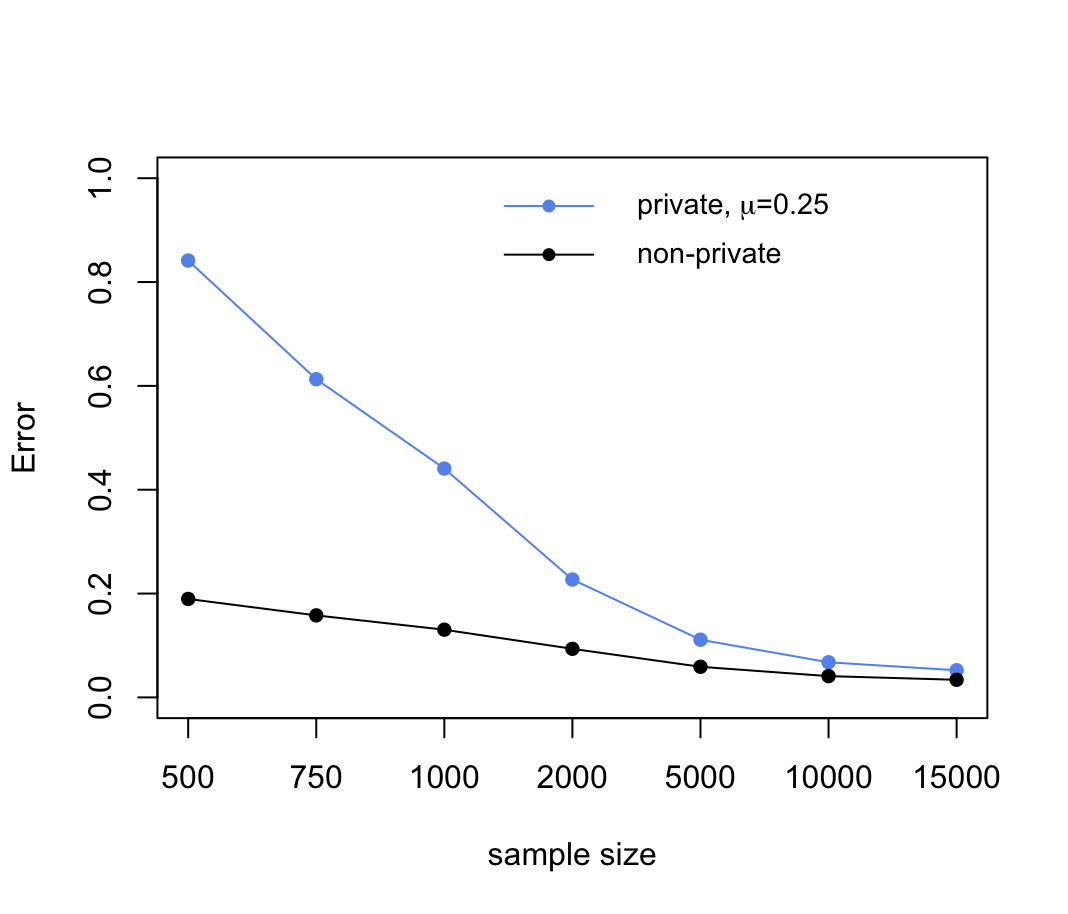} &
\includegraphics[width=6cm]{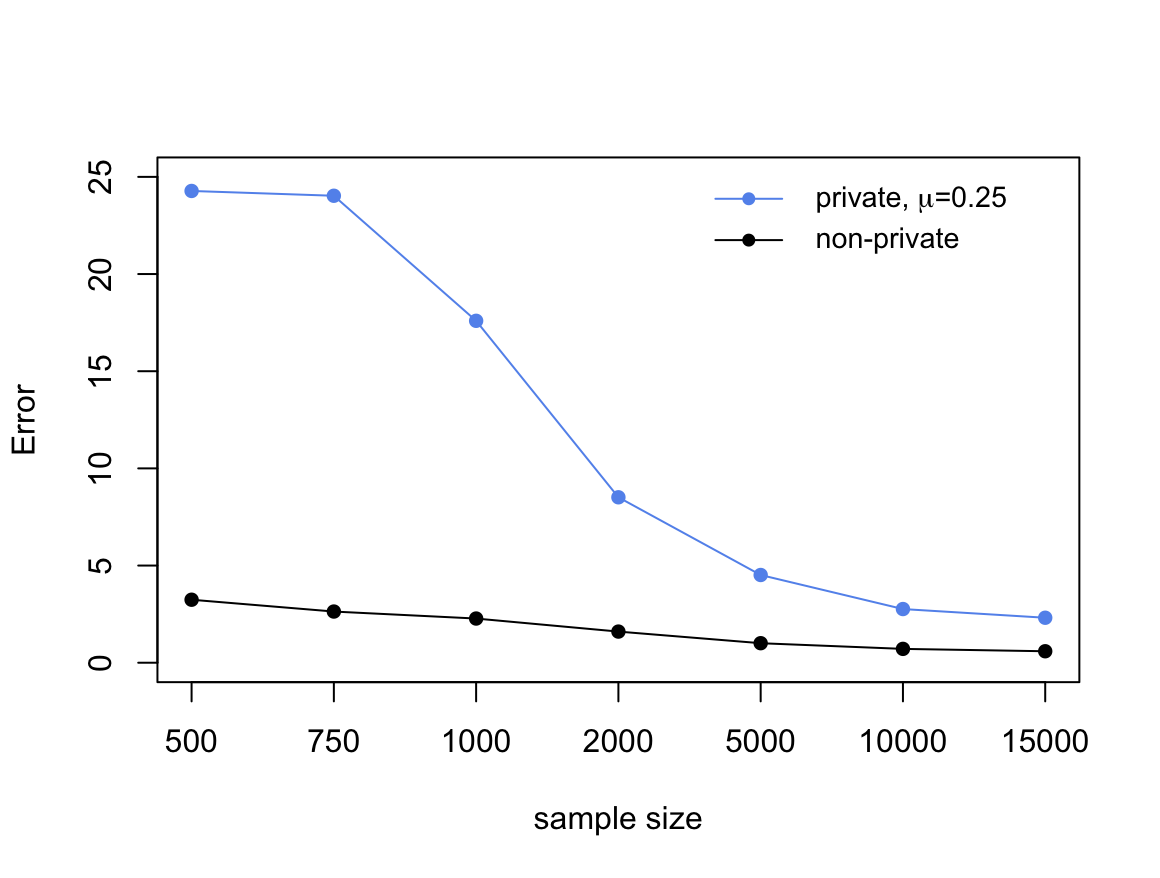} \\
(a) Linear regression with housing data & (b) Logistic regression with bank data
\end{tabular}
\caption{Subsampled estimates of $\E (\| \hat{\beta}-\beta \|_{2})$.}
\label{FigSubsample}
\end{figure}
\end{center}

Figure~\ref{FigSubsample}(a) demonstrates the error of private and non-private estimates of the parameter vector $\beta$ calculated from subsamples of the housing data set. Error is calculated with respect to the non-private estimator obtained from the full data set and averaged over 400 repetitions at each sample size. The number of iterations $K$ was set to 50 for sample size 500; 75 for sample sizes 750, 1000, and 2000; and 100 for all larger sample sizes. The privacy mechanism has a relatively large impact on the estimation error for $\beta$ in small samples, but the gap between private and non-private estimators closes with increasing sample size.

\subsection{Logistic regression}

In this example, we fit a logistic regression model to a data set from \cite{moroetal2014}, describing customer relationships with a bank in Portugal. To fit this model, we used noisy gradient descent to minimize a version of cross-entropy loss modified to include Mallows-style weights:
\begin{equation}
    \mathcal{L}_n(\beta)=\frac{1}{n}\sum_{i=1}^{n} \left( -y_{i} \log \left( \frac{1}{1+\exp(x_{i}^\top \beta)} \right) + (1-y_{i}) \log \left(\frac{\exp(x_{i}^\top \beta)}{1+\exp(x_{i}^\top \beta)} \right) \right)w(x_{i}),
\end{equation}
where $w(x_{i})=\textrm{min} \left(1,\frac{25}{\left \|x_{i}  \right \|_{2}^{2}} \right)$.

This data set contains 45,211 records consisting of customer attributes such as age, job type, and types of business conducted with the bank. The response is whether the customer subscribed to a term deposit. For preprocessing, numeric covariates were standardized and categorical covariates were converted to one-hot encoding (i.e., representing a categorical variable with $k$ levels, with $k-1$ binary indicator variables). One covariate (days since previous marketing contact with customer) was excluded as it was undefined for over 80\% of observations. After this preprocessing, the data set had 41 covariates.
Figure~\ref{FigSubsample}(b) shows the error of private and non-private estimates of the parameter vector $\beta$. As in the previous section, error is calculated with respect to the non-private estimator obtained from the full data set. Results are averaged over 400 repetitions at each sample size. The number of iterations $K$ was set to 50 for sample size 500; 75 for sample sizes 750, 1000, and 2000; and 100 for all larger sample sizes. All optimizations were initialized at $\beta^{(0)} = 0$.
%


\section{Private backtracking line search}
\label{AppBacktrack}

In this appendix, we provide details for a proposed private version of backtracking line search.

\subsection{Procedure}

To avoid divergence associated with poor initializations in Newton's method, backtracking line search may be used to adaptively select a step size smaller than 1. The template for backtracking line search is as follows: At iteration $k$ of Newton's method, set step size $\eta=1$, and choose constants $\gamma \in (0, 1)$ and $\alpha \in (1, \frac{1}{2})$. Then

\vspace{\baselineskip}

\textbf{while} $\Loss_n(\theta^{(k)} - \eta \nabla^2 \Loss_n(\theta^{(k)})^{-1} \nabla \Loss_n(\theta^{(k)})) - \Loss_n(\theta^{(k)}) >  - \alpha \eta \nabla \Loss_n(\theta^{(k)})^{T}\nabla^2 \Loss_n(\theta^{(k)})^{-1} \nabla \Loss_n(\theta^{(k)}): $

\hspace{1cm} \textbf{do} $\eta \leftarrow \gamma \eta$

\vspace{\baselineskip}

That is, we multiplicatively reduce the step size until the condition above (the Armijo-Goldstein condition) no longer holds.

At iteration $k$ of the differentially private Newton's method from equation \eqref{eq:NNewton}, differentially private estimates of $\nabla \Loss_n(\theta^{(k)})$ and $\nabla^2 \Loss_n(\theta^{(k)})$ are already available. We can substitute these private quantities into the backtracking condition above without affecting the privacy guarantee, due to closure under postprocessing. Let $H_{k}$ denote the private Hessian of the loss function at iteration $k$:
$$ H_k= \frac{1}{n}\sum_{i=1}^n\dot\Psi(x_i,\theta^{(k)}) +\frac{2\bar{B}\sqrt{2K}}{\mu n}W_{k} , $$
and let $G_{k}$ denote the corresponding gradient:
$$G_{k}= \frac{1}{n}\sum_{i=1}^n\Psi(x_i,\theta^{(k)})+\frac{2B\sqrt{2K}}{\mu n}Z_{k}, $$
where $K$ is the number of iterations budgeted for the Newton's method algorithm, and $B$, $\bar{B}$, $W_{k}$, and $Z_{k}$ are as defined in Section \ref{SecNewtonAlgo}.

The backtracking condition we wish to evaluate can then be approximated as
\begin{equation}
    \label{eq:AGcondition}
 \Loss_n( \theta^{(k)} - \eta H_{k}^{-1} G_{k}) - \Loss_n(\theta^{(k)}) > -\alpha \eta G_{k}^{T} H_{k}^{-1} G_{k}.
\end{equation}
The right-hand side of inequality~\eqref{eq:AGcondition} can be computed at no additional privacy cost, while the left-hand side requires a modification to achieve differential privacy.
Recall that the loss function is of the form $\Loss_n(\theta) = \frac{1}{n}\sum_{i=1}^n\rho(x_i,\theta)$. Furthermore, we assume throughout this paper that $\rho(x, \theta)$ is differentiable with respect to $\theta$; and denoting $\Psi(x,\theta)=\frac{\partial}{\partial\theta}\rho(x,\theta)$, we assume an upper bound $B \geq \sup_{x\in\mathcal X,\theta\in\Theta}\|\Psi(x,\theta)\|_2$.
It follows that
$$ | \rho(x, \theta - \eta H_{k}^{-1} G_{k}) - \rho(x, \theta) | \leq \| \eta H_{k}^{-1} G_{k} \|_{2} \| \Psi(x, \tilde{\theta}) \|_{2} $$
for some $\tilde{\theta} \in (\theta, \theta - \eta H_{k}^{-1} G_{k})$, implying that
$$ | \rho(x, \theta - \eta H_{k}^{-1} G_{k}) - \rho(x, \theta) | \leq  \eta \| H_{k}^{-1} G_{k} \|_{2} B. $$
Thus, the global sensitivity of the function $ \Loss_n( \theta^{(k)} - \eta H_{k}^{-1} G_{k}) - \Loss_n(\theta^{(k)})$ is $\frac{2 \eta}{n}\| H_{k}^{-1} G_{k} \|_{2} B$. 

The backtracking condition may therefore be privately evaluated by checking the condition
\begin{equation}
\label{eq:private_backtracking_condition}
  \Loss_n( \theta^{(k)} - \eta H_{k}^{-1} G_{k}) - \Loss_n(\theta^{(k)}) + \frac{2B \eta \| H_{k}^{-1} G_{k} \|_{2} \sqrt{T}}{n \mu} \xi_{k}> -\alpha \eta G_{k}^{T} H_{k}^{-1} G_{k},
\end{equation}
where $\xi_{k}$ is a standard Gaussian random variable.
Each evaluation of this condition satisfies $\frac{\mu}{\sqrt{T}}$-GDP. As with the number of algorithm iterations, the total number of backtracking steps must be budgeted ahead of time to ensure a specific privacy guarantee. Budgeting for $T$ total backtracking steps, the entirety of the noisy Newton's method algorithm satisfies $\sqrt{2} \mu$-GDP. The one theoretical drawback of this proposal is that we do not have an a priori upper bound on the number of backtracking iterations needed before exiting the damped Newton phase, though this number is often not large in practice (see the example below).

As discussed in \cite[Ch.\ 9.5.3]{boydandvanderberghe2004}, this scheme for selecting step sizes causes Newton's method to exhibit an initial ``damped" phase, followed by a ``pure" phase in which all step sizes are equal to 1. Convergence in the ``damped" phase is slower, while the ``pure" phase is characterized by faster, quadratic convergence. Empirically, the differentially private counterpart exhibits similar behavior, as illustrated in Figure \ref{fig:backtracking} below. Note that for implementation, we configured the private algorithm to discontinue backtracking in subsequent iterations after the first time a step size of 1 was chosen, as this performed favorably in numerical simulations.
\begin{figure}[H]
    \centering
    \includegraphics[width=10cm]{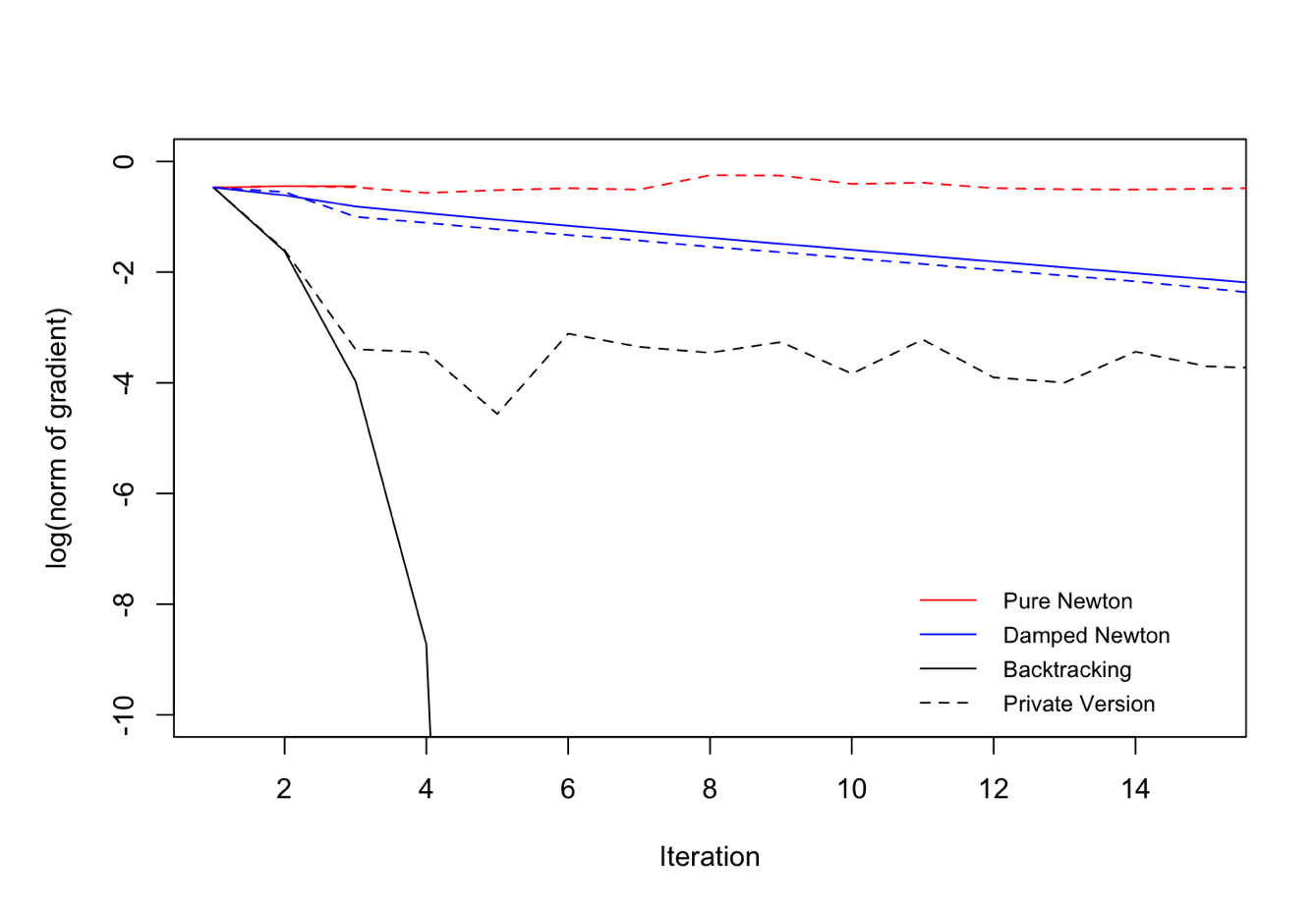} \\
    \caption{Behavior of Newton's method with backtracking line search, compared to pure and damped Newton's method}
    \label{fig:backtracking}
\end{figure}
In this example, Newton's method is used to estimate the slope parameter in a linear regression model. The data set $\left \{ \left (x_{i},y_{i} \right ) \right \} _{i=1}^{2000}$ was generated according to the model $y_{i}=x_{i}^{T} \beta +\epsilon_{i}$, where $\epsilon_{i} \stackrel{i.i.d.}{\sim} N(0,1)$, the covariate vectors are given by $x_{i}=\left (1, z_{i} \right )^{T}$, where $z_{i} \stackrel{i.i.d.}{\sim} N(0, \mathbb{I}_{3})$, and the true slope is $\beta= (1,1,1,1)^{T}$. We take our loss function to be
$$\mathcal{L}_n(\beta,\sigma)=\frac{1}{n}\sum_{i=1}^{n} \rho_{c} ( y_{i}-x_{i}^{T} \beta )  w(x_{i}),$$
where $\rho_{c}$ is the Huber loss function with tuning parameter $c$, and $w(x_{i})=\textrm{min} \left(1,\frac{2}{\left \|x_{i}  \right \|_{2}^{2}} \right)$.

\vspace{\baselineskip}
With an initial point $\beta_{0}= (4,-2,0,4)^{T}$, both the private and non-private versions of pure Newton diverge. With the same initial point, backtracking line search performs 4 backtracking steps in the first iteration, then chooses a step size of 1 in subsequent iterations. (In this example, both the private and non-private algorithms performed the same number of backtracking steps, though in general those numbers often differ.) This switch to a step size of 1 allows for faster convergence compared to damped Newton. In this example, a fixed step size of approximately $\eta=0.1$ was sufficiently small to avoid divergence. However, using this small step size at every iteration caused damped Newton to converge more slowly and output a less accurate estimate of $\beta$ within the budgeted number of iterations.

\subsection{Practical considerations}
In privately evaluating the condition for backtracking in equation~\eqref{eq:private_backtracking_condition}, it is possible for the added noise term to reverse the inequality, resulting in an incorrect choice of step size. This introduces two potential pitfalls: first, the algorithm may prematurely stop backtracking, and second, the algorithm may continue backtracking longer than necessary. A premature exit from backtracking can lead the algorithm to diverge; taking extra, unnecessary backtracking steps can cause the algorithm to exhaust its backtracking budget without reaching the pure Newton phase, which can prevent the algorithm from reaching a region around the true solution within the budgeted number of iterations (supposing that after the backtracking budget has been consumed, any remaining iterations of Newton's method proceed with a fixed step size equal to the most recent step size). The parameter $\alpha$ influences the tradeoff between these two vulnerabilities. Small values of $\alpha$ make it easier to satisfy the backtracking exit condition, which tends to reduce the risk of unnecessary backtracking steps, but increase the risk of stopping backtracking too early. This tradeoff is illustrated in Figure \ref{fig:backtracking_alpha} below.

\begin{figure}[H]
    \centering
    \includegraphics[width=8cm]{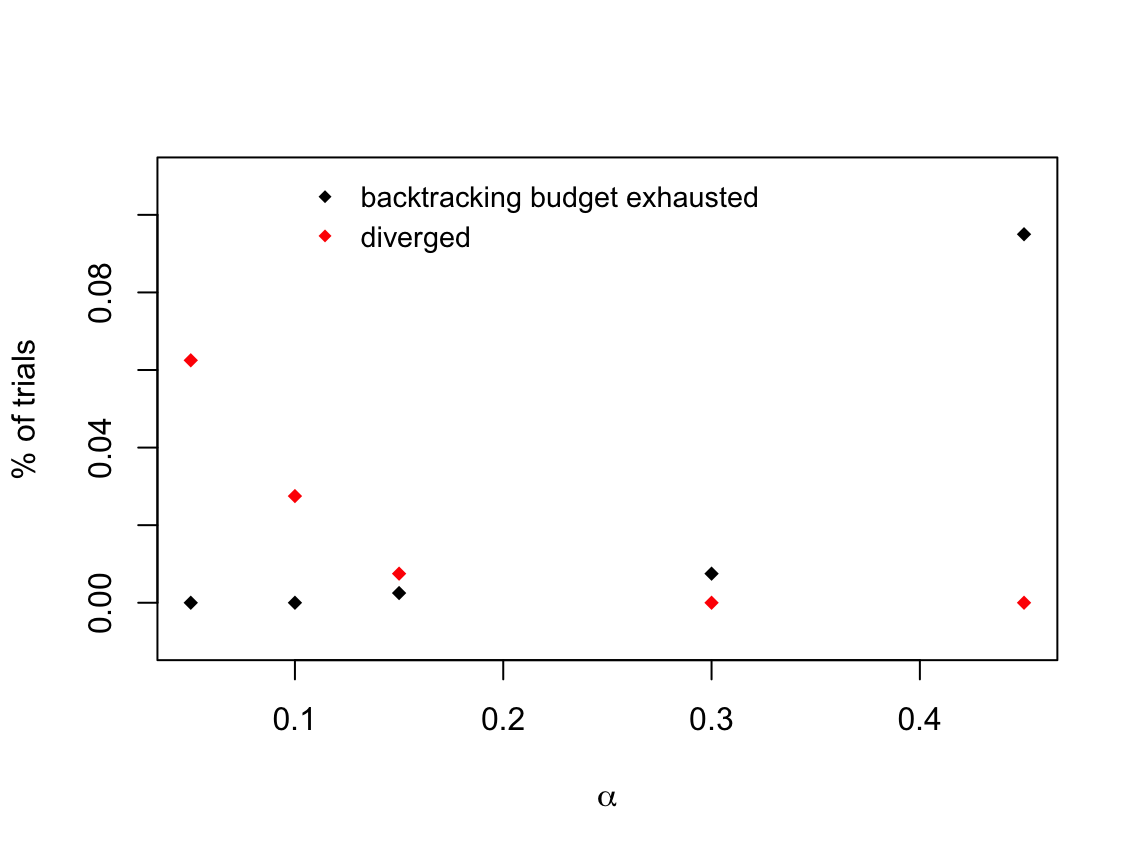}
    \caption{Behavior of 400 repetitions of noisy Newton's method with backtracking line search, applied to data generated from a logistic regression model with true $\beta= (1,1,1,1)^{T}$ and initialization at $\beta^{(0)}=(-4,6,-5,6)^{T}$}
    \label{fig:backtracking_alpha}
\end{figure}


\section{Benchmark experiments}
\label{AppBenchmark}

In this appendix, we provide some additional numerical comparisons between our proposed noisy gradient descent algorithm and existing methods in the literature.

\subsection{Clipping}
\label{AppClipping}

The plots of the clipping experiments mentioned in Section~\ref{SecLogisticReg} are provided in Figures~\ref{fig:clipping_linear} and~\ref{fig:clipping_logistic}.

\begin{figure}[h]
    \centering
    \includegraphics[width=12cm]{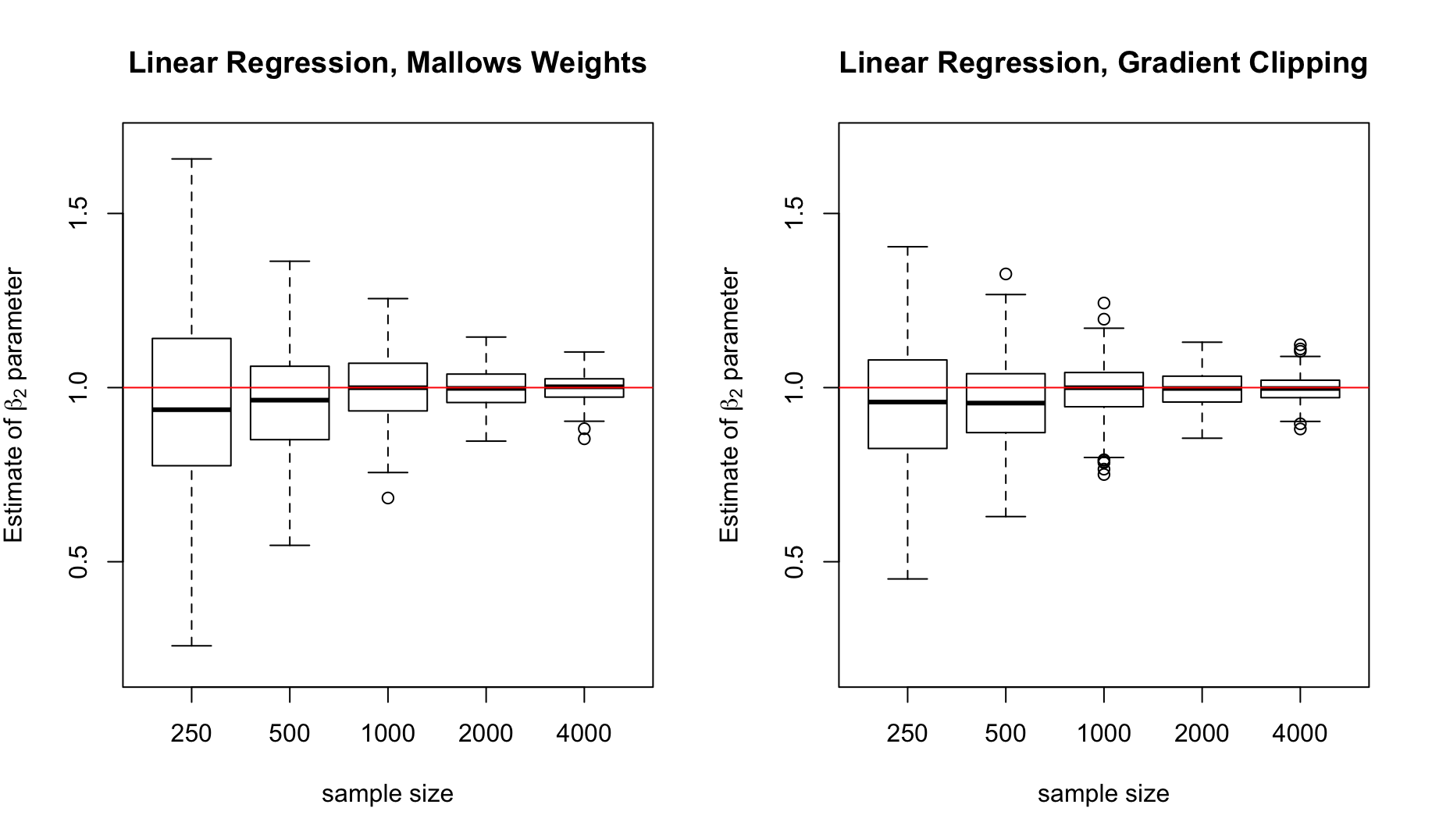} \\
        \begin{tabular}{cc}
    (a) \hspace{1in} & \hspace{1in} (b)
    \end{tabular}
    \caption{Gradient clipping and consistency. In the linear regression setting, both Mallows weights (plot (a)) and gradient clipping (plot (b)) approaches are consistent.}
    \label{fig:clipping_linear}
\end{figure}

\begin{figure}[h]
    \centering
    \includegraphics[width=12cm]{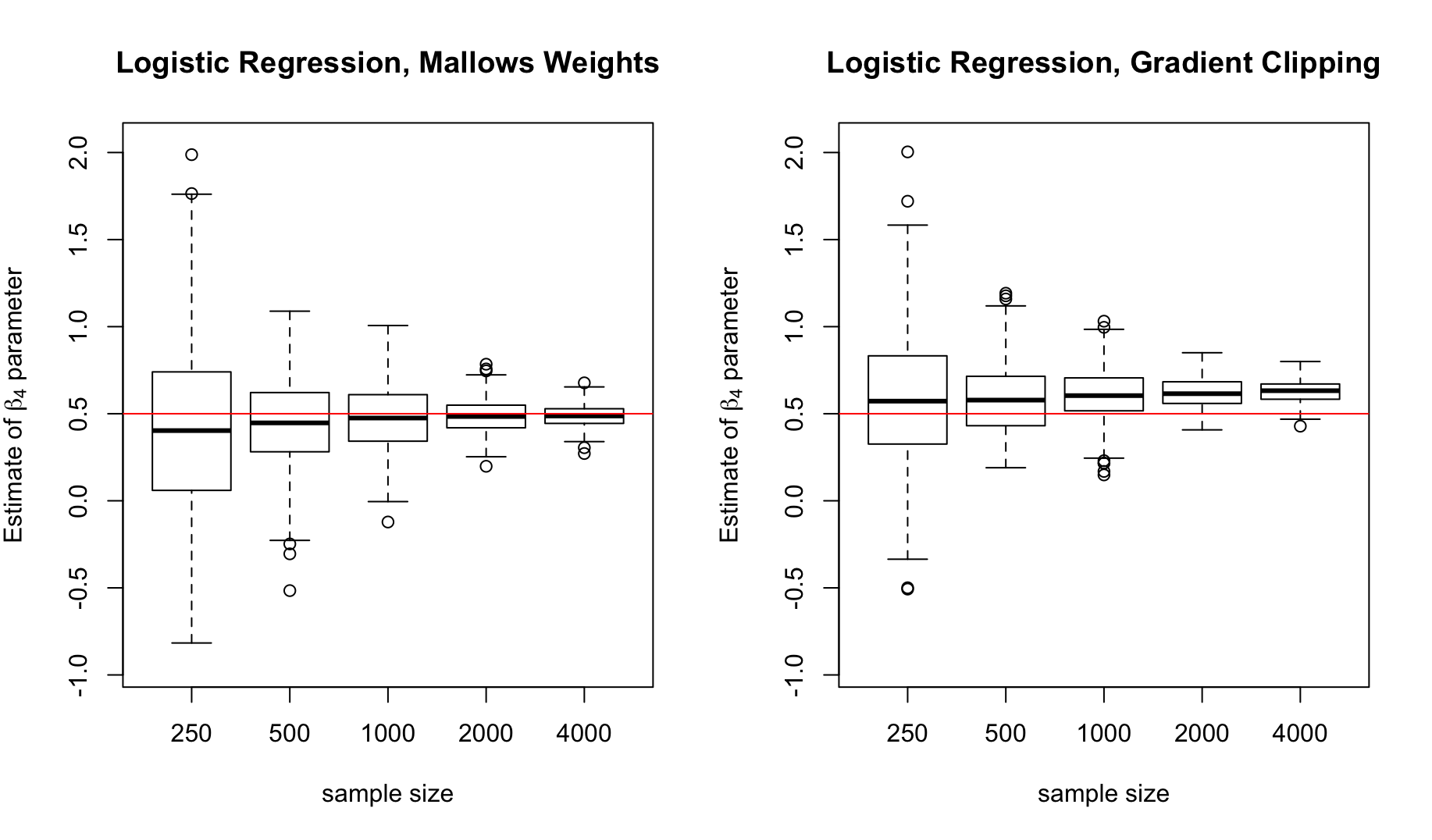} \\
        \begin{tabular}{cc}
    (a) \hspace{1in} & \hspace{1in} (b)
    \end{tabular}
    \caption{Gradient clipping and consistency. Clipping results in a positive bias for logistic regression (plot (b)), which does not arise using Mallows weights (plot (a)).}
    \label{fig:clipping_logistic}
\end{figure}

\subsection{Linear Regression}

We compare the noisy gradient descent approach~\eqref{eq:NGD} with a clipped gradient descent approach and with Sufficient Statistic Perturbation (SSP) \citep{vu2009differential, wang2018regression}, which requires a priori bounds on $\left| y_{i} \right|$ and $\left \|x_{i}  \right \|_{2}^{2}$. We consider two data scenarios: In the first, data are generated from a model obeying these constraints, and all methods are applied to the raw data. A key advantage of the noisy gradient descent method proposed in Section \ref{SecNGD} is that it is not subject to the data constraints of SSP. Nevertheless, in Figure~\ref{fig:method_comparison_linear1}, we see that noisy gradient descent remains a competitive estimation method in this constrained data regime. In addition to competitive estimation, noisy gradient descent can easily be extended to produce inference.


\begin{figure}[H]
    \centering
    \includegraphics[width=10cm]{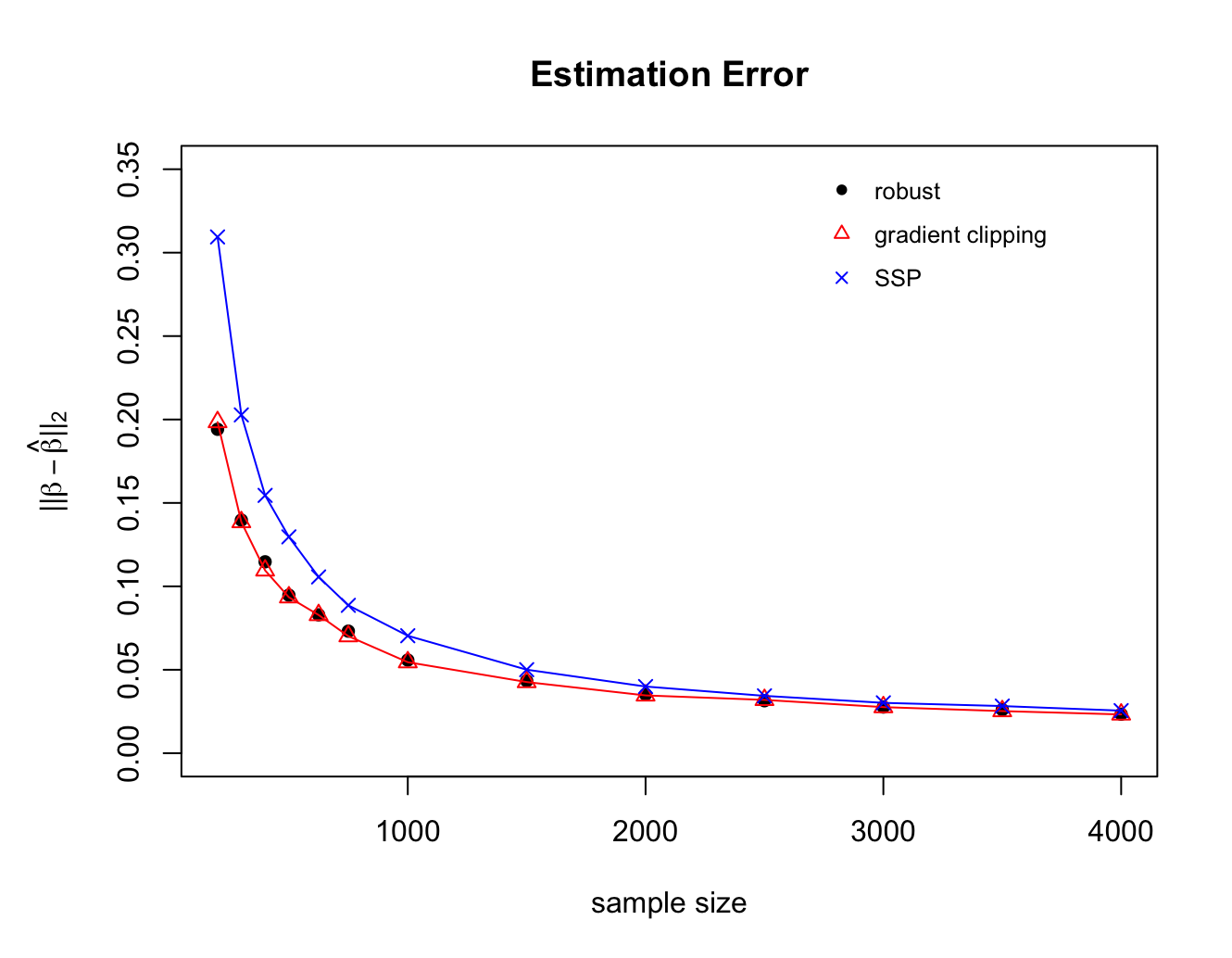} \\
    \caption{Linear regression with bounded covariates and responses}
    \label{fig:method_comparison_linear1}
\end{figure}

Data for Figure \ref{fig:method_comparison_linear1} were generated as follows:
\begin{enumerate}
    \item Generate $z_{i} \in \R^{3}$ by drawing $z_{i,j} \stackrel{i.i.d.}{\sim} Unif(-2,2)$ for $j=1,2,3$ and $i=1,...,n$.
    \item Impose an upper bound on $\|z_{i} \|_{2}$ by taking $\tilde{z_{i}}= z_{i} \cdot \min\left(1, \frac{\sqrt{3}}{\|z_{i} \|_{2}}\right)$.
    \item Set $x_{i}=(1, \tilde{z_{i}})^{T}$. By construction, $\| x_{i}\|_{2} \leq 2$ for all $i$.
    \item Draw i.i.d.\ noise terms $u_i$ from a truncated normal distribution, $u_{i} \sim TN(0, 0.75^{2}, q_{0.025}, q_{0.975})$, where $q_{0.025}$ and $q_{0.975}$ are the 2.5$^{\text{th}}$ and 97.5$^{\text{th}}$ quantiles of a $N(0, 0.75^{2})$ distribution, respectively.
    \item Generate responses $y_{i} = x_{i}^{T} \beta + u_{i}$ for $i=1, \dots, n$, with $\beta = (1, -1, 0.5, -0.5)^{T}$. By construction, $|y_{i}| \leq 1 + \sqrt{2} + \sqrt{\frac{1}{2}} + q_{0.975}$ for all $i$.
\end{enumerate}
Implementation details:
\begin{itemize}
\item \textbf{Noisy gradient descent (``robust"):} This method estimates $\beta$ using the noisy gradient descent iterations of equation~\eqref{eq:NGD}, with the loss function specified in equation~\eqref{EqnLinearLoss}, in which $c=1.345$ and $w(x_{i})=\textrm{min} \left(1,\frac{2}{\left \|x_{i}  \right \|_{2}} \right)$. For each sample size $n$, the number of iterations is $6 \log n$, rounded to the nearest integer. The step size at each iteration is taken to be $\frac{1}{L}$, where $L$ is the Lipschitz constant of the gradient. This implementation is calibrated for $\mu$-GDP with $\mu=1$.
\item \textbf{Gradient clipping:} This method performs gradient descent in which gradients are clipped to a user-specified maximum magnitude $h$. These simulations use a clipped version of the usual squared loss, such that the clipped gradients are of the form $$\nabla_{\beta} \mathcal{L}_{n,clip}(\beta) = \frac{-1}{n}\sum_{i=1}^{n} (y_{i}- x_{i}^{T} \beta)x_{i} \textrm{min} \left(1,\frac{h}{\left \|(y_{i}-x_{i}^{T} \beta) x_{i}  \right \|_{2}} \right).$$ 
Specifically, the clipping threshold $h$ is set equal to $2c$ ($c$ is the tuning parameter of the Huber loss in the example above), so that the maximum magnitude of the gradient is the same for the robust and clipped algorithms. The step size at each iteration is equal in value to the step size used for noisy gradient descent, and the number of iterations also remains the same. The privacy budget is $\mu=1$.
\item \textbf{Sufficient Statistic Perturbation (``SSP"):} This method uses the known bounds on $\|x_{i} \|_{2}$ and $|y_{i}|$ to privately estimate $X^{T}X$ and $X^{T}y$, and subsequently produce a private estimate of the ordinary least squares estimator $\hat{\beta} = (X^{T}X)^{-1} X^{T}y$. The private estimator is
$$ \hat{\beta}_{priv} = \left(X^{T}X + \frac{x_{max}^{2}}{\mu / \sqrt{2}} W\right)^{-1} \left(X^{T}y + \frac{2 x_{max} y_{max}}{\mu / \sqrt{2}} Z\right),$$
where $x_{max} = \sup_x \| x \|_{2} = 2$, $y_{max} = \sup_y |y| = 1 + \sqrt{2} + \sqrt{\frac{1}{2}} + q_{0.975}$, $Z \sim N(0, \mathbb{I}_{4})$, and $W$ is a symmetric matrix whose upper-triangular elements, including the diagonal, are i.i.d.\ $N(0,1)$. The privacy budget in these simulations is $\mu=1$.
\end{itemize}


Next, we consider a second data scenario, in which data are generated from a linear model with Gaussian errors, and do not obey the SSP constraints. Both gradient-based methods are applied to the raw data, while SSP is applied to data which have been non-privately preprocessed to satisfy the requirement of bounded $Y$ and $X$. In Figure~\ref{fig:method_comparison_linear2}, we see that the two gradient-based methods outperform SSP, particularly in smaller samples. Unlike SSP, noisy gradient descent can naturally address regression problems with unbounded covariates and responses.

\begin{figure}[H]
    \centering
    \includegraphics[width=10cm]{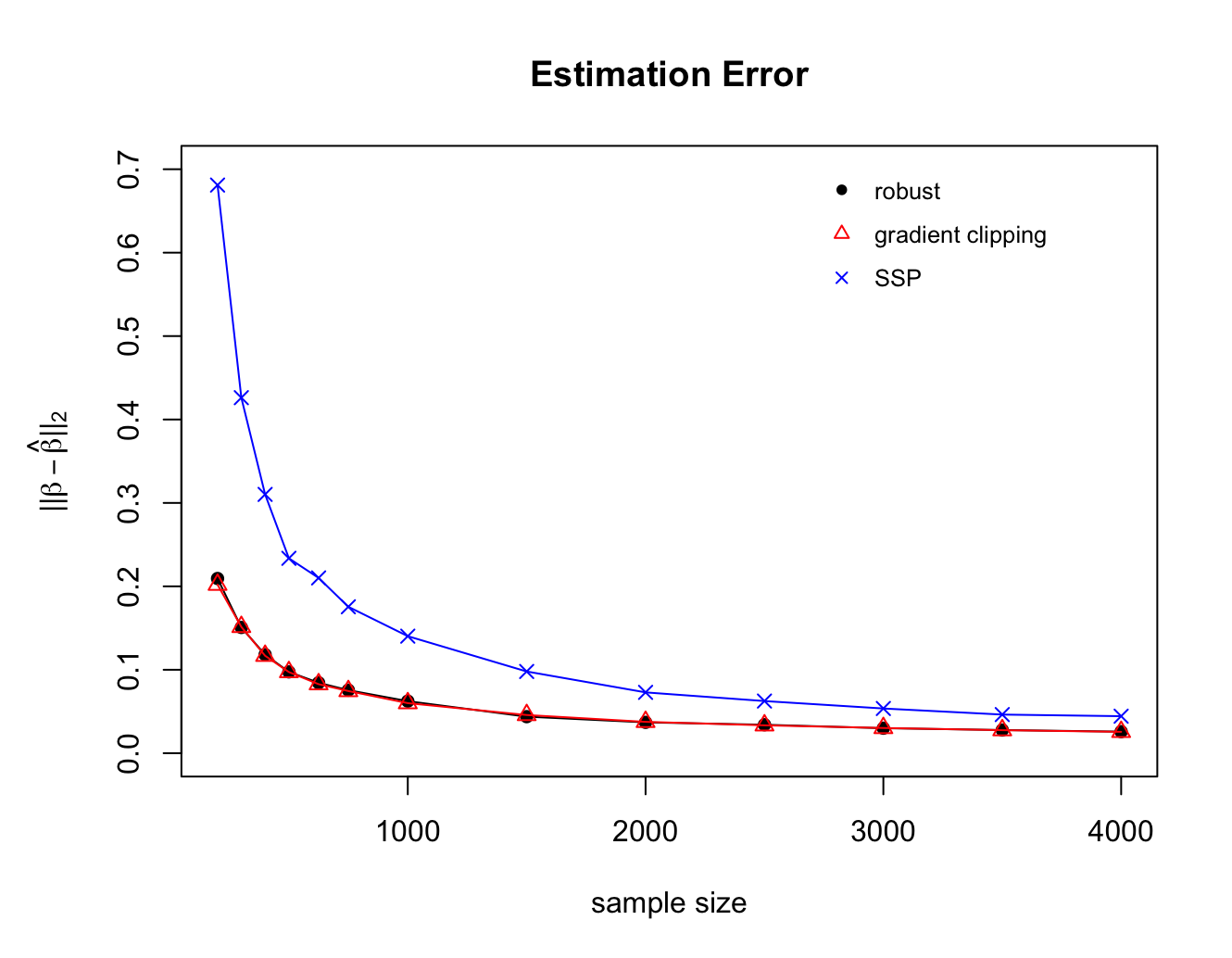} \\
    \caption{Linear regression with unbounded covariates and responses}
    \label{fig:method_comparison_linear2}
\end{figure}

Data for Figure \ref{fig:method_comparison_linear2} were generated as follows:
\begin{enumerate}
    \item Draw $x_{i} \overset{i.i.d.}{\sim} N(0, \mathbb{I}_{4})$.
    \item Draw i.i.d.\ noise terms $u_i \sim N(0, 0.75^{2})$.
    \item Generate responses $y_{i} = x_{i}^{T} \beta + u_{i}$ for $i=1, \dots, n$, with $\beta = (1, -1, 0.5, -0.5)^{T}$.
    \item Define $\alpha = \max ( \max_{i} \| x_{i} \|_{2}, \max_{i} |y_{i} | )$.
    \item Define $\tilde{X}= \frac{1}{\alpha}X$ and $\tilde{y} = \frac{1}{\alpha} y$.
\end{enumerate}
Noisy gradient descent and clipped gradient descent are applied to the raw data $(X, y)$, with implementation details unchanged from the previous scenario. SSP is applied to the (non-privately) preprocessed data $(\tilde{X}, \tilde{y})$. This is motivated by the observation that
\begin{equation*}
(\tilde{X}^{T} \tilde{X})^{-1} \tilde{X}^{T} \tilde{y} = \left(\frac{X^{T}}{\alpha} \cdot \frac{X}{\alpha}\right)^{-1} \frac{X^{T}}{\alpha} \cdot \frac{y}{\alpha} = (X^{T}X)^{-1} X^{T}y.
\end{equation*}


\subsection{Logistic Regression}

We now compare the noisy gradient descent approach~\eqref{eq:NGD} with clipped gradient descent and Objective Perturbation. Objective Perturbation for regression tasks, as proposed in \cite{chaudhurietal2011} and \cite{kiferetal2012}, requires a bound on $\left \|x_{i}  \right \|_{2}$ be known a priori.\footnote{More specifically, the requirements are a bounded gradient of the loss function, and a bounded maximum eigenvalue of the Hessian of the loss function. Due to the special structure of logistic regression, these requirements can be met via a constraint on the norm of $x$, although in general, this is not the case. For example, in linear regression, one would need to impose additional constraints on the responses $y$ and the radius of the parameter space for $\beta$.} Neither noisy gradient descent with the updates~\eqref{eq:NGD} nor clipped gradient descent have this constraint on the covariate space. We first compare these methods in the case of a known bound on $\left \|x_{i}  \right \|_{2}$. Within this case, we consider two scenarios:
\begin{itemize}
    \item Scenario 1: Many data points have $\left \|x_{i}  \right \|_{2}$ close to the upper bound. The known bound on the covariates means that clipping or downweighting is not necessary to achieve bounded global sensitivity. Clipping/downweighting thresholds can be set such that no clipping or downweighting is performed, effectively reducing noisy gradient descent and gradient clipping to the same algorithm. Data for this scenario were generated as follows:
        \begin{enumerate}
            \item Draw $z_{i} \overset{i.i.d.}{\sim} N(0, \mathbb{I}_{3})$ for $i=1,...,n$.
            \item Impose an upper bound on $\|z_{i} \|_{2}$ by taking $\tilde{z_{i}}= z_{i} \cdot \min\left(1, \frac{\sqrt{3}}{\|z_{i} \|_{2}}\right)$.
            \item Set $x_{i}=(1, \tilde{z_{i}})^{T}$. By construction, $\| x_{i}\|_{2} \leq 2$ for all $i$.
            \item Generate response $y_{i} \sim  \textrm{Bernoulli} \left( \frac{ 1}{1+\textrm{exp}(-x_{i}^{T} \beta )}\right)$, with $\beta = (1, -1, 0.5, -0.5)^{T}$.
        \end{enumerate}
    \item Scenario 2: Few data points have $\left \|x_{i}  \right \|_{2}$ close to the upper bound. While clipping or downweighting is not strictly necessary, it may be advantageous to perform some clipping/ downweighting, because it affects only a small fraction of the data points, in exchange for reducing the amount of added noise.
    Data for this scenario were generated in the same manner as above, modifying step 3 to instead set $\tilde{z_{i}}= z_{i} \textrm{min}(1, \frac{\sqrt{15}}{\|z_{i} \|_{2}})$.
\end{itemize}

Algorithm performance in these two scenarios is depicted in Figure~\ref{fig:method_comparison_logistic_bounded}.

\begin{figure}[h]
\begin{center}
\begin{tabular}{cc}
    \includegraphics[width=7cm]{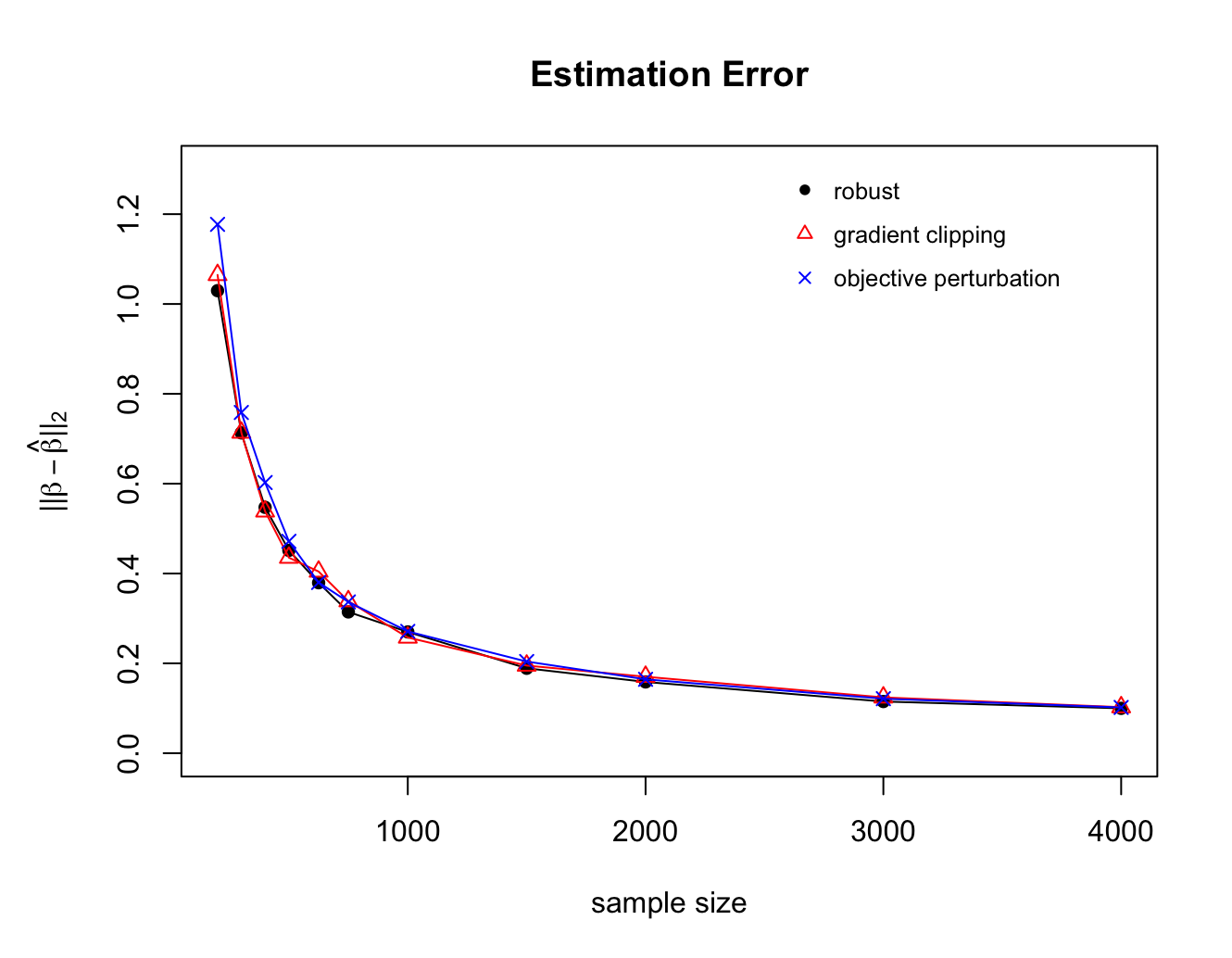} & \includegraphics[width=7cm]{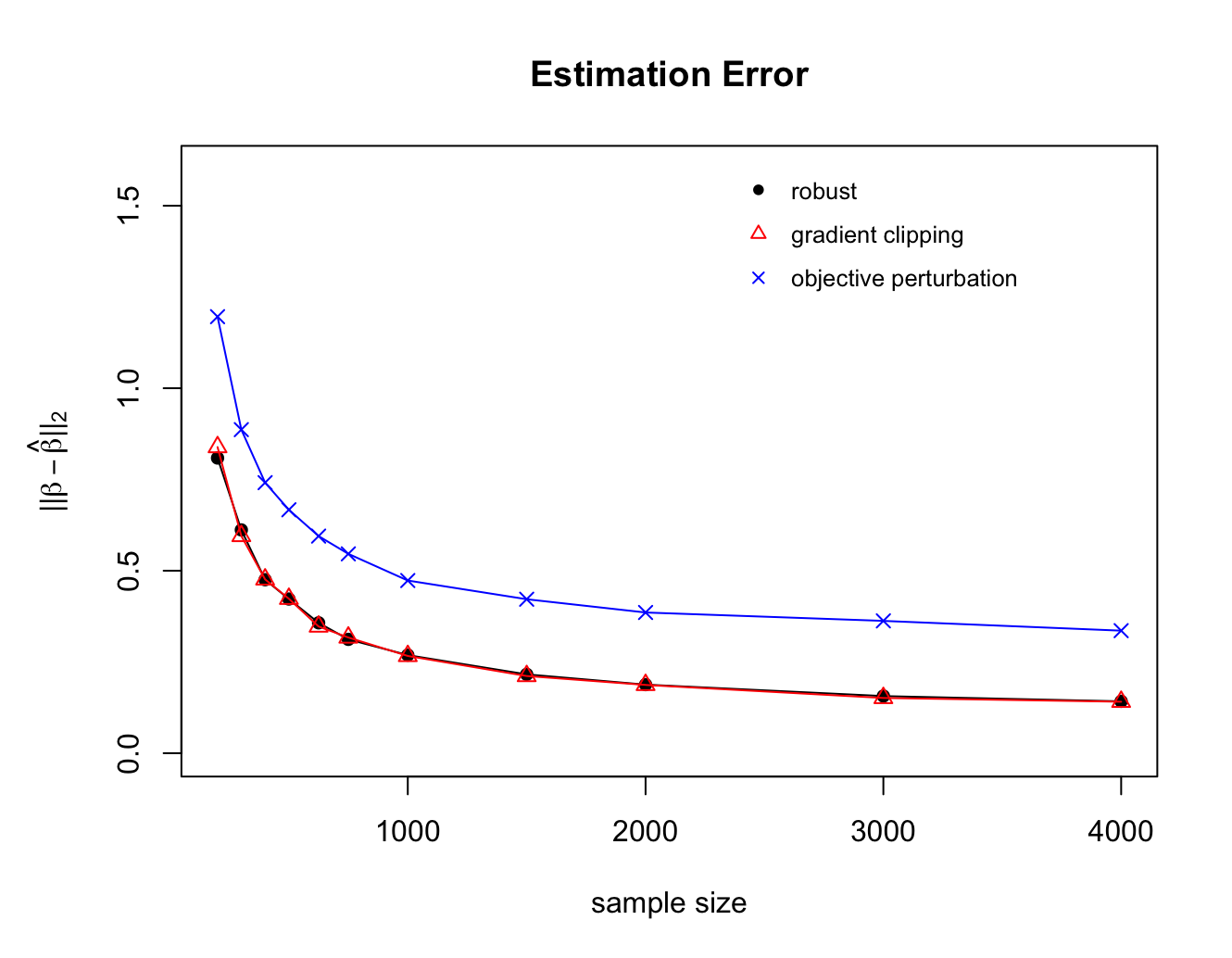} \\
    (a) & (b)
\end{tabular}
    \caption{(a) Scenario 1: Many data points near $\sup_{i} \left \|x_{i}  \right \|_{2}$. (b) Scenario 2: Few data points near $\sup_{i} \left \|x_{i}  \right \|_{2}$.}
\label{fig:method_comparison_logistic_bounded}
\end{center}
\end{figure}

In Scenario 1 (Figure~\ref{fig:method_comparison_logistic_bounded}(a)), all methods achieve similar performance. In Scenario 2 (Figure~\ref{fig:method_comparison_logistic_bounded}(b)), noisy gradient descent and gradient clipping achieve lower estimation error than Objective Perturbation. Only a small fraction (roughly 5\%) of data points had their contributions clipped/downweighted, in exchange for a smaller amount of added noise, allowing these methods to outperform Objective Perturbation. Note that because the amount of clipping/downweighting is small, noisy gradient descent and gradient clipping perform comparably. As discussed in Section~\ref{SecExamples}, gradient clipping can result in biased estimates, but this effect is less pronounced when the amount of clipping is small.

Implementation details:
\begin{itemize}
\item \textbf{Noisy gradient descent (``robust"):} This method estimates $\beta$ using the noisy gradient descent iterates of equation~\eqref{eq:NGD}, with the robust cross-entropy loss function specified in equation~\eqref{EqnLogisticLoss}. In Scenario 1, we take $w(x_{i})=\min \left(1,\frac{2}{\left \|x_{i}  \right \|_{2}} \right)$, which results in no downweighting, since $\|x_{i}  \|_{2} \leq 2$ for all $i$. In Scenario 2, we take $w(x_{i})=\min \left(1,\frac{3}{\left \|x_{i}  \right \|_{2}} \right)$, which does result in downweighting, as the upper bound in this scenario is $\|x_{i}  \|_{2} \leq 4$. For each sample size $n$, the number of iterations is $6 \log n$, rounded to the nearest integer. The step size at each iteration is taken to be $\frac{1}{L}$, where $L$ is the Lipschitz constant of the gradient. This implementation is calibrated for $\mu$-GDP with $\mu=0.5$.
\item \textbf{Gradient clipping:} This method performs gradient descent with a clipped version of the usual cross-entropy loss, such that the clipped gradients are of the form
$$ \nabla \mathcal{L}_{n,clip} (\beta) = \frac{1}{n} \sum_{i=1}^{n} \left((1+e^{-x_{i}^{T} \beta})^{-1} - y_{i}\right) x_{i} \textrm{min} \left(1,\frac{h}{\left \|((1+e^{-x_{i}^{T} \beta})^{-1} - y_{i}) x_{i}  \right \|_{2}} \right).$$
Specifically, the clipping threshold $h$ is set equal to $2$ in Scenario 1 and to $3$ in Scenario 2, so that the maximum magnitude of the gradient is the same for the robust and clipped algorithms. The step size at each iteration is taken to be $\frac{1}{L}$, where $L$ is the Lipschitz constant of the gradient. The number of iterations is the same as above, and the privacy budget is $\mu=0.5$.
\item \textbf{Objective Perturbation:} This method follows the $(\epsilon, \delta)$-differentially private objective perturbation procedure outlined in Algorithm 1 of \cite{kiferetal2012}. The loss function is the cross-entropy loss, i.e., equation~\eqref{EqnLogisticLoss} with no weight terms $\{w(x_{i})\}$. Because this algorithm satisfies $(\epsilon, \delta)$-DP and we wish to make a relevant comparison to $\mu$-GDP methods, we refer to Corollary 2.13 of \cite{dongetal2021}: a mechanism is $\mu$-GDP if and only if it is $(\epsilon, \delta(\epsilon))$-DP for all $\epsilon \geq 0$, where $\delta(\epsilon) = \Phi(- \epsilon / \mu + \mu/2) - e^\mu \Phi(- \epsilon / \mu - \mu/2)$. In these simulations, we choose a pair $(\epsilon, \delta(\epsilon))$ such that $\delta < 1/n$ for each sample size $n$, with $\mu=0.5$ in the expression for $\delta(\epsilon)$. The step size at each iteration is taken to be $\frac{1}{L}$, where $L$ is the Lipschitz constant of the gradient.
\end{itemize}

Next we consider logistic regression without a priori bounds on $\left \|x_{i}  \right \|_{2}$. Objective perturbation does not naturally address this data context, but can be extended to such data by including Mallows weights in the objective function (i.e. taking Equation \ref{EqnLogisticLoss} as the objective function). We have included this version of Objective Perturbation in our comparison. As seen in Figure~\ref{fig:method_comparison_logistic}, noisy gradient descent and Objective Perturbation applied to the robustified version of the cross-entropy loss achieve comparable performance in this unbounded data regime. Here, we see gradient clipping results in biased estimates, a behavior detailed in Section \ref{SecLogisticReg}.

\begin{figure}[H]
    \centering
    \includegraphics[width=10cm]{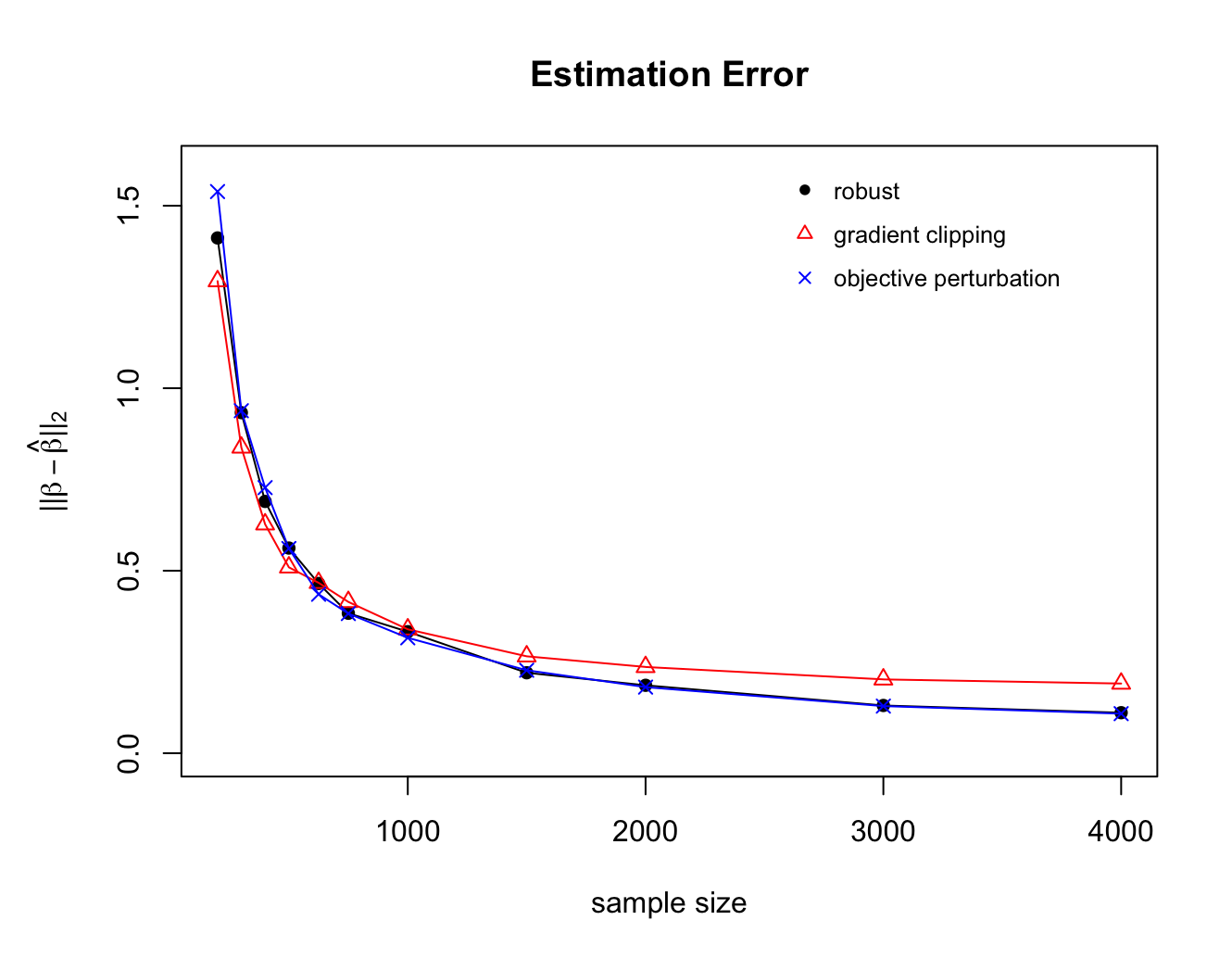} \\
    \caption{Logistic regression with unbounded covariates}
    \label{fig:method_comparison_logistic}
\end{figure}

Data for Figure \ref{fig:method_comparison_logistic} were generated in a similar manner to the bounded case above, without the rescaling of covariates in step 3. That is, we set $x_{i}=(1, z_{i})^T$ directly. The gradient clipping threshold is $h=\sqrt{2}$, and noisy gradient descent uses weights $w(x_{i})=\textrm{min} \left(1,\frac{\sqrt{2}}{\left \|x_{i}  \right \|_{2}} \right)$. Objective Perturbation requires a known upper bound on the largest eigenvalue of the Hessian of the loss function. To achieve such a bound for this problem, the weights in its objective function must downweight based on $ \| x_{i} \|_{2}^{2}$. Specifically, we take these weights to be $\tilde{w_{i}}= \min\left(1, \frac{2}{\|x_{i} \|_{2}^{2}}\right)$. All other implementation details are unchanged from the bounded-data case above.


\bibliographystyle{plainnat}
\bibliography{biblio}

\end{document}